\documentclass[11pt,reqno]{amsart}
\usepackage{amsmath,amsfonts,amssymb,amscd,amsthm,amsbsy,bbm, epsf,calc,  comment, caption, enumerate}
\usepackage{color}
\usepackage[usenames,dvipsnames]{xcolor}
\usepackage{datetime}

\usepackage{thmtools, thm-restate}
\usepackage{thm-patch, thm-kv, thm-autoref}
\usepackage[colorlinks=true, urlcolor=blue, linkcolor=blue, citecolor=black]{hyperref}

\usepackage[pdftex]{graphicx}
\usepackage{wrapfig}
\usepackage{geometry}
\numberwithin{equation}{section}
\newcounter{hours}\newcounter{minutes}

\theoremstyle{plain}
\declaretheorem[title=Theorem, parent=section]{theorem}
\declaretheorem[title=Lemma,sibling=theorem]{lemma}
\declaretheorem[title=Corollary,sibling=theorem]{corollary}
\declaretheorem[title=Proposition,sibling=theorem]{proposition}
\declaretheorem[title=Definition,sibling=theorem]{definition}
\declaretheorem[title=Remark,sibling=theorem]{rem}

\declaretheorem[title=Assumption,sibling=theorem]{assumption}

\newtheorem*{prop*}{Proposition}

\makeatletter
\newcommand\RedeclareMathOperator{%
  \@ifstar{\def\rmo@s{m}\rmo@redeclare}{\def\rmo@s{o}\rmo@redeclare}%
}
\newcommand\rmo@redeclare[2]{%
  \begingroup \escapechar\m@ne\xdef\@gtempa{{\string#1}}\endgroup
  \expandafter\@ifundefined\@gtempa
     {\@latex@error{\noexpand#1undefined}\@ehc}%
     \relax
  \expandafter\rmo@declmathop\rmo@s{#1}{#2}}
\newcommand\rmo@declmathop[3]{%
  \DeclareRobustCommand{#2}{\qopname\newmcodes@#1{#3}}%
}
\@onlypreamble\RedeclareMathOperator
\makeatother

\def\ep{\varepsilon}
\def\al{\alpha}
\def\del{\delta}
\def\om{\omega}
\def\Om{\Omega}
\def\gam{\gamma} 
\def\Gam{\Gamma}
\def\lam{\lambda}
\def\Lam{\Lambda}

\def\grad{\nabla}
\def\real{\mathbb R}
\def\K{\mathcal K}
\def\L{\mathcal L}
\def\M{\mathcal M}
\def\R{\mathcal R}
\def\T{\mathcal T}

\def\hull{\textnormal{hull}}

\def\Id{\textnormal{Id}}

\def\Union{\bigcup}
\def\intersect{\cap}
\def\Intersect{\bigcap}
\def\Indicator{{\mathbbm{1}}}
\def\graph{\textnormal{graph}}

\def\Tr{\textnormal{tr}}
\def\Id{\textnormal{Id}}

\DeclareMathOperator*{\osc}{osc}
\RedeclareMathOperator{\div}{\textnormal{div}}

\def\tr{\textnormal{tr}}

\def\polhk#1{\setbox0=\hbox{#1}{\ooalign{\hidewidth
	    \lower1.5ex\hbox{`}\hidewidth\crcr\unhbox0}}}

\newcommand{\abs}[1]{\left| #1 \right|}

\newcommand{\norm}[1]{\lVert#1\rVert}

\begin{document}

\title{Some free boundary problems recast as\\ nonlocal parabolic equations}

\author{H\'ector A. Chang-Lara}
\author{Nestor Guillen}
\author{Russell W. Schwab}

\address{Centro de Investigaci\'on en Matem\'aticas, A.C., Jalisco S/N, Col. Valenciana CP: 36023 Guanajuato, Gto, M\'exico, Apartado Postal 402, CP 36000.
}
\email{hector.chang@cimat.mx}

\address{Department of Mathematics\\
University of Massachusetts, Amherst\\
Amherst, MA  90095}
\email{nguillen@math.umass.edu}

\address{Department of Mathematics\\
Michigan State University\\
619 Red Cedar Road \\
East Lansing, MI 48824}
\email{rschwab@math.msu.edu}

\begin{abstract}
  In this work we demonstrate that a class of some one and two phase free boundary problems can be recast as nonlocal parabolic equations on a submanifold.  The canonical examples would be one-phase Hele Shaw flow, as well as its two-phase analog.  We also treat nonlinear versions of both one and two phase problems.   In the special class of free boundaries that are graphs over $\real^d$, we give a precise characterization that shows their motion is equivalent to that of a solution of a nonlocal (fractional), nonlinear parabolic equation for functions on $\real^d$.  Our main observation is that the free boundary condition defines a nonlocal operator having what we call the Global Comparison Property. A consequence of the connection with nonlocal parabolic equations is that for free boundary problems arising from translation invariant elliptic operators in the positive and negative phases, one obtains, in a uniform treatment for all of the problems (one and two phase), a propagation of modulus of continuity for viscosity solutions of the free boundary flow.
 
\end{abstract}

\date{\today,\ arXiv Ver. 1}
\thanks{Two of the authors acknowledge partial support from the NSF leading to the completion of this work: N. Guillen DMS-1700307; R. Schwab DMS-1665285.}

\maketitle

\markboth{H. Chang-Lara, N. Guillen, and R. Schwab}{Free boundary problems as nonlocal parabolic equations}


\section{Introduction}\label{sec:introduction}
\setcounter{equation}{0}

In this work we demonstrate that some free boundary problems can be formulated as nonlocal parabolic equations.  For both simplicity and technical restrictions, we focus on those problems that describe the motion of the graph of a function over $\real^d$.  Further assumptions appear in Section \ref{sec:Assumption}.  We elaborate on these assumptions and the possibility to reduce and/or modify them in Section \ref{sec:CommentsOnAssumptions}.  Specifically, we investigate the class of free boundary problems that have the following form:
\begin{align*}
  U: \real^{d+1} \times [0,T]\to\mathbb{R},
\end{align*}
and $U$ solves (in a way to be made precise in Section \ref{sec:WeakSolutions}),
\begin{align}\label{eqIn:MainFBEvolution}
	\begin{cases}
  F_1(D^2U,\grad U) = 0 &\textnormal{ in } \{U>0\} \\
  F_2(D^2U, \grad U)  = 0 &\textnormal{ in } \left(\{U>0\}^C\right)^\circ \\
  V  = G(\partial^+_nU,\partial^-_n U) &\textnormal{ on } \Gamma = \partial\{U>0\}, 
  \end{cases} 
\end{align}
where $n$ is the inward normal direction to the positivity region, $\{U>0\}$.  The derivatives, $\partial^\pm_nU$, are computed from respectively the positive and the negative phases, normalized so that both $\partial^+_nU$ and $\partial^-_n U$ are positive quantities.
This means that the free boundary, denoted as $\Gamma(t)$, evolves by the normal velocity field given by $V$, which of course depends on $U$ (so, when $U$ is nice enough, these dynamics track the evolution of the graph of $\partial\{U(\cdot,t)>0\}$).  Or, if one prefers to think in a level set formulation of the flow, this velocity condition would mean that
\begin{align*}
	\text{on}\ \partial\{U(\cdot,t)>0\},\ \text{the equation is}\   \partial_t U= G(\partial^+_nU,\partial^-_n U)\abs{\grad U}.
\end{align*}
As one can see, $\grad U^\pm$ play the roles of the gradient of the pressures in their respective phases; these quantities drive the flow, and the velocity is determined by the balance law, $G$.  Some common examples are expanded upon in Section \ref{sec:Examples}.

Here, the operators, $F_1$ and $F_2$, are translation invariant and either both uniformly elliptic or one is uniformly elliptic and the other is zero only when $U$ is identically zero (for one-phase problems). More detailed assumptions appear in Section \ref{sec:Assumption}.  The canonical example would be $F_1(D^2U,\grad U)=\Delta U$ and the operator $F_2(D^2U,\grad U)$ is only zero when $U^-$ is identically zero, which would give what many authors call the one-phase Hele-Shaw flow.  The two-phase analog would be when, e.g. $F_1(D^2U,\grad U)=F_2(D^2U,\grad U)=\Delta U$.
The nonlocal parabolic equations that arise take place on $\mathbb{R}^{d}$, and they have the form,
\begin{align}\label{eqIn:FractionalParabolic}
	\begin{cases}
  & f:\real^d \times [0,T]\to\mathbb{R}\\
  & \partial_t f = G(I^+(f))\cdot\sqrt{1+\abs{\grad f}} \textnormal{ in } \real^d\times [0,T]\ \ \text{(for the one-phase problem)},\\
  \text{or}\ &\partial_t f = G(I^+(f), I^-(f))\cdot\sqrt{1+\abs{\grad f}} \textnormal{ in } \real^d\times [0,T]\ \ \text{(for the two-phase problem)},
  \end{cases}
\end{align}
where $G$ is the prescribed velocity function as above, and $I^\pm$ are (nonlinear) fractional order nonlocal operators acting on functions $f\in C^{1,\gam}(\real^d)$, and they enjoy the global comparison principle.  Roughly speaking, $I^\pm$ linearize to integro-differential operators in the class that given by
\begin{align}\label{eqIn:LinearizedIntDiff}
  L(f,x) & = b(x)\cdot\grad f(x) +  \int_{\real^d} f(x+h)-f(x)-(\nabla f(x)\cdot h)\;\Indicator_{B_{1}}(h) \;\mu(x,dh),
\end{align}
where for each $x$, $\mu(x,dh)$ is a L\'evy measure on $\real^d$.

We informally state our main results here, and then they will be expanded upon and split into separate, more precise results in later sections.  They state, under a graph assumption for the free boundary, an equivalence between viscosity solutions of the free boundary and fractional parabolic equations, and also a preservation of modulus of continuity of the initial free boundary.  This first theorem will be addressed more precisely in Sections \ref{sec:Comparison}, \ref{sec:WeakSolutions}, and \ref{sec:PropagationOfModulus}.  We also tie together all of the arguments into one place in Section \ref{sec:Proofs}.

\begin{theorem}\label{thm:MainMetaVersion}
	If $F_1$ and $F_2$ are uniformly elliptic and rotationally invariant in the Hessian variable, $G$ is Lipschitz and monotone (see the assumptions of Section \ref{sec:Assumption}), and $f:\real^d\times[0,T]\to\real$ and $U_f:\real^{d+1}\times[0,T]\to\real$ are such that
	\begin{align*}
		\forall\ t\in[0,T],\ \ \Gamma(t)=\partial\{U_f(\cdot,t)>0\}=\graph(f(\cdot,t)),
	\end{align*} 
	then,
	\begin{enumerate}[(i)]
		\item $U_f$ is a viscosity solution of the free boundary evolution, (\ref{eqIn:MainFBEvolution}), with appropriate boundary conditions, if and only if $f$ is a viscosity solution of the fractional parabolic equation, (\ref{eqIn:FractionalParabolic});
		\item if additionally, $\graph(f(\cdot,0))=\partial\{U_f(\cdot,0)\}$ enjoys a modulus of continuity, $\abs{f(x,0)-f(y,0)}\leq \om(\abs{x-y})$, then for all $t\in[0,T]$ this modulus is preserved for $f(\cdot,t)$ (and hence $\partial\{U_f(\cdot,t)>0\}$);
		\item if $\graph(f(\cdot,0))=\partial\{U_f(\cdot,0)\}$ enjoys a modulus of continuity, $\abs{f(x,0)-f(y,0)}\leq \om(\abs{x-y})$, then there exists a unique viscosity solution to both the respective parabolic equation and the free boundary evolution, assuming the graph assumption for $\partial\{U>0\}$ listed above.
	\end{enumerate}
\end{theorem}

\begin{rem}
	The precise definitions of viscosity solutions are given in Section \ref{sec:Comparison} for the parabolic nonlocal equation and Section \ref{sec:WeakSolutions} for the free boundary problem.
\end{rem}

\begin{rem}
	As mentioned above, the one-phase problem can be interpreted as taking $F_2$ to only be zero when $U$ is identically zero.  Another way to capture the one-phase problem from Theorem \ref{thm:MainMetaVersion} is to take $G(\partial_n^+U,\partial_n^-U)=g(\partial_n^+ U)$, which ignores whatever is the prescribed equation in the set $\{U>0\}^C$.
\end{rem}

The main feature behind the results in Theorem \ref{thm:MainMetaVersion} is the analysis of what could be called a nonlinear version of the Dirichlet-to-Neumann operator, but modified to suit the free boundary problems like (\ref{eqIn:MainFBEvolution}).  This is the operator we called $I$, above, and we define it here.  Assume that $F$ satisfies assumptions listed in Section \ref{sec:Assumption}, and that $f\in C^{1,\gam}(\real^d)$ with the property that $f\geq \del>0$.  Then, for the domain given by $f$, which we define as 
\begin{align}\label{eqIn:DefOfDomainD}
D_f=\{(x,x_{d+1})\in\real^d\times\real_+\ :\ 0<x_{d+1}<f(x)\}, 
\end{align}
there is a unique solution, $U_f$, to the equation
\begin{align}\label{eqIn:BulkOnePhaseExtension}
	\begin{cases}
		F(D^2U_f,\grad U_f)=0\ &\text{in}\ D_f\\
		U_f=1\ &\text{on}\ \real^d\times\{0\}\\
		U_f=0\ &\text{on}\ \graph(f).
	\end{cases}
\end{align}
This allows us to define an operator, $I$, via the (inward) normal derivative of $U_f$, which is given by
\begin{align}\label{eqIn:DefOfI}
	I(f,x) := \partial_n U_f(x,f(x)).
\end{align}
One of the main results in our work that leads to Theorem \ref{thm:MainMetaVersion} is the following property of $I$.  Again, as with Theorem \ref{thm:MainMetaVersion}, we state it informally here, and revisit it with more precise assumptions later, in Sections \ref{sec:NewOperator}, \ref{sec:MinMax}, and \ref{sec:Proofs}.

\begin{theorem}\label{thm:MetaILipAndMinMax}
	If $F$ is uniformly elliptic and rotationally invariant in the Hessian variable (see the assumptions of Section \ref{sec:Assumption}), $I$ is defined via (\ref{eqIn:BulkOnePhaseExtension}) and (\ref{eqIn:DefOfI}), and $\gam\in(0,1)$ is fixed, then there exists $\gam'$ with $0<\gam'<\gam$ so that  
	\[
	\displaystyle I: \left(\Union_{\del>0} C^{1,\gam}(\real^d)\intersect\{f\ :\ f\geq\del>0\}\right) \to C^{\gam'}(\real^d),\ \ \text{and}\ I\ \text{is locally Lipschitz}.
	\]
 Additionally, $I$ enjoys the following representation for $a^{ij}$, $c^{ij}$, $b^{ij}$, and $\mu^{ij}(dh)$ that are independent of $x$, and that depend on the bounded set $\{f\ :\ f\geq \del,\ \norm{f}_{C^{1,\gam}}\leq m\}$:
	\begin{align*}
	I(f,x) = \min_i \max_j \left\{a^{ij} + c^{ij}f(x) + b^{ij}\cdot\grad f(x) 
	+ \int_{\real^d} \left( f(x+h)-f(x)-\Indicator_{B_1}(h)\grad f(x)\cdot h \right)\mu^{ij}(dh)\right\}.	
	\end{align*} 
	Furthermore, there exists a $C$ such that $\abs{a^{ij}}\leq C$, $-C\leq c^{ij}\leq 0$, $\abs{b^{ij}}\leq C$, and
	\begin{align*}
		\int_{\real^d}\min\left(  \abs{h}^{1+\gam}, 1  \right) \mu^{ij}(dh)\leq C
		\ \ \text{and}\ \ 
				\int_{\real^d\setminus B_{R}}\mu^{ij}(dh)\leq C\om(R),
	\end{align*}
	for some modulus, $\om$, with $\om(R)\to0$ as $R\to\infty$.  The contant, $C$, depends on the Lipschitz bound for $I$ over the corresponding subset of $ \left(\Union_{\del>0} C^{1,\gam}(\real^d)\intersect\{f\ :\ f\geq\del>0\}\right)$.
\end{theorem}

\begin{rem}
	We want to emphasize that in most of our results, above and below, there is a technical assumption that $F_i$ must be rotationally invariant in the Hessian variable.    The curious reader can see its use in the proof of Theorem \ref{thmNO:LipschitzProperty}, but it is not clear if this was an artifact of our chosen method or if it is a true obstacle to such representations of the free boundary problem.
\end{rem}

\begin{rem}
	We note that to properly study the parabolic equation (\ref{eqIn:FractionalParabolic}), the most important part is to understand the Lipschitz, comparison, and integro-differential  properties of the map, $I$, defined in (\ref{eqIn:DefOfI}).  As the reader will see at the end of Section \ref{sec:MinMax}, all the desired properties then also carry over to the operator
	\begin{align*}
		G(I(f))\cdot\sqrt{1+\abs{\grad f}},
	\end{align*}
which is the one that appears in (\ref{eqIn:FractionalParabolic}) and corresponds to the free boundary evolution.
\end{rem}

\begin{rem}
	In the interest of transparency, we would like to mention what is new and what is closer to what is already known about the results that appear above.  
	\begin{itemize}
		\item In principle, the free boundary existence and uniqueness results mentioned in Theorem \ref{thm:MainMetaVersion} are not surprising, and most experts would probably say they closely mirror the results on existence and uniqueness found in \cite{Kim-2003UniquenessAndExistenceHeleShawStefanARMA}.  However, to the best of our knowledge, even with such special assumptions as the graph condition, we are not aware that the fully nonlinear nor the two-phase cases have been written or appeared anywhere.  Thus, this step is technically new, but maybe not substantially new.
		\item The preservation of a modulus of continuity for the free boundary again, is to the best of our knowledge, new, but given the fact that it is really a consequence of the definition of viscosity solutions, combined with the underlying translation invariance of the problem, one could also suggest that this may follow in the Hele-Shaw case from some arguments similar to \cite{Kim-2003UniquenessAndExistenceHeleShawStefanARMA}.  However, this result seems to only appear when the modulus is a Lipschitz norm and the flow is Hele-Shaw, which was proved in \cite{ChoiJerisonKim-2007RegHSLipInitialAJM} (under a star-shaped assumption instead of a graph assumption).  Thus for the two-phase and the fully nonlinear one-phase cases, we believe this part of the results are technically new.
		\item The Lipschitz property of the mapping, $I$, to the best of our knowledge is completely new.  The integro-differential formula in the min-max of Theorem \ref{thm:MetaILipAndMinMax} is new, even though previous works had noted connections with operators similar to the 1/2-Laplacian (more discussion of this appears in Section \ref{sec:Background}).
		\item  The correspondence of the free boundary evolution with the viscosity solution of the parabolic equation (\ref{eqIn:FractionalParabolic}) is relatively new, and the closest result is that of two of the authors, where in \cite{ChangLaraGuillen-2016FreeBondaryHeleShawNonlocalEqsArXiv} they showed that certain blow-up limits of the free boundary will be viscosity solutions of a parabolic equation.  The novelty in our current work is to avoid the blow-up argument, and work directly with the original free boundary as a fractional parabolic equation.
	\end{itemize}
\end{rem}

\subsection{A note on the level-set formulation of (\ref{eqIn:MainFBEvolution})}

One way to see the connection between the original free boundary evolution given by (\ref{eqIn:MainFBEvolution}) and the nonlinear fractional parabolic equation in (\ref{eqIn:FractionalParabolic}) is by recognizing the flow as a level-set equation.  Indeed, the graph assumption means that we naturally have two different choices of defining functions for the boundary, $\partial\{U(\cdot,t)>0\}=\partial\Om(t)$.  Of course, this is the zero level set of $U$, which is the obvious defining function.  However, if we assume that we have parametrized $\Om(t)$ by \emph{some} defining function, $\Phi(X,t)$, we have
\begin{align*}
	\partial\Om(t)=\{\Phi(\cdot,t)=0\}.
\end{align*}
Regardless of from where the prescribed evolution of $\Om(t)$ originates, if one insists that the normal velocity is $V$ and the (outer!) normal vector is $n$, then it is not hard to check that the level set equation becomes
\begin{align*}
	0=\grad\Phi\cdot (nV) + \partial_t\Phi.
\end{align*}
Of course, as a defining function of a level set, we have, depending upon whether or not $\Phi>0$ in $\Om(t)$ or $\Phi<0$ in $\Om(t)$,
\begin{align*}
	n=\pm\frac{\grad\Phi}{\abs{\grad\Phi}},
\end{align*}
and hence the flow reduces to 
\begin{align*}
	\partial_t\Phi = \pm V\abs{\grad\Phi}.
\end{align*}
In our case, to be consistent with the natural parametrization, we choose $\Phi>0$ in $\Om(t)$, and hence a second natural choice then becomes (where $X=(x,x_{d+1})$)
\begin{align*}
	\Phi(X,t) = f(x,t)-x_{d+1}.
\end{align*}
Furthermore, we have chosen $V=G(\partial_n^+U_f,\partial_n^-U_f)$, and so given the definition of the operator $I$ (and correspondingly $I^+$ and $I^-$), we finally obtain
\begin{align*}
	\partial_t f = G(I^+(f),I^-(f))\cdot\sqrt{1+\abs{\grad f}^2}.
\end{align*}


\section{Some definitions and assumptions}\label{sec:DefsAndAssumptions}

\subsection{Definitions}\label{sec:Defs}

\begin{definition}\label{def:TranslationOperator}
	The translation operator, $\tau_x$, is defined, for a fixed $x$, acting on functions on $\real^{d}$, as 
	\begin{align*}
		\tau_x u:= u(\cdot+x).
	\end{align*}
\end{definition}

\begin{definition}[GCP]\label{def:GCP}
(Part 1) The global comparison property (GCP) for $I: C^{1,\gam}(\real^d)\to C^0(\real^d)$ requires that for all $u,v\in C^{1,\gam}(\real^d)$ such that $u(x)\leq v(x)$ for all $x\in \real^d$ and such that for some $x_0$, $u(x_0)=v(x_0)$, then the operator $I$ satisfies $I(u,x_0)\leq I(v,x_0)$.  That is to say that $I$ preserves ordering of functions on $\real^d$ at any points where their graphs touch.

(Part 2) Analogously, we say that $I$ has the GCP at $x_0$ if the above property is only required to hold for one fixed $x_0$, instead of all possible $x_0$.
	
\end{definition}

\begin{definition}[Extremal Operators]\label{def:MaximalPucci}
	The second order $(\lam,\Lam)$-Pucci extremal operators are defined as $\M^-$ and $\M^+$, for a function, $U$ that is second differentiable at $X$, via
	\begin{align*}
		\M^-(D^2U,X) = \min_{\lam\Id\leq B\leq \Lam\Id}\left( \Tr(BD^2U(X)) \right)\ \ \text{and}\ \ 
		\M^+(D^2U,X) = \max_{\lam\Id\leq B\leq \Lam\Id}\left( \Tr(BD^2U(X)) \right).
	\end{align*}
	When $\{e_i\}_{i=1,\dots,d+1}$ are the eigenvalues of $D^2U(X)$, an equivalent representation is
	\begin{align*}
		\M^-(D^2U,X) = \Lam\sum_{e_i\leq 0}e_i + \lam\sum_{e_i>0} e_i
		\ \ \text{and}\ \ 
		\M^+(D^2U,X) = \lam\sum_{e_i\leq 0}e_i + \Lam\sum_{e_i>0} e_i.
	\end{align*}
\end{definition}

\begin{definition}[Uniformly Elliptic]\label{def:UniformlyElliptic}
	When $F$ is linear, i.e. $F(D^2U,\grad U)=\tr(A(X)D^2U(X))+B(X)\cdot\grad U(X)$ we say that $F$ is $(\lam,\Lam)$-uniformly elliptic  if $\norm{B}_{L^\infty}\leq \Lam$ and
	\begin{align*}
		\lam\Id\leq A(X)\leq \Lam\Id\ \ \text{for all}\ \ X,
	\end{align*}
	and when $F$ is nonlinear,  if for all $U,V\in C^2$,
    \begin{align*}
      \mathcal{M}^-(D^2U-D^2V)-\Lambda |\grad U- \grad V| &\leq F(D^2U,\grad U)-F(D^2V,\grad V)\\ 
	  &\leq \mathcal{M}^+(D^2U-D^2V)+\Lambda |\grad U- \grad V|.
    \end{align*}	
\end{definition}

Next, we recall some definitions from Clarke's book on nonsmooth analysis, \cite{Clarke-1990OptimizationNonSmoothAnalysisBook}.  They are appropriately modified for our context.

\begin{definition}[Upper Gradient]\label{def:UpperGradient}
	Assume that $\K$ is an open convex subset of $C^{1,\gam}(\real^d)$ and that $\phi: \K \subset C^{1,\gam}(\real^d)\to\real$ is Lipschitz.
  The upper gradient of $\phi$ at $f\in \mathcal{K}$ in the direction of $g\in C^{1,\gam}$, is defined as
  \begin{align*}
    \phi^0(f;g) := \limsup\limits_{t\searrow 0} \frac{\phi(f+tg)-\phi(f)}{t}.
  \end{align*}	
  This can be seen as a function $\phi^0:\mathcal{K}\times C^{1,\gam}\to \real$.
    
\end{definition}

\begin{definition}[Subdifferential]\label{def:ClarkeDifferential}
  Let $\phi$ be as in Definition \ref{def:UpperGradient}. The generalized gradient, or as we will say, Clarke differential of $\phi$ at $f\in \mathcal{K}$ is the subset of $\displaystyle \left(C^{1,\gam}(\real^d)\right)^*$ given by
  \begin{align*}
    \partial \phi(f) := \{ \ell \in \left(C^{1,\gam}(\real^d)\right)^*\  \mid\  \forall\ \psi\in C^{1,\gam}(\real^d),\  \phi^0(f;\psi)\geq \langle \ell,\psi\rangle \}. 
  \end{align*}	
  We will denote simply by $[\partial \phi]_{\K}$ the convex hull of the union of $\partial \phi(f)$,
  \begin{align*}
    [\partial \phi]_{\K} := \hull \left (  \bigcup \limits_{f\in \mathcal{K}} \partial \phi(f)\right).	  
  \end{align*}	        
\end{definition}

\begin{rem}
	We want to stress that $[\partial\phi]_\K$ depends on the original set, $\K$.  This plays an essential role in Sections \ref{sec:MinMax} and \ref{sec:Comparison}.
\end{rem}

\subsection{Assumptions}\label{sec:Assumption}

We keep the following standing assumptions on $F_i$ for $i=1,2$:

\begin{enumerate}[(a)]
	\item there exists $\lam\leq\Lam$ such that $F_i$ is uniformly elliptic in the sense of Definition \ref{def:UniformlyElliptic};
	\item $F(0,0)=0$;
	\item for all $p$, $F(A,p)$ is rotationally invariant in the $A$ variable (the Hessian variable);
	\item  $\displaystyle G:(0,\infty)^2 \to \mathbb{R},$ and for a.e. $(a,b)$,
	$\displaystyle \lambda_0 \leq \frac{\partial }{\partial a}G(a,b) \leq \Lambda_0,\;\;\lambda_0 \leq -\frac{\partial }{\partial b}G(a,b) \leq \Lambda_0$. 
\end{enumerate}

\subsection{Notation}\label{sec:Notation}

Here we will collect some notation that is used in this work.

\begin{itemize}
	\item $F$ is an elliptic second order operator, described in Definition \ref{def:UniformlyElliptic}
	\item Upper case $U$ will be functions $\real^{d+1}\to\real$, and lower case $u$ will be functions $\real^d\to\real$.
	\item similarly, we try to keep the convention that upper case letters, $X\in\real^{d+1}$, and lower case letters, $x\in\real^{d}$.
	\item $\displaystyle D_f=\{(x,x_{d+1})\in\real^d\times\real_+\ :\ 0<x_{d+1}<f(x)\}$
	\item $\displaystyle \Gamma_f=\graph(f)$, and in most cases, $\Gamma_f=\partial\{U_f>0\}$.
	\item $\displaystyle \Gamma_0=\real^d\times\{0\}$
	\item $\displaystyle \Gamma_L=\real^d\times\{L\}$, where $L>0$.
	\item $n$ is the inward normal vector to $D_f$, along $\Gamma_f$, which in most cases is the inward normal vector to $\{U>0\}$.	
	\item For $X_0\in\Gamma_f$, \\$\displaystyle\partial^+_nU(X_0)=\lim_{t\to0+}\frac{U(X_0+tn(X_0))-U(X_0)}{t}$ and  $\displaystyle\partial^-_nU(X_0)=-\lim_{t\to0+}\frac{U(X_0-tn(X_0))-U(X_0)}{t}$ (note that in most cases, $-tn(X_0)\in\{U<0\}$ and $U(X_0)=0$; $\partial^-_n U$ is normalized to be non-negative).
	\item The translation operator, $\tau_h$, is given in Definition \ref{def:TranslationOperator}.
	\item $\displaystyle C^{1,\gam}(\real^d)=\{f:\real^d\to\real\ :\ \norm{f}_{L^\infty}+[f]_{C^\gam} + \norm{\grad f}_{L^\infty} + [\grad f]_{C^{\gam}} < \infty \}$
	\item (equivalently) 
	\begin{align*}
	\displaystyle C^{1,\gam}(\real^d)=\{f\in L^\infty(\real^d) :\ 
	\sup_{z\in\real^d}\sup_{r>0}r^{-1-\gam}\inf_{P(x)=c+p\cdot x\ :\ c\in\real,\ p\in\real^d} \norm{f-P}_{L^\infty(B_r(z))}   <\infty   \}
	\end{align*}
	\item Punctually $C^{1,\gam}(x)$ is defined in Definition \ref{defNO:PointwiseC1gam}.
	\item For $\gam\in(0,1)$, $\del>0$, $m>0$, the convex set, $\K(\gam,\del,m)$ is defined as
    \begin{align*} 
      \mathcal{K}(\gam,\delta,m) := \{ f \in C^{1,\gam}(\mathbb{R}^d) \mid  f(x)\geq \delta\;\;\forall\;x\in\mathbb{R}^d,\;\|f\|_{C^{1,\gam}(\mathbb{R}^d)} \leq m\}.
    \end{align*}
	\item $\displaystyle \K(\gam,\del)=\Union_{m>0}\K(\gam,\del,m)$.
	\item $\displaystyle \K(\del)=\Union_{\gam\in(0,1)}\Union_{m>0}\K(\gam,\del,m)$.
	\item $\K^*(\gam,\del,m)$ and $\K_*(\gam,\del,m)$ appear after Definition \ref{defNO:PointwiseC1gam}, in equations (\ref{eqNO:DefOfKUpStar}), (\ref{eqNO:DefOfKLowStar}).
\end{itemize}

\section{Three examples}\label{sec:Examples}

Here, we mention three natural examples to which Theorems \ref{thm:MainMetaVersion} and \ref{thm:MetaILipAndMinMax} will apply.  The first two examples are treated explicitly in this paper, and the third is closely related, but not explicitly treated.  We want to point out that although the equations that govern the pressure and free boundary may be linear, the resulting nonlocal parabolic flow is fully nonlinear.

\subsection{One phase Hele-Shaw on a half space}

To every $f:\mathbb{R}^d\to \mathbb{R}$ which is continuous, non-negative, bounded, and bounded away from zero we associate the domain
\begin{align*}
  D_f := \{ X= (x,x_{d+1}) \in \mathbb{R}^{d+1} \mid 0<  x_{d+1} < f(x) \},
\end{align*}
as well as the hypersurface
\begin{align*}
  \Gamma_f := \{ X=(x,x_{d+1}) \in \mathbb{R}^{d+1} \mid x_{d+1} = f(x) \}.
\end{align*}
Let $U_f :D_f \to \mathbb{R}$ be the unique bounded solution to the Dirichlet problem
\begin{align*}
  \begin{cases}
    \Delta U  = 0 &\textnormal{ in } D_f,\\
    U  = 1 &\textnormal{ on } \{x_{d+1}=0\},\\
    U  = 0 &\textnormal{ on } \Gamma_f.	 	  
  \end{cases}	  
\end{align*}

Then, with $n$ denoting the unit normal to $\Gamma_f$ pointing towards $D_f$, define
\begin{align*}
  I(f,x) := (\partial^+_n U_f)(x,f(x)).
\end{align*}

\begin{center}
  \includegraphics[scale=0.75]{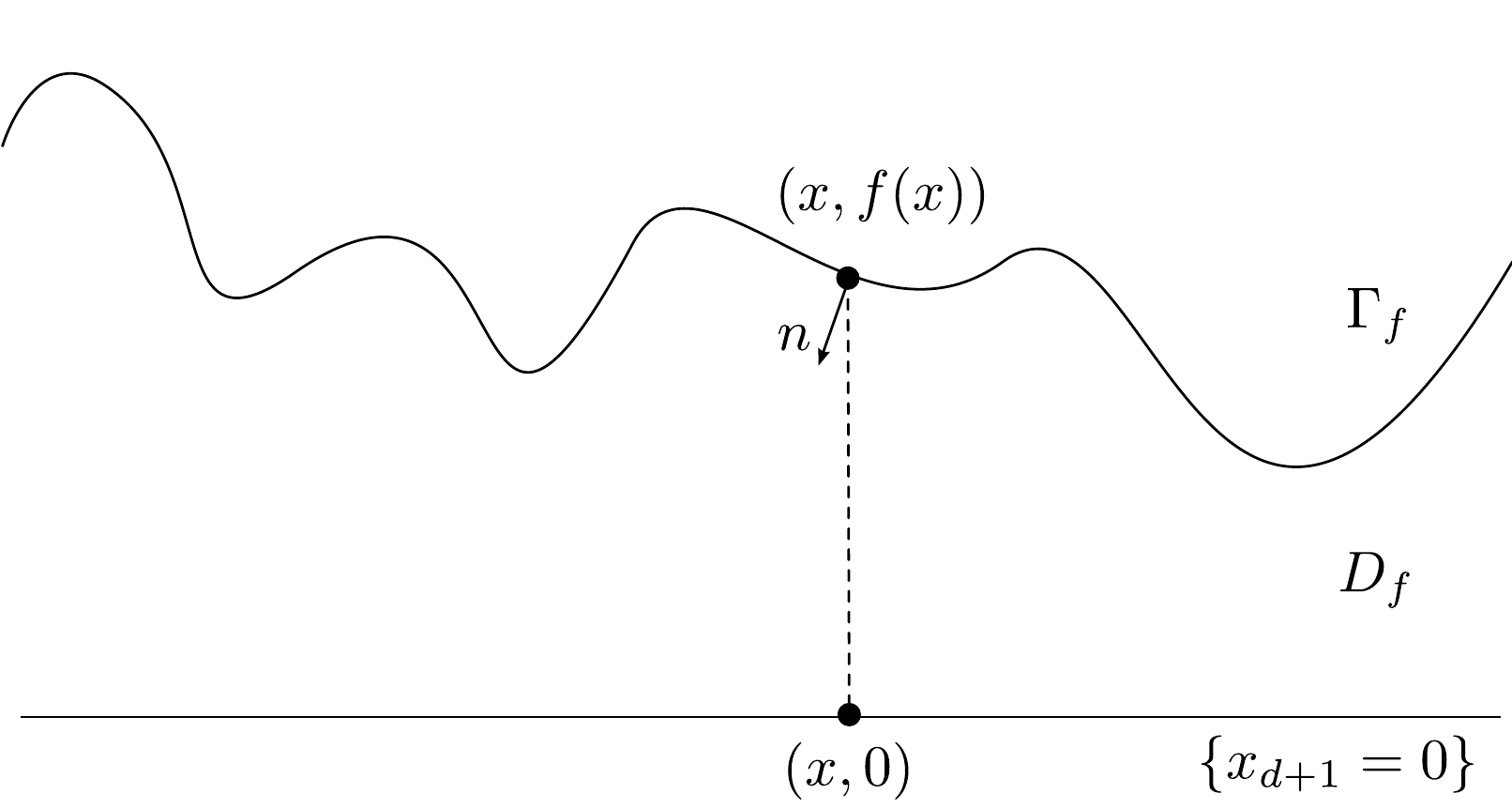}
\end{center}

For sufficiently smooth $f$, the property to be a solution of 
\begin{align*}
  \partial_t f = I(f,x)\cdot\sqrt{1+|\nabla f|^2},\ \ \text{on}\ \ \real^d\times[0,T],
\end{align*}
is equivalent to $U=U_f$ being a solution of the free boundary problem
\begin{align*}
  \begin{cases}
    \Delta U  = 0 & \textnormal{ in } \{U> 0\},\\
    U  = 1 & \textnormal{ along } \{x_{d+1}=0\},\\
    V  = \partial_n^+ U & \textnormal{ along } \partial\{U>0\}.
    \end{cases}	
\end{align*}
This is the one-phase Hele-Shaw problem on the upper half space, and it is not hard to check that the map, $I$, enjoys the GCP. The operator $I$ is easily seen to be translation invariant, and we emphasize that it is also nonlinear and nonlocal. The existence/comparison results in Sections \ref{sec:Comparison}-\ref{sec:WeakSolutions} and regularity result in Section \ref{sec:PropagationOfModulus} apply to this problem.

\subsection{Two-phase problems along an infinite strip}  In our framework, dealing with the two-phase problem is nearly at the same level of complexity and difficulty as the one-phase problem.  We fix an upper boundary, $L>0$, and a globally Lipschitz function
\begin{align*}
  G: (0,\infty)^2 \to \mathbb{R},
\end{align*}
which satisfies Assumption \ref{assumptionMM} (d), for example, $G(a,b)= a^2-b^2$.  This function, $G$, will give the normal velocity of the flow, depending upon $\partial^{\pm}_n U$ along $\partial\{U>0\}$.

Now, to every $f:\mathbb{R}^d\to \mathbb{R}$ which is continuous, $0\leq f\leq L$, and bounded away from zero and $L$, we associate the two domains (for $\{U>0\}$ and $\{U<0\}$) 
\begin{align*}
  D^+_f & = \{ (x,x_{d+1}) \mid 0< x_{d+1} <f(x) \},\\
  D^-_f & = \{ (x,x_{d+1}) \mid f(x) < x_{d+1}< L \}, 
\end{align*}	
and we also associate to such $f$ the same hypersurface, $\Gamma_f$, that is defined in the previous example. 	Let $U_f^+ :D_f^+ \to \mathbb{R}$ and $U_f^- :D_f^- \to \mathbb{R}$ be respectively the unique bounded solution to the Dirichlet problem
\begin{align*}
  \begin{cases}
    \Delta U^{\pm}  = 0 &\textnormal{ in } D_f^{\pm}\\
    U^{\pm}  = \pm 1 &\textnormal{ on } \{x_{d+1}=0\} \;\;(\textnormal{resp. } \{x_{d+1}=L\}),\\
    U^{\pm}  = 0 &\textnormal{ on } \Gamma_f.	 	  
  \end{cases}	  
\end{align*}
Then, we define $H(f,x)$ by 
\begin{align*}
  H(f,x) :=  G(  \partial_n^+ U(x,f(x)) , \partial_n^- U(x,f(x))  ).
\end{align*}

\begin{center}
  \includegraphics[scale=0.75]{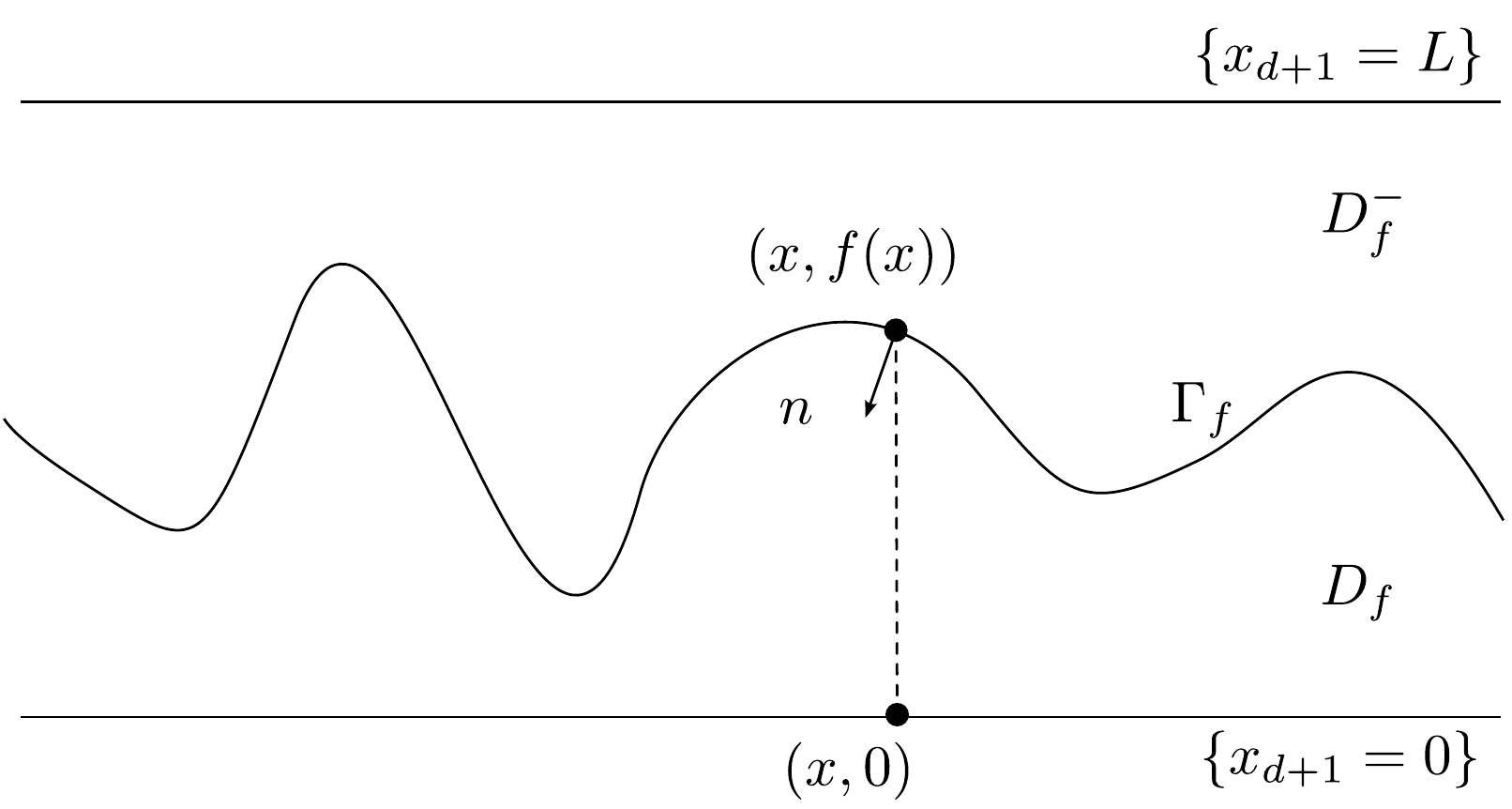}
\end{center}

As in the previous example, as long as $f$ is sufficiently smooth, solving 
\begin{align*}
  \partial_t f = H(f,x)\cdot\sqrt{1+|\nabla f|^2},\ \ \text{on}\ \ \real^d\times[0,T],
\end{align*}
is equivalent to $U=U_{f}$ solving the two-phase free boundary problem
\begin{align*}
  \begin{cases}
  \Delta U = 0 & \textnormal{ in } \{U\neq 0\},\\
    U  = 1 & \textnormal{ on } \{x_{d+1}=0\},\\
    U  = -1 & \textnormal{ on } \{x_{d+1}=L\},\\
    V = G(\partial_n^+ U,\partial_n^- U) & \textnormal{ on } \partial\{U>0\}.
    \end{cases}
\end{align*}
Again, the resulting operator $H$ enjoys the GCP, is translation invariant, nonlinear, nonlocal, and the existence/comparison results in Sections \ref{sec:Comparison}-\ref{sec:WeakSolutions} and regularity propagation result in Section \ref{sec:PropagationOfModulus} apply.  Furthermore, as in Section \ref{sec:TwoPhase}, the reader will see that $H$ is just a combination of two operators like the $I$ from the first example, but interacting through the function, $G$.

\subsection{Prandtl-Batchelor as a non-linear integro-differential equation on the sphere}

This example illustrates how the ideas in this paper could possibly be applied to free boundaries which are not given by a graph over $\mathbb{R}^d$. The Prandtl-Batchelor is a two dimensional model in fluid mechanics that models a vortex patch occupying a convex region, the patch being surrounded by a steady flow (see \cite{Ack-1998}).  The interface of the vortex patch is what plays the role of the free boundary.  Although not covered directly by our results, it is close to them in spirit, and most of the tools we use can be applied to the situation of operators that act on functions over a reference manifold.

In the Prandtl-Batchelor flow we denote the stream function of the flow by $U$, and it satisfies the following
\begin{align*}
\begin{cases}
   \Delta U = 0\ &\textnormal{ in } \{U>0\},\\
   \Delta U = 1\  &\textnormal{ in } \{U<0\},\\
   |\partial^+_n U|^2 - |\partial^-_n U|^2 = 1\  &\textnormal{ on } \partial\{U>0\}.
\end{cases}
\end{align*}
The set $\{U<0\}$ corresponds to the vortex patch. Here we focus on the case in which $\{U<0\}$ is a convex set.

\begin{center}
  \includegraphics[scale=0.75]{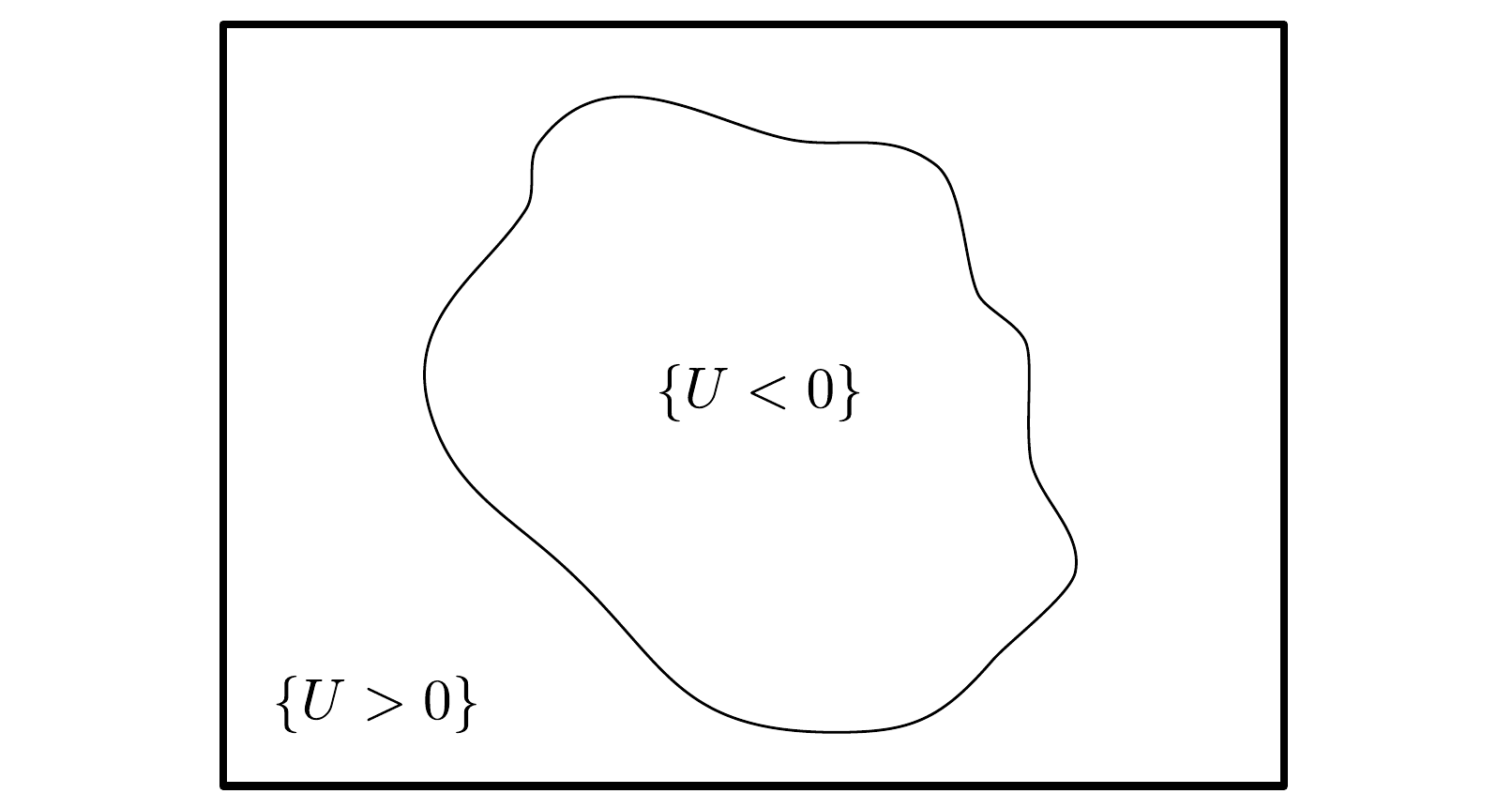}
\end{center}

This problem can be recast as a (steady state) non-linear integro-differential equation on the sphere (or circle). To see this, consider the set of continuous functions  $f:\mathbb{S}^{d-1}\to \mathbb{R}$ such that 
\begin{align*}
  0 < f(x) \leq \overline{f}(x) \textnormal{ in } \mathbb{S}^{d-1},
\end{align*}
where $\overline{f}$ is a given positive, continuous function. For each such $f$, we define the sets
\begin{align*}
   D_f^- & := \{ x \in \mathbb{R}^d \mid  x = re \textnormal{ where } e\in \mathbb{S}^{d-1} \textnormal{ and } 0\leq r<f(e)\},\\
   D_f^+ & := D_{\overline{f}}^- \setminus D_f^-,
\end{align*}
furthermore, we define functions $U_f^{\pm}:D_f^{\pm} \to \mathbb{R}$, each given as the unique solution to the Dirichlet problem
\begin{align*}
  \left \{ \begin{array}{rl}
    \Delta U^{-} & = 1 \textnormal{ in } D_f^{-},\\
    U^{-} & = 0 \textnormal{ on } \partial D_f^-,
  \end{array} \right.	  
\end{align*}
and 
\begin{align*}
  \left \{ \begin{array}{rl}
    \Delta U^{+} & = 0 \textnormal{ in } D_f^{+},\\
    U^{+} & = 0 \textnormal{ on } \partial D_f^-,\\
    U^{+} & = 1 \textnormal{ on } \partial D_{\bar f}.	
  \end{array} \right.	  
\end{align*}
Then, we define an operator $P:C^{2}(\mathbb{S}^{d-1})\to C^0(\mathbb{S}^{d-1})$ by
\begin{align*}
  P(f,x) :=  |\partial_n^+U^+_f(x,f(x))|^2-|\partial_n^- U^-_f(x,f(x))|^2
\end{align*}

\begin{center}
  \includegraphics[scale=0.75]{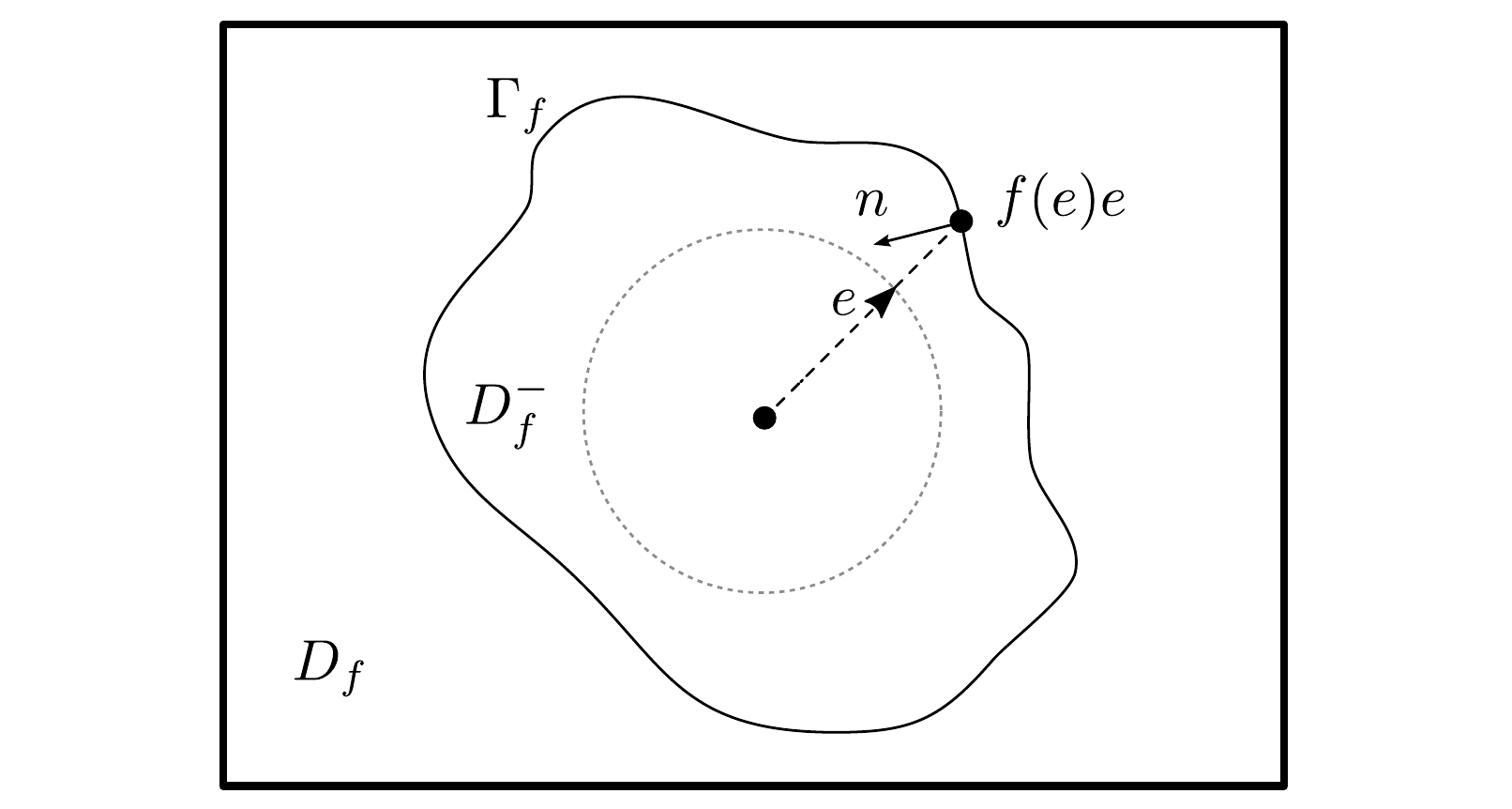}
\end{center}

The operator $P(f,x)$ admits the GCP. Moreover, for sufficiently smooth $f$, solving
\begin{align*}
  P(f,x) & = 1 \textnormal{ in } x \in \mathbb{S}^{d-1},
\end{align*}
means that the (radial) graph of $f$ is the boundary of a vortex patch in a Prandtl-Batchelor flow, and that $U^{\pm}_f$ gives the positive and negative phases of the stream function.  We note that thanks to the invariance of the Laplacian and the boundary conditions by rotations of the domain, the operator, $P$, will also be invariant under the action of rotations on $C^2(\mathbb{S}^{d-1})$.  This rotational invariance would play the same role in the min-max and existence-uniqueness theory as does the translation invariance that we use in Sections \ref{sec:MinMax}, \ref{sec:Comparison}, \ref{sec:PropagationOfModulus}.

\section{Background}\label{sec:Background}

As mentioned above, our goal is to show how some free boundary problems correspond to parabolic integro-differential equations for a scalar over a submanifold. In some aspects, this observation is not new, but in other aspects it does seem new-- we will try to put it in the context of some existing results. For the purposes of the following discussion we will focus solely on free boundary problems involving one scalar and one hypersurface. This, of course, is not the case for all free boundary problems, but will be so for all of the free boundary problems we consider here.  We will also limit our discussion to those free boundary problems that non-stationary.

Examples of free boundary problems include a large class of fluid problems such as cavitation \cite{BirkhoffZarantonello-1957} obstacle problems \cite{Caffarelli-1998,PetShaUra-2012}, the porous medium equation \cite{CaffarelliVazquez1999,Vazquez-2007}, the Muskat problem \cite{SiegelCaflischHowison2004}, problems with surface tension \cite{Luck-1990}, singular perturbation problems \cite{CafVaz-1995}, Prandtl-Batchelor flows \cite{Ack-1998}, sharp interface limits of phase field models \cite{AliBatXin-1994,Cag-1989,CagXin-1998}, and more. 

Many of these problems admit a variational formulation, often in terms of variational inequalities (see for instance Duvaut and Lions \cite{DuvLions-1972} or Kinderlehrer and Stampacchia \cite{KinSta-1980} for further discussion of such problems). In many situations, not necessarily mutually exclusive with the former, the problem has a comparison principle at the level of the scalar field. Such a comparison principle has made it possible to define viscosity solutions, and the origins of this seem to be in the work of Caffarelli for the stationary two-phase problem in \cite{Caffarelli-1988HarnackApproachFBPart3PISA}.  This notion of solution was then used for two-phase Stefan problems in Athanasopoulos-Caffarelli-Salsa in \cite{AthanaCaffarelliSalsa-1996RegFBParabolicPhaseTransitionACTA}, and was later adapted to the Porous Medium equation in work of Caffarelli-Vazquez \cite{CaffarelliVazquez1999}, as well as for the Hele-Shaw problem in work of Kim \cite{Kim-2003UniquenessAndExistenceHeleShawStefanARMA} and the two-phase Stefan problem of Kim-Pozar \cite{KimPozar-2011ViscSolTwoPhaseStefanCPDE}.  In particular, the work of \cite{Kim-2003UniquenessAndExistenceHeleShawStefanARMA} shows existence, uniqueness, and a comparison principle for viscosity solutions for the one phase Hele-Shaw problems.  Regularity for these free boundary problems was studied, for Stefan, in \cite{AthanaCaffarelliSalsa-1996RegFBParabolicPhaseTransitionACTA}, and for Hele-Shaw in  \cite{ChoiJerisonKim-2007RegHSLipInitialAJM}, \cite{ChoiJerisonKim-2009LocalRegularizationOnePhaseINDIANA}, \cite{JerisonKim-2005OnePhaseHSWithSingularityJGA}.  It is interesting to point out that although in a slightly different fashion than what we treat here, there was an occurence of something resembling an integro-differential equation in \cite{ChoiJerisonKim-2007RegHSLipInitialAJM}.  In \cite[Lemma 10.4]{ChoiJerisonKim-2007RegHSLipInitialAJM}, for the one-phase Hele-Shaw flow, they establish that (using our notation above) when $\partial\{U>0\}$ is Lipschitz with a small enough Lipschitz norm, $U$ satisfies an equation of the form
\begin{align*}
	\partial_t U = \int_{\partial\{U>0\}} \abs{\grad U(x,t)}^2 k(x,t;y)d\sigma(y);
\end{align*}
where $k$ is the Poisson kernel for $\partial\{U>0\}$, and $\sigma$ is the surface measure on $\partial\{U>0\}$.

There are free boundary problems where the scalar is constant (say, zero) along the free boundary. The free boundary is then a level set of the solution of a PDE, and accordingly, statements about the free boundary can be reduced to statements about the scalar function itself; such as in the Porous Medium equation or the Hele-Shaw equation.  Matters are different in problems where the free boundary is not a level set of the scalar, such as the Muskat problem, and where the above approach is not as directly applicable. Accordingly, the emphasis does not reside exclusively on the scalar or underlying potential function, but on the free boundary itself, and on how its shape determines the scalar and then its evolution. 

A representative example of this latter situation can be found in work of Cordoba and Gancedo for the Muskat problem. In  \cite{CordobaGancedo2007}, the authors use contour-integration to describe the evolution equation for the interface for the Muskat problem. This requires in particular that the interface at time $t$ is given by the graph of a function $f(x,t)$ defined over $\mathbb{R}^d$. Then, the scalar $f$ is shown to solve the integro-differerential equation \cite[Section 2, Equation 13]{CordobaGancedo2007}
\begin{align*}
  \partial_t f(x,t)  = \frac{\rho_2-\rho_1}{4\pi} \textnormal{P.V.} \int_{\mathbb{R}^2} \frac{(\nabla f(x,t)-\nabla f(x-y,t))\cdot y}{\left ( |y|^2+(f(x,t)-f(x-y,t)) \right )^{\frac{2}{3}} }\;dy 
\end{align*}
Here, $\rho_1$ and $\rho_2$ represent the respective (constant) densities of each of the fluids.  As observed in \cite[Section 2]{CordobaGancedo2007}, the linearization of the above equation at a constant $f$ is the fractional heat equation.

The Muskat problem, also known as the Muskat-Leibenzon problem, describes the interface bounding a fluid, see \cite{SiegelCaflischHowison2004} for a thorough discussion and further references. The Muskat problem is an accurate model for the two-phase regime of a ``Hele-Shaw cell'', and is accordingly known also as the two-phase Hele-Shaw problem. In this problem the fluid is assumed to be irrotational, but the respective pressure is no longer assumed to be constant along the free boundary.  In \cite{Amb-2004} Ambrose obtained local well-posedness for the Muskat problem with initial data in $H^2$ that satisfies the Rayleigh-Taylor condition. Siegel, Caflisch, and Howison \cite{SiegelCaflischHowison2004} studied global solutions with periodic boundary conditions and initial datum close to equilibrium.

The approach in \cite{CordobaGancedo2007} was used in several subsequent results, among which we highlight a few. Constantin et al \cite{ConstantinCordobaGancedoStrain2013} showed the existence of global weak solutions which are globally Lipschitz in space as long as the initial data had a Lipschitz constant strictly smaller than $1$. In contrast,  Castro et al \cite[Theorem 2.1]{CastroCordobaFeffermanGancedoLopezFernandez2011} showed the 2D Muskat problem may develop a singularity (in $C^1$) in finite time, even if the initial data is smooth.  More recently, Cameron  \cite{Cameron2017-GlobalWellPoseMuskatArXiv}, studied a closely related problem, using integro-differential methods combined with the modulus preservation technique of \cite{KiselevVolbergNazarov-2007WellPoseQSGEINVENTMAT}.

An approach amenable to many problems involves changing variables, either to Lagrangian coordinates (for fluid problems) or through the Hanzawa transform, pulling back the free boundary onto a fixed, reference interface, and writing the free boundary problem as a system on a fixed domain. The book of Pr\"uss and Simonett \cite{PruSim-2016} has a thorough presentation of this approach in combination with $L^p$-maximal regularity. The Hanzawa transform, for instance, entails fixing a reference free boundary $\Gamma_0$, and considering interfaces giving as a normal graph over $\Gamma_0$
\begin{align*}
  \Gamma(t) = \{ x + h(x,t)\nu_{\Gamma_0} \mid  x \in \Gamma_0 \}.
\end{align*}
The function $h:\Gamma_0\times [0,T]\to \mathbb{R}$ is known as a \emph{height} function.  The PDE in ``the bulk'' of the new domain will have coefficients determined by the change of variables, and the boundary conditions on this new domain are represented by a coupling between $h$ and the fields along the fixed boundary \cite[p. 33]{PruSim-2016}). Having set up the problem in the fixed domain one can pose this as a non-linear evolution problem in some properly chosen Banach space. The problem can be written, under the right circumstances, as a non-linear perturbation of a linear problem where the linear operator is elliptic. Elliptic here meaning in terms of a Fourier symbol, and operators of order higher than $2$ are allowed (this, in particular makes it possible to treat problems with surface tension). With this approach it is possible to prove short time existence for smooth initial data, global existence near equilibrium, stability of equilibria, and  more. This is done for a large class of problems, including two-phase Stefan problems \cite{PruSim-2007}, problems involving surface tension \cite{EscSim-1997}, and problems where the underlying fields are not necessarily scalar. For works making use of a Lagrangian approach to analyze free boundary problems, we mention work of Hadzic and Shkoller \cite{HadShk-2015} where they obtain global stability for the classical Stefan problem, and Cheng, Coutand, and Shkoller \cite{CheCouShk-2012} where the authors study the Hele-Shaw problem with surface tension.

We see there is a vast literature where free boundary problems are treated by putting the focus on the free boundary in one way or another. Either by assuming it is given by a global graph, and using contour integration to represent the free boundary condition as non-linear integro-differential equation (as done in the Muskat problem) or by representing it as the normal graph over a reference interface (e.g. using the Hanzawa transform) and pulling this back to a reference configuration to obtain a coupled system in a fixed domain.    Finally, we note a result that is similar to our own in both spirit and results:  using a blow up argument and a change of variables, two of the authors in \cite{ChangLaraGuillen-2016FreeBondaryHeleShawNonlocalEqsArXiv} were able to deduce some regularity results of Hele-Shaw flow by invoking recent results for integro-differential equations.  

The approach in this current paper shares features with both of these previous approaches-- there is a nontrivial overlap as they all involve a ``height function'' $h$ to represent the free boundary. However, the work we present here differs a bit in the sense that it is based on the structure of the underlying operators that is enforced by the comparison principle. The result is a description of (certain classes of) problems in terms of a single a scalar solving an equation with non-linear integro-differential operator reminiscent of the Dirichlet to Neumann map. This operator satisfies the global comparison principle so it can be studied via viscosity solutions methods for nonlocal equations and hence also opens up the possibility to subsequently apply non-divergence regularity results obtained in recent years.


\section{The free boundary operator}\label{sec:NewOperator}

This section is dedicated to a study of something that we call a ``free boundary'' operator, which is defined in (\ref{eqIn:BulkOnePhaseExtension}) and (\ref{eqIn:DefOfI}).  Eventually we will prove many properties of $I$, including its Lipschitz nature as a function from special convex subsets of $C^{0,1}$ to $C^{\gam'}$.  The results of this section lead in two different, but related, directions with more or less common goals.  The first is to be able to establish a min-max integro-differential representation for $I$ and subsequently derive some basic properties of the linear operators that make up this min-max (e.g. no second order terms, and negative zero-order terms).  The second direction is to use the Lipschitz nature of $I$, with some special functions, to derive a comparison theorem for sub and super solutions of the fractional parabolic problem, (\ref{eqIn:FractionalParabolic}).  Both of these inquiries are useful in their own right, and of course, they overlap at the stage of the comparison theorem.  They will be developed in the sequels, Sections \ref{sec:MinMax} and \ref{sec:Comparison}.

\subsection{Setup} 

As above, we assume that $F$ is uniformly elliptic and satisfies standard assumptions for existence and uniqueness of viscosity solutions, listed in section \ref{sec:Assumption}.
  
  Given an $f\in C^{1,\gam}(\mathbb{R}^d)$ with $\inf f >0$ we assign $U_f$, the unique viscosity solution to
  \begin{align}\label{eqNO:BulkNonlinearDefI}
	  \begin{cases}
	      F(D^2U_f,\nabla U_f) = 0 & \textnormal{in } D_f,\\
	      U_f = 0 & \textnormal{on } \Gamma_f,\\
	      U_f =  1 & \textnormal{on } \Gamma_0.
	  \end{cases}	
  \end{align}
  We recall that for shorthand purposes, in this section, $\Gamma_f=\graph(f)$ and $\Gamma_0=\real^d\times\{0\}$ (Section \ref{sec:Notation} has this and other notations).
  We mention some sufficient conditions for the existence of such a $U_f$ below.
  We note that, thanks to the boundary gradient estimates in  \cite[Remark 3.2]{Trudinger-1988HolderGradientFullyNonlinearPRO-ROY-SOC-ED}, for some $\gam'\in(0,\gam]$, this $U_f$ will enjoy a global in $D_f$ estimate, $U_f\in C^{1,\gam'}(\overline{D_f})$.
  Thus, we may define the operator
  \begin{align}\label{eqNO:DefOfI}
    I(f,x) := \partial_n U_f(x,f(x)),\ \ (\text{equivalently},\ I(f,x)=|\nabla U_f (x,f(x))|).
  \end{align}
  We see that for this $\gam'$, we have
  \begin{align*}
	  I: C^{1,\gam}(\real^d)\intersect\{f\ :\ \inf f>0\} \to C^{\gam'}(\real^d).
  \end{align*}
  The study of this operator will be our chief concern.   It turns out that there are more relaxed situations in which $I$ is well defined, and we build these results below.  They will use a generalization of semi-concavity that is well known in the field of optimal transport (usually referred to as $c$-convexity), and a special case is the next definition. 
  \begin{definition}[$C^{1,\gam}$-semi-concave]\label{defNO:SemiConvex}
    Let $\gam \in (0,1]$ and $m>0$. A Lipschitz function $f:\mathbb{R}^d\to \mathbb{R}$ will be said to be $C^{1,\gam}$-semi-concave with constant $m$ if there is a real valued function, $r(y)$, such that
    \begin{align*}
      f(x) = \inf \limits_{y \in \mathbb{R}^d} \{ r(y)+m|x-y|^{1+\gam}\}.		
    \end{align*}		
    A function $f$ is said to be $C^{1,\gam}$-semi-convex with constant $m$ if $(-f)$ is $C^{1,\gam}$-semi-concave with constant $m$. 	
	
  \end{definition}
  
  We also need the notion of a function to be pointwise ``$C^{1,\gam}$''.  
  
  \begin{definition}[Pointwise $C^{1,\gam}$]\label{defNO:PointwiseC1gam}
	  As above, let $\gam\in(0,1]$ and $m>0$ be fixed.  We say that $f:\real^d\to\real$ is pointwise $m\text{-}C^{1,\gam}$ at $x_0$, denoted $f\in m\text{-}C^{1,\gam}(x_0)$, if $\grad f(x_0)$ exists, there exists $r>0$, such that
	  \begin{align*}
		  &\abs{f(x_0)}\leq m,\ \abs{\grad f(x_0)}\leq m,\\ 
		  \text{and}\ \ \forall\ x\in B_r(x_0),\ 
		  &\abs{f(x)-f(x_0)-\grad f(x_0)\cdot (x-x_0)}\leq  m\abs{x-x_0}^{1+\gam}.
	  \end{align*}
  \end{definition}

  For $m,\delta>0$ and $\gam \in (0,1)$, we consider the convex set of functions
  \begin{align}\label{eqNO:DefOfK} 
    \mathcal{K}(\gam,\delta,m) := \{ f \in C^{1,\gam}(\mathbb{R}^d) \mid  f(x)> \delta\;\;\forall\;x\in\mathbb{R}^d,\;\|f\|_{C^{1,\gam}(\mathbb{R}^d)} < m\},
  \end{align}
  as well as a respective convex sets of ``semi-concave/convex'' functions, which are larger,
  \begin{align} 
    \mathcal{K}^*(\gam,\delta,m) & := \{ f \in C^{0,1}(\mathbb{R}^d) \mid  f(x)> \delta\;\;\forall\;x\in\mathbb{R}^d, f \textnormal{ is } C^{1,\gam}\textnormal{-semi-concave with constant } m \;\},\label{eqNO:DefOfKUpStar}\\
 \mathcal{K}_*(\gam,\delta,m) & := \{ f \in C^{0,1}(\mathbb{R}^d) \mid  f(x)> \delta\;\;\forall\;x\in\mathbb{R}^d, f \textnormal{ is } C^{1,\gam}\textnormal{-semi-convex with constant } m \;\}.\label{eqNO:DefOfKLowStar}
  \end{align}
  It is clear that we have the inclusion
  \begin{align*}
    \mathcal{K}(\delta,\gam,m) \subset \mathcal{K}^*(\delta,\gam,m) \cap \mathcal{K}_*(\delta,\gam,m).
  \end{align*}

	Our goal, as noted above, is that we intend to use $I$ to define (degenerate) parabolic equations, and to show a comparison theorem for viscosity solutions of these equations.  This means that there are, more or less, three primary concerns:
  
  \begin{enumerate}
  	\item For a fixed $x\in\real^d$, for which functions, $f$, is $I(f,x)$ classically defined?  Certainly, we will require that there is a $\delta>0$ so that $f\geq \delta$.  Furthermore, it is not too hard to show, and we do below, that $I(f,x)$ is well defined if also for some $r>0$, $f\in C^{0,1}(\real^d)\intersect C^{1,\gam}(B_r(x))$.  However, we prefer to have a slightly weaker situation, and indeed, we are able to show that $I(f,x)$ is still well defined when $f\in C^{0,1}(\real^d)\intersect\left( m\text{-}C^{1,\gam}(x)\right)$.
  	\item Over which collection of functions does $I$ enjoy the GCP?  (The GCP appears in Definition \ref{def:GCP}.)
  	\item Over which set is the mapping $I$ Lipschitz?  We will show that $I$ is locally Lipschitz on the convex sets given by $\mathcal{K}^*_{\delta,\gam,m}$ above, and that the Lipschitz norm grows as either $\del$ decreases or $m$ increases.  (As a corollary, $I$ will be locally Lipschitz on $\K(\del,\gam,m)$.)
  \end{enumerate}

  
\subsection{Basic properties}

We have defined $I(f,x)$ in (\ref{eqNO:DefOfI}) for functions which are globally of class $C^{1,\gam}$ for some $\gam\in(0,1)$.   We now carefully check the existence and uniqueness for $U_f$ as well as the well-posedness of $I(f,x)$ in some less restrictive situations.  Eventually, we show that $I$ is well defined whenever $f$ is either Lipchitz and locally $C^{1,\gam}$ in a neighborhood of the point of evaluation, or when $f$ is Lipschitz and $C^{1,\gam}$-semi-concave.

\begin{proposition}\label{propNO:ExistenceForUf}
	If $f\in C^{0,1}(\real^d)$ and $f\geq \del$, then there exists a unique $U_f\in C(\overline{D_f})$ that is the viscosity solution of (\ref{eqNO:BulkNonlinearDefI}) and continuously attains its boundary values.
\end{proposition}

\begin{proof}[Proof of Proposition \ref{propNO:ExistenceForUf}]
	This proposition is more or less standard for viscosity solutions, but we were unable to find a standard reference that contained the particular situation stated above.  Thus, we have included some of the main details.

	So long as we can construct a barrier for $U_f$ to force it to attain its boundary values on $\Gam_f$ continuously, we will have the existence and uniqueness from the results by, e.g. Ishii \cite{Ishii-1989UniqueViscSolSecondOrderCPAM}.  This is a consequence of the fact that the Perron Method will produce a solution in the interior, and so we may assume that $U_f$ is already defined in $D_f$.  We just focus on the boundary values.  
	
	We first note that the fact that $\Gam_0=\real^d\times\{0\}$ is flat (hence $C^2$), the existence of lower and upper barriers attaining the value $1$ on $\Gam_0$ is standard.  We instead focus our attention on barriers at the upper boundary, $\Gam_f$.

	 The Lipschitz nature of $f$ gives $D_f$ an exterior cone condition, and so it is possible to construct an upper barrier for (\ref{eqNO:BulkNonlinearDefI}) (we note that the constant, zero, function serves as a lower barrier).  Let $X_0=(x_0,f(x_0))\in\Gam_f$ be fixed.  We will construct a function $\psi_{X_0}$ that serves as an upper barrier for $U_f$ at $X_0$.  Let $\phi$ be a continuous function such that 
	\begin{align*}
		\phi(X_0)=0,\ \phi\geq 0,\ \phi(Y)\geq 1\ \text{for}\ Y\in \left( \partial Q_{\del}(X_0)\intersect \overline{D_f}\right),
	\end{align*}
	where we use $Q_\del(X_0)$ to be the cube of side length, $\del$, centered at $X_0$.
	We note that since $D_f$ enjoys the exterior cone condition, then so does $Q_{\del}(X_0)\Intersect D_f$.  Thus, by \cite[Corollary 3.10]{CaCrKoSw-96}, there exists a unique solution, $\psi_{X_0}\in C^{2,\gam}(Q_{\del}(X_0)\Intersect D_f)\Intersect C(\overline{Q_{\del}(X_0)\Intersect D_f})$, to the extremal equation,
	\begin{align*}
		\begin{cases}
			\M^+(D^2\psi_{X_0})+\abs{\grad \psi_{X_0}}=0\ &\text{in}\ Q_{\del}(X_0)\Intersect D_f\\
			\psi_{X_0}=\phi\ &\text{on}\ \partial\left( Q_{\del}(X_0)\Intersect D_f  \right).
		\end{cases}
	\end{align*}
	From the definition of uniform ellipticity, it follows that this $\psi_{X_0}$ is a viscosity supersolution of $F(D^2U,\grad U)=0$ in the domain $Q_{\del}(X_0)\Intersect D_f$.

	Because we know that $U_f\leq 1$ in $\overline{D_f} $ and $U_f=0$ on $\Gam_f$, we see that $U_f\leq \psi$ on $\partial(Q_{\del}(X_0)\Intersect D_f)$.  Thus, by comparison, we see that
	\begin{align*}
		0\leq U_f(X_0)\leq \psi(X_0)=0.
	\end{align*}
	Hence, $U_f(X_0)=0$.  Since $X_0$ was generic, we conclude that such a $U_f$ exists and continuously attains $U_f=0$ on $\Gam_f$.

\end{proof}

\begin{proposition}\label{propNO:UfMonotoneIn-f}
  If $f,g\in C^{0,1}(\real^d)$, $f\geq \del$, $g\geq\del$, and $f\leq g$ in $\mathbb{R}^d$, then for $U_f$, $U_g$ solving (\ref{eqNO:BulkNonlinearDefI}),
  \begin{align*}
    U_f \leq U_g \textnormal{ in } D_f.
  \end{align*}
\end{proposition}

\begin{proof}[Proof of Proposition \ref{propNO:UfMonotoneIn-f}]
	First, we note that $f\leq g$ implies that $D_f\subset D_g$.  Furthermore, since $f=0$ on $\Gam_f$ and since $g\geq0$ in $D_f$, we see that $U_g$ is a supersolution of (\ref{eqNO:BulkNonlinearDefI}) for $f$.  Hence, $U_f\leq U_g$ is a direct application of the comparison theorem for elliptic equations in \cite{Ishii-1989UniqueViscSolSecondOrderCPAM}.

\end{proof}

\begin{proposition}\label{propNO:WellDefinedLocalC1Al}
  Let $f:\mathbb{R}^d\to \mathbb{R}$ be a bounded, globally Lipschitz function such that $f \geq \delta$ in $\mathbb{R}^d$ for some $\delta>0$, and let $x_0 \in \mathbb{R}^d$. If $f$ is $C^{1,\gam}(B_r(x_0))$, for some fixed $r\in (0,1)$ then $I(f,x_0)$ is well defined. Moreover, there is a constant $C$, depending on $d,\Lambda,\lambda$ and the $C^{1,\gam}(B_r(x_0))$ norm of $f$, such that
  \begin{align*}
    |I(f,x_0)| \leq Cr^{-1}.	  
  \end{align*}	  
\end{proposition}

\begin{proof}
  The assumptions on $f$ have already been shown in Proposition \ref{propNO:ExistenceForUf} to give the existence and uniqueness of $U_f$. In particular, $U_f$ satisfies in the viscosity sense
  \begin{align*}
    F(D^2U_f,\nabla U_f) = 0 \textnormal{ in } D_f \cap B_r^{d+1}((x_0,f(x_0))).
  \end{align*}
  Since $\partial D_f$ is of class $C^{1,\gam}$, from \cite[Remark 3.2]{Trudinger-1988HolderGradientFullyNonlinearPRO-ROY-SOC-ED} it follows that $U_f$ is of class $C^{1,\gam'}$ in the smaller domain $D_f \cap B_{\frac{r}{2}}^{d+1}((x_0,f(x_0)))$, with the estimate
  \begin{align*}
    \|\nabla U_f\|_{L^\infty(D_f \cap B_{\frac{r}{2}}^{d+1}((x_0,f(x_0))))} \leq Cr^{-1}\|U_f\|_{L^\infty(D_f)} \leq Cr^{-1}.
  \end{align*}	 
  
\end{proof}

\begin{lemma}\label{lemNO:IHasGCP}
 Assume that $x_0\in\real^d$, $\del>0$, and $r>0$ are fixed. The map, $I$, has the global comparison property based at $x_0$ (Definition \ref{def:GCP}) for functions, $f,g$ that satisfy
\begin{align*}
	f,g\in \left(C^{0,1}(\real^d)\Intersect C^{1,\gam}(B_{r}(x_0))\Intersect \{h:\real^d\to\real\ :\ h\geq \del\}\right).
\end{align*}
Consequently $I$ also enjoys the GCP for functions in $\K(\gam,\del)=\Union_m \K(\gam,\del,m)$. 
\end{lemma}

\begin{proof}
  Let $f,g \in C^{0,1}(\real^d)\Intersect C^{1,\gam}(B_{r}(x_0))$ and $x_0 \in \mathbb{R}^d$ be such that  
  \begin{align*}
    f(x) \leq g(x)\;\forall\;x\in\mathbb{R}^d,\;f(x_0)=g(x_0).    	
  \end{align*}
  It is immediate that
  \begin{align*}
    D_f \subset D_g,\ \ \text{and}\ \ X_0:= (x_0,f(x_0)) = (x_0,g(x_0)) \in \Gamma_f \cap \Gamma_g.
  \end{align*}	
  Now, since the boundary values of $U_g$ are nonnegative, the maximum principle shows that
  \begin{align*}	
    U_g \geq 0 \textnormal{ in } D_g.    	
  \end{align*}	
  In particular, since $\Gamma_f \subset D_f$, we have $U_g \geq 0$ on $\Gamma_f$. Thus, $U_g \geq U_f$ on the boundary of $D_f$. Then, from the comparison principle, it follows that
  \begin{align*}	
    U_g\geq U_f\;\;\textnormal{ in } D_f. 	
  \end{align*}	
  As both functions vanish at $X_0$, we conclude that at $X_0$
  \begin{align*}
    |\nabla U_g| \geq |\nabla U_f|,	  
  \end{align*}	  
  in other words, $I(f,x_0) \leq I(g,x_0)$, as we wanted.

\end{proof}

The next proposition says the translation invariance of our setup implies, as one would expect, that the operator $I$ itself is translation invariant.

\begin{proposition}\label{propNO:TranslationInvariant}
  For each $\gam\in(0,\infty)$ and $\del>0$, $I:\K(\gam,\del)\to C^0$ is translation invariant; i.e. for $f$ given, then for any $h\in\mathbb{R}^d$ that is fixed, we have
  \begin{align*}
    \forall\ x\in\real^d,\ \ [\tau_h I(f)](x) = I(\tau_h f,x).	  
  \end{align*}	  
\end{proposition}

\begin{proof}
	  For $h\in\real^d$, fixed, let us extend the translation operator, $\tau_h$, to act on functions on $\real^{d+1}$ as 
	\[
	W:\real^{d+1}\to\real,\ (x,x')\in\real^{d+1}, \ \tau_h W(x,x'):= W(x+h,x').
	\]
  The proposition will follow immediately from the observation
  \begin{align*}
    \tau_h U_{f} = U_{\tau_h f},
  \end{align*}
  which itself is a consequence of the fact that the operator, $F$, and the lower boundary of $D_f$ are translation invariant.
  Indeed, if we define $W=\tau_h U_{f}$, we see that 
\begin{align*}
	F(D^2W(x,x'),\grad W(x,x'))=0,\ \ \text{whenever}\ \ (x+h,x')\in D_f;  
\end{align*}
and furthermore,
\begin{align*}
	W(x,f(x+h))=U_f(x+h,f(x+h))=0\ \ \text{and}\ \ W(x,0)=U_f(x,0)=0.
\end{align*}
Hence $W$ solves (\ref{eqNO:BulkNonlinearDefI}) in the domain $D_{\tau_h f}$.  By the uniqueness of solutions of (\ref{eqNO:BulkNonlinearDefI}), we see that we have
\begin{align*}
	\tau_h U_f = W= U_{\tau_h f}.
\end{align*}
Since the operator $\partial_n$ commutes with $\tau_h$, we conclude that
\[
\tau_h(\partial_n U_f)(x)=\partial_n(\tau_h U_f)(x)=\partial_n U_{\tau_h f}(x),
\]
whence
\[
\tau_h[I(f)](x)=I(\tau_h f,x).
\]

\end{proof}

 The next two propositions say respectively that $U_f$ is monotone increasing in $f$, and that if $f$ and $g$ are close in $L^\infty$, then $U_f$ and $U_g$ are also close in $L^\infty$ in their common domain.

\begin{proposition}\label{propNO:SemiconcavityImpliesUfLipschitzAndBoundsOnNormalDeriv}
  There is a constant $C = C(d,\lambda,\Lambda,\delta,\gam,m)$ such that for all $f\in \mathcal{K}^*(\gam,\delta,m)\intersect C^{0,1}(\real^d)$,
\begin{align*}
    \forall\ (x,x_{d+1}) &\in D_f,\ \ U_f(x,x_{d+1}) \leq C(f(x)-x_{d+1}).
\end{align*}
Furthermore, 
\begin{align*}
	\grad U_f\in L^\infty(D_f)\ \ \text{and}\ \ 
	 0\leq\partial_n U_f(x,f(x))\leq C,\ \text{whenever}\ \partial_n U_f\ \text{exists}.	 
  \end{align*}	  
\end{proposition}

\begin{proof}[Proof of Proposition \ref{propNO:SemiconcavityImpliesUfLipschitzAndBoundsOnNormalDeriv}]
	 Assume that $x_0$ is fixed.  Since $f$ is $C^{1,\gam}$-semi-concave, there exists a $\psi\in C^{1,\gam}(\real^d)$ so that $f\leq \psi$ and $f(x_0)=\psi(x_0)$, and $\norm{\psi}_{C^{1,\gam}}$ depends only on $\del$  and $m$.  Thus, $D_f\subset D_{\psi}$. Furthermore, we know already, from \cite[Remark 3.2]{Trudinger-1988HolderGradientFullyNonlinearPRO-ROY-SOC-ED} that there is some $\gam'$ so that $U_\psi\in C^{1,\gam'}(\overline{D_\psi})$.  In particular, $U_\psi$ is globally Lipschitz in $\overline{D_\psi}$.  Finally, because of the ordering of $f\leq\psi$ and $f(x_0)=\psi(x_0)$, we see that both $U_f\leq U_\psi$ (from Proposition \ref{propNO:UfMonotoneIn-f}) and $0\leq \partial_n U_f(x_0,f(x_0))\leq \partial_n U_\psi(x_0,\psi(x_0))$.  The $C^{1,\gam}$ nature of $\psi$ that depends only on $\del$ and $m$ means that the Lipschitz norm of $U_\psi$ depends only on $d$, $\lam$, $\Lam$, $\del$, $\gam$, and $m$.  Since this Lipschitz property of $U_\psi$ implies the result of the lemma with $U_f$ replaced by $U_\psi$, we conclude by the previously noted ordering of $U_f\leq U_\psi$, that the outcome of the lemma is valid for $U_f$ as well.
\end{proof}

\begin{proposition}\label{propNO:fToUfisLipschitz}
  There is a constant $C=C(d,\lambda,\Lambda,\delta,\gam,m)$ such that if $f,g \in \mathcal{K}^*(\gam,\delta,m)$, then
  \begin{align*}
    \|U_f-U_{g}\|_{L^\infty(D_f \cap D_g)} \leq C\|f-g\|_{L^\infty(\mathbb{R}^d)}.	  
  \end{align*}	  
\end{proposition}

\begin{proof}
  Let us show there is a $C=C(\lambda,\Lambda,\delta,\gam,m)$ such that if $f\in \mathcal{K}^*(\delta,\gam,m)$ and $s>0$, then
  \begin{align}\label{eqNO:UfLipschitzWrtConstants}
    U_{f+s} \leq  U_{f}+Cs \textnormal{ in } D_f.         
  \end{align}
  Indeed, by Proposition \ref{propNO:SemiconcavityImpliesUfLipschitzAndBoundsOnNormalDeriv}
  \begin{align*}
    U_{f+s}(x,x_{d+1}) \leq C(f(x)+s-x_{d+1})_+\;\;\forall\;(x,x_{d+1}) \in D_{f+s},
  \end{align*}
  In particular, $U_{f+s}(x,f(x)) \leq Cs$ for every $x\in \mathbb{R}^d$, which is the same as
  \begin{align*}
    U_{f+s} \leq C s \textnormal{ on } \Gamma_f. 
  \end{align*}
  Then, if $\tilde U := U_{f}+Cs$ with this same $C$, we have
  \begin{align*}
     U_{f+s}\leq \tilde U\textnormal{ on } \partial D_{f}.
  \end{align*}
  Furthermore, $\tilde U$ is a viscosity solution of $F(D^2U,\nabla U)=0$ in $D_f$, so by the comparison principle, $U_{f+s}\leq \tilde U$ everywhere in $D_f$, and \eqref{eqNO:UfLipschitzWrtConstants} is proved. Now, given a second function $g \in \mathcal{K}^*(\delta,\gam,m)$, let $s=\|f-g\|_\infty$, so that
  \begin{align*}
    g \leq f+s.
  \end{align*}
  Then, by Proposition \ref{propNO:UfMonotoneIn-f} we have $U_g \leq U_{f+s} \textnormal{ in } D_{g}$. Applying \eqref{eqNO:UfLipschitzWrtConstants}, it follows that
  \begin{align*}
    U_g \leq U_f+Cs \textnormal{ in } D_g.
  \end{align*}
  Arguing in the exact same manner but reversing the roles of $f$ and $g$, we conclude that (with the same constant $C$ as before)
  \begin{align*}
    U_f \leq U_g+Cs \textnormal{ in } D_f,
  \end{align*}
  and this proves the proposition.
\end{proof}

In the next proposition, we state (with an abreviated proof) a basic estimate for a one parameter family of barrier functions $\{H_s\}_{s>0}$.

\begin{proposition}\label{propNO:ClassicalEvaluationLowerBarrier}
  Fix $w\in \mathcal{K}(\gam,\delta,m)$ and $s\in(0,\infty)$. Let $H_s:D_w \to \mathbb{R}$ be the unique viscosity solution of
  \begin{align*}
    \begin{cases}
	 F(D^2H_s,\nabla H_s) =0\  &\textnormal{in}\  D_w,\\
    H_s =0\  &\textnormal{on}\ \Gamma_w,\\		
    H_s =s\   &\textnormal{on}\  \Gamma_{0}.
	\end{cases}
  \end{align*}
  Then, for some universal $\gam'\in(0,\gam]$ and a constant $C=C(d,\lambda,\Lambda,\gam,\delta,m)$, we have for $X_0 \in \Gamma_w$ and $X\in \overline D_w$
  \begin{align*}
    & |H_s(X)-(\nabla H_s(X_0),X-X_0)| \leq Cs|X-X_0|^{1+\gamma'},\\
    & C^{-1}s \leq |\nabla H_s(X_0)|\leq Cs.
  \end{align*}	 
  Moreover, there is a constant $C$ such that
  \begin{align*}
    |H_{s_1}(X)-H_{s_2}(X)|\leq C|s_1-s_2|d(X,\Gam_w).
  \end{align*}
\end{proposition}

\begin{proof}
	
	We just provide a small sketch of the details.  The first two claims are immediate from the $C^{1,\gam'}$ regularity of solutions in the domain, $D_w$, also using lower and upper barriers to bound the gradient along $\Gam_w$ for the second claim.  The lower bound on $\abs{\grad H_s}$ uses the Hopf principle for fully nonlinear equations.
  
  The third assertion follows from the fact that, if we assume that $s_1\geq s_2$, then we have $0\leq H_{s_1}-H_{s_2}\leq \psi_{up}$, where $\psi_{up}$ is a barrier function that solves
\begin{align*}
	\mathcal{M}^+(D^2\psi_{up})+\Lambda |\nabla \psi_{up}|=0\ \ \text{with}\ \ 
	\psi_{up}=s_1-s_2\ \text{on}\ \Gam_0,\ \ \text{and}\ \ \psi_{up}=0\ \text{on}\ \Gam_w.
\end{align*}
Furthermore, standard regularity theory, e.g. \cite{Trudinger-1988HolderGradientFullyNonlinearPRO-ROY-SOC-ED}, shows that $\psi_{up}\in C^{0,1}(\overline{D_w})$, with
\begin{align*}
	0\leq \psi_{up}\leq s_1-s_2,\ \ \text{and}\ \ \abs{\grad \psi_{up}}\leq C\norm{\psi_{up}}_{L^\infty}.
\end{align*}
The claim follows by using the barrier up to the boundary at $\Gam_w$.
  
\end{proof}

The following Lemma follows an argument about the behavior of harmonic functions near regular points of their boundary.  Here we adapt the details from, e.g. \cite[Lemma 11.17]{CaffarelliSalsa-2005GeometricApproachtoFB}.

\begin{lemma}\label{lemNO:PointwiseEvaluation}
  Assume $f$ is such that $\inf f>0$ and $f$ is differentiable at $x_0$, and furthermore that $f$ satisfies for some $\gamma \in (0,1)$ and $C>0$
  \begin{align*}
    |f(x)-f(x_0)-(\nabla f(x_0),x-x_0)| \leq C|x-x_0|^{1+\gamma}.
  \end{align*}
  Then, the function $U_f$ is differentiable at $X_0=(x_0,f(x_0))$, in the sense that for some $\alpha>0$ 
  \begin{align*}
    U_f(X) = \alpha(n(X_0),X-X_0) + o(|X-X_0|),
  \end{align*}
  as $X\to X_0$ non-tangentially in $D_f$. Here $n(X_0)$ denotes the inner normal to $D_f$ at $X_0$.
  
\end{lemma}

\begin{proof}
	Let $\del=\inf f>0$.
  Consider the function defined as
  \begin{align*}
    w(x) = \rho_0 \big ( f(x_0) +(\nabla f(x_0),x-x_0)- C|x-x_0|^{1+\gamma} \big ),
  \end{align*}	
  for $x$ close to $x_0$, where $\rho_0$ denotes a smooth, monotone function of one variable such that 
  \begin{align*}
    \rho_0(t) & = t \textnormal{ for } t\geq \delta,\\
    \rho_0(t) & \leq \delta \textnormal{ for } t \leq \delta,\\	
    \rho_0(t) & = \delta/2 \textnormal{ for } t\leq \delta/2.	  
  \end{align*}
  Then, it is clear that $w \in \mathcal{K}(\gamma,\delta/2,m')$ for some $m'=m'(\gamma,\delta,C)$. Moreover, from the assumption on $f$ (that $f(x)\geq f(x_0)+(\nabla f(x_0),x-x_0)- C|x-x_0|^{1+\gamma}$), we see that $w(x_0)=f(x_0)$ and $w\leq f$.  Thus,
  \begin{align*}
    D_w \subset D_f \textnormal{ and } X_0 \in \partial D_w \cap D_f. 
  \end{align*}
  Let $H_s(X)$ be as in Proposition \ref{propNO:ClassicalEvaluationLowerBarrier}, for the domain given by $D_w$. For $k \in \mathbb{N}$ sufficiently large (so that $2^{-k} \leq \delta$) we define
  \begin{align*}
    \alpha_k = \sup \{ s \mid U_f\geq H_s \textnormal{ in } B_{2^{-k}}(X_0) \cap D_w \}.
  \end{align*}
  \begin{center}
    \includegraphics[scale=0.75]{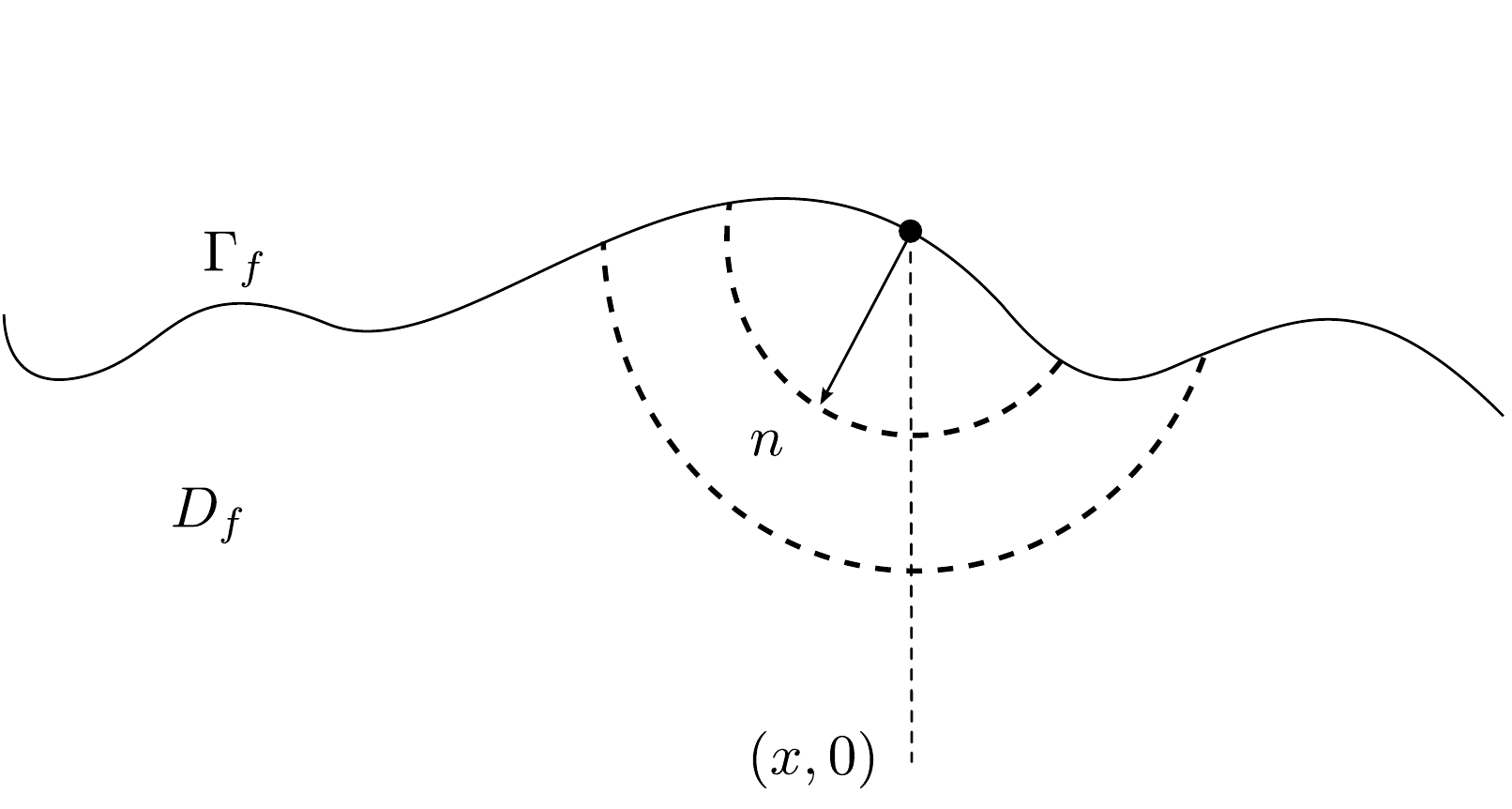}
  \end{center}
  It is clear that the sequence $\alpha_k$ is non-decreasing with respect to $k$. Therefore, there is some real number $\alpha_*$ such that $\alpha_* = \lim \alpha_k$ or else $U_f$ has superlinear growth at $X_0$. The latter option is impossible, since $U_f$ is Lipschitz at $X_0$; seen by using a $C^{1,\gam}$ function, $\psi$, that touches $f$ from above at $x_0$ and arguing as in Proposition \ref{propNO:SemiconcavityImpliesUfLipschitzAndBoundsOnNormalDeriv}. 
 
  Let $\{X_m\}_m$ be any sequence converging to $X_0$ in $D_w$, and let us write
  \begin{align*}
    U_f(X_m) - H_{\alpha_*}(X_m) = U_f(X_m) - H_{\alpha_k }(X_m) + (H_{\alpha_k}(X_m)-H_{\alpha_*}(X_m)),
  \end{align*}
  for any $k$, and let $k_m$ be defined by the inequalities
  \begin{align*}
    2^{-k_m-2} \leq r_m < 2^{-k_m-1},\; r_m := |X_m -X_0|.
  \end{align*}
  Then, for every $m$ we have
  \begin{align*}
    U_f(X_m) -  H_{\alpha_*}(X_m) \geq H_{\alpha_k}(X_m)-H_{\alpha_*}(X_m).
  \end{align*}
  On the other hand, thanks to Proposition \ref{propNO:ClassicalEvaluationLowerBarrier}, there is $C>0$ such that
  \begin{align*}
    |H_{\alpha_k}(X_m)-H_{\alpha_*}(X_m)| \leq C|\alpha_k-\alpha_*|\cdot|X_m-X_0|.
  \end{align*}
  We know that $\alpha_*-\alpha_{k_m} \to 0$ as $m\to \infty$, so the right hand side is $o(|X_m-X_0|)$. Since the sequence $X_m$ was arbitrary, we have proved that 
  \begin{align*}
    U_f(X) \geq H_{\alpha_* }(X)+ o(|X-X_0| ),\;\; X\in D_w.
  \end{align*}
Thus, using Proposition \ref{propNO:ClassicalEvaluationLowerBarrier}, we can establish a further characterization of $\al_*$, which is
  \begin{align*}
    \alpha_* = \sup \big \{ s \mid U_f(X) \geq H_{s}(X) + o(|X-X_0|), \textnormal{ as } X\to X_0 \textnormal{ non-tangentially} \big \}.
  \end{align*}
  Let us show that in fact 
  \begin{align}\label{eqnNO:ClassicalEvaluationLowerBarrierAsymptotic} 
    U_f(X)= H_{\alpha_* }(X)+ o(|X-X_0|), \textnormal{ as } X\to X_0 \textnormal{ non-tangentially}.
  \end{align} 
  We argue by contradiction. If \eqref{eqnNO:ClassicalEvaluationLowerBarrierAsymptotic} does not hold, there is a sequence $X_m \in D_w$ converging non-tangentially to $X_0$ along which $U_f - H_{\alpha_*}$ is larger than a quantity comparable to the distance to $X_0$. In other words, there is some $\theta>0$ such that 
  \begin{align*}
    X_m \to X_0, \; |X_m-X_0| \geq \theta |X_m'-X_0|,\;X_m \in D_{w},
  \end{align*}
  ($X_m'$ is the projection of $X_m$ onto the tangent hyperplane to $\Gam_w$ at $X_0$) and at the same time 
  \begin{align*}
    U_f(X_m) \geq H_{\alpha_* }(X_m)+ \beta |X_m-X_0|,
  \end{align*}
  for some $\beta>0$ independent of $m$. To take advantage of this, let us define for $k\in\mathbb{N}$  the function
  \begin{align*}
    W_{k}(X) := U_f(X) - H_{\alpha_k}(X),\;\; \textnormal{ for } X \in D_w \setminus D_{w-r_0}.
  \end{align*} 
  Since $\alpha_k \leq \alpha_*$ for all $k$, we have
  \begin{align*}
    W_k(X_m) \geq \beta |X_m-X_0| \textnormal{ and } W_k \geq 0 \textnormal{ in } B_{2^{-k}}(X_0) \cap D_w;
  \end{align*}	
  furthermore, $W_k$, satisfies (in the viscosity sense)
  \begin{align*}
    \mathcal{M}^+(D^2W_k)+\Lambda |\nabla W_k| & \geq -C_1,\\
    \mathcal{M}^-(D^2W_k)-\Lambda |\nabla W_k| & \leq C_1.	   
  \end{align*}
  \begin{center}
    \includegraphics[scale=0.75]{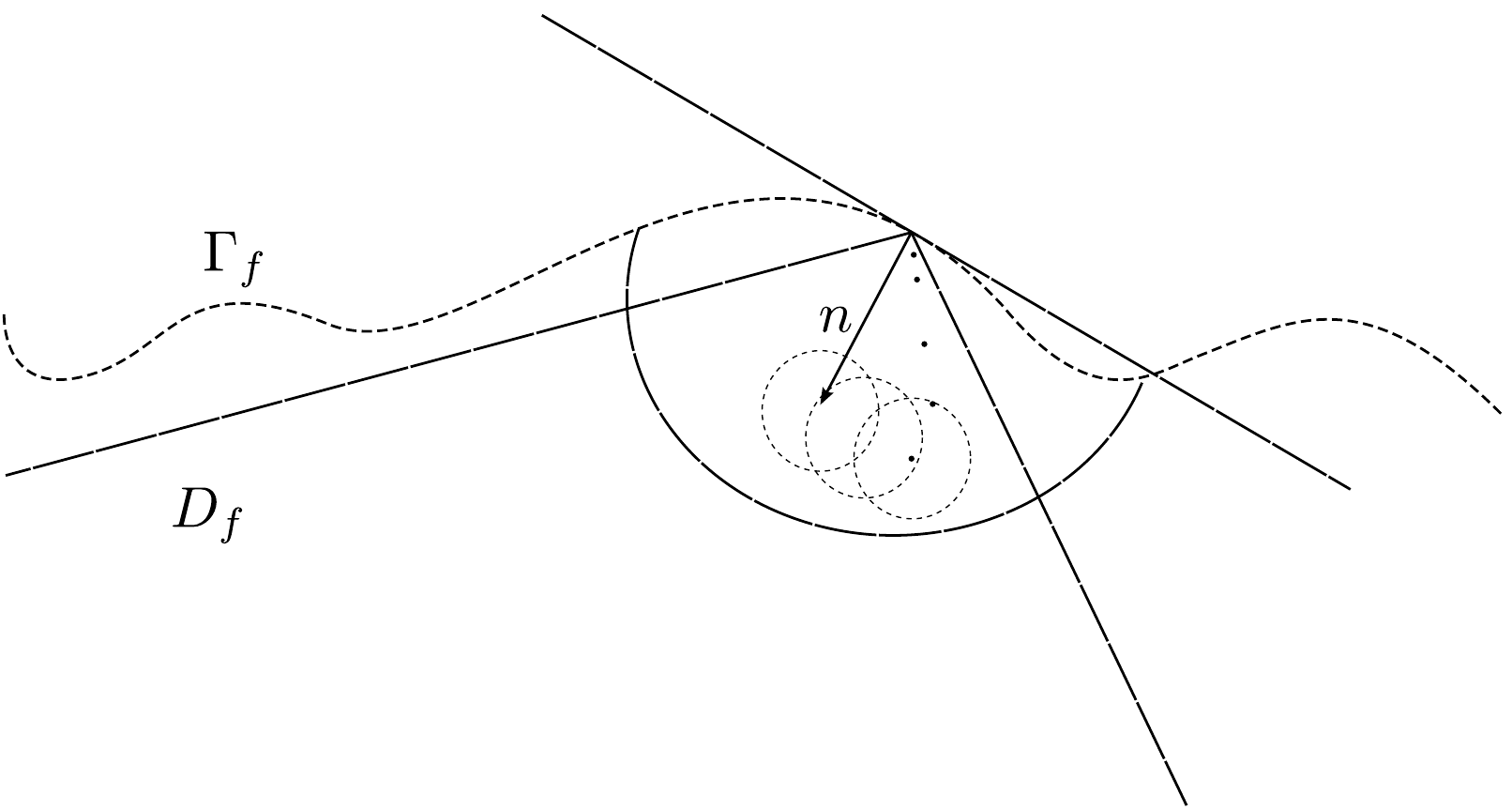}
  \end{center}
  Here we make use of the non-tangential convergence of the sequence: there is a $c_1>0$ such that
  \begin{align*}
    B_{c_1 r_m} (X_m) \subset B_{2^{-k_m}}(X_0).
  \end{align*}	
  Applying the Harnack inequality (e.g. a rescaled statement of \cite[Theorem 4.3]{CaCa-95} or \cite[Theorem 4.18]{Wang-1992ParabolicRegularity-one}, which includes the gradient term) it follows that (for some $\beta_0>0$) 
  \begin{align*}
    W_{k_m}(X) \geq \beta_0 r_m  \textnormal{ in } B_{\frac{1}{2}c_1r_m}(X_m) \cap D_w.
  \end{align*}  
  This together with $W_{k_m}\geq 0$ in $B_{\frac{1}{2}c_1r_m}(X_m) \cap D_w$, guarantees (via a standard barrier argument) that
  \begin{align*}
    W_{k_m}(X) \geq c_2 \beta_0  d(x,\partial D_w) \textnormal{ in } B_{2^{-k_m}}(X_0) \cap D_w,
  \end{align*}  
  for some constant $c_2 = c_2 (d,\lambda,\Lambda,\gamma,C)$. In terms of $U_f$, this says that 
  \begin{align*}
    U_f(X) \geq H_{\alpha_{k_m}}(X)+c_2 \beta_0 d(X,\partial D_w)  \textnormal{ in } B_{2^{-k_m}}(X_0) \cap D_w.
  \end{align*}  
  On the other hand, by Proposition \ref{propNO:ClassicalEvaluationLowerBarrier},
  \begin{align*}
    H_{\alpha}(X)-H_{\alpha_{k_m}}(X) \leq |\alpha-\alpha_{k_m}|d(X,\partial D_w). 
  \end{align*}  
  Therefore, choosing some $\alpha>\alpha_{k_m}$ sufficiently close to $\alpha_{k_m}$ (i.e. $|\alpha-\alpha_{k_m}| \leq c_2\beta_0$) and $\al>\al_*$, we have
  \begin{align*}
    U_f(X) \geq H_{\alpha}(X) \textnormal{ in } B_{2^{-k_m}}(X_0) \cap D_w,
  \end{align*}
  which is impossible because $\al_*$ was the supremum over all such $\al$. Therefore \eqref{eqnNO:ClassicalEvaluationLowerBarrierAsymptotic} holds and the Lemma is proved.
  
\end{proof}

The following is a useful corollary of Lemma \ref{lemNO:PointwiseEvaluation} that we state without a proof.

\begin{corollary}\label{corNO:PointwiseDefI}
	The value, $I(f,x_0)$, is well defined whenever $f$ is pointwise $m\text{-}C^{1,\gam}(x_0)$.
\end{corollary}


\subsection{The Lipschitz property}\label{section:sub Lipschitz}

In this section, we show that $I$ is a Lipschitz mapping on $\K^*(\gam,\del,m)$, which are convex subsets of $C^{0,1}$.  This Lipschitz property will have two very important consequences: a representation of $I$ via integro-differential operators; and a comparison theorem within $\K^*(\gam,\del,m)\intersect \K_*(\gam,\del,m)$ and hence for general viscosity solutions.  The main result is the following theorem and its corollary.  

\begin{theorem}\label{thmNO:LipschitzProperty}
  There is a constant $C=C(d,\lambda,\Lambda,\delta,\gam,m)$ such that if $f,g$ satisfy
  \begin{align*}
    f,g \in \mathcal{K}^*(\delta,\gam,m)\ \text{and}\ f-g\in C^{1,\gam}(\real^d),
  \end{align*}
  then for any $x\in\mathbb{R}^d$ at which $I(f,x)$ and $I(g,x)$ are both defined classically, we have
  \begin{align*}
    I(g,x)-I(f,x) \leq C \|f-g\|_{C^{1,\gam}(\mathbb{R}^d)}.
  \end{align*}	  
\end{theorem}

\begin{corollary}\label{corNO:OnePhaseLipschitzrSemiRegularity}
  If $f$ and $g$ satisfy $f \in \mathcal{K}_*(\delta,\gam,m)$ and $g\in \mathcal{K}^*(\delta,\gam,m)$ as well as $\phi \in C^{1,\gam}(\mathbb{R}^d)$ with $g+\phi\in\K^*(\delta,\gam,m)$, and all are such that
  \begin{align*}
    (f-g)-\phi \textnormal{ has a non-negative global maximum at } x_0 \in \mathbb{R}^d,
  \end{align*}	  
  then, with $C=C(d,\lambda,\Lambda,\delta,\gam,m)$, we have
  \begin{align*}
    I(f,x_0)-I(g,x_0) \leq C\|\phi\|_{C^{1,\gam}(\mathbb{R}^d)}.	  
  \end{align*}	  
  
  Similarly, as above, if $g-\phi\in\K^*(\del,\gam,m)$ and 
  \begin{align*}
    (g-f)-\phi \textnormal{ has a non-positive global minimum at } x_0 \in \mathbb{R}^d,
  \end{align*}
  then
  \begin{align*}
    I(f,x_0)-I(g,x_0) \leq C\|\phi\|_{C^{1,\gam}(\mathbb{R}^d)}.	  
  \end{align*}
\end{corollary}

The proofs of Theorem \ref{thmNO:LipschitzProperty} and Corollary \ref{corNO:OnePhaseLipschitzrSemiRegularity} will require some preliminary propositions and lemmas.  We will finish the proofs of the Theorem and Corollary at the end of this section.

In estimating $I(g,x)-I(f,x)$ at a given $x$, it will be necessary to reduce matters to the situation where $f=g$ at $x$. This is the purpose of the next result.
\begin{lemma}\label{lemNO:IDependenceOnConstants}
  There is a $C=C(d,\lambda,\Lambda,\delta,\gam,m)$ such that if $s\geq 0$ and $f$ is a function such that $f,f+s \in \mathcal{K}^*(\delta,\gam,m)$, then for all $x$ such that $I(f,x)$ is well defined,
  \begin{align*}
    I(f+s,x)\leq I(f,x) \leq I(f+s,x)+Cs,\;\;\forall\;x\in\mathbb{R}^d.	 
  \end{align*}	  	
\end{lemma}	

\begin{proof}
  Let us compare $U_{f}$ to a vertical shift of $U_{f+s}$.  We use the coordinates, $(x,x_{d+1})\in\real^d\times[0,\infty)$. Define $\tilde U: D_{f} \to \mathbb{R}$ by
  \begin{align*}
    \tilde U(x,x_{d+1}) := U_{f+s}(x,x_{d+1}+s).
  \end{align*}
  Observe that $\tilde U = U = 0$ on $\Gamma_{f}$ while $1-Cs\leq \tilde U \leq 1$ on $\Gamma_0$.  Furthermore, the translation invariance of (\ref{eqNO:BulkNonlinearDefI}) ensures that $\tilde U$ also solves (\ref{eqNO:BulkNonlinearDefI}) in a translated domain. Thus, by the comparison principle
  \begin{align*}
    \tilde U \leq U_f \textnormal{ and } U_f-\tilde U \leq Cs \;\textnormal{ in } D_f.
  \end{align*}	  
  The first inequality implies that
  \begin{align*}
    |\nabla \tilde U| \leq |\nabla U_f| \textnormal{ on } \Gamma_f,
  \end{align*}
  while the second inequality, together with the same upper gradient estimate at the boundary used in Proposition \ref{propNO:SemiconcavityImpliesUfLipschitzAndBoundsOnNormalDeriv} (which uses that $f \in \mathcal{K}^*$), implies that
  \begin{align*}
    |\nabla (U_f-\tilde U)| \leq Cs \textnormal{ on } \Gamma_f.
  \end{align*}
  In conclusion, we have
  \begin{align*}
    |\nabla \tilde U| \leq |\nabla U_f| \leq |\nabla \tilde U|+Cs \textnormal{ on } \Gamma_{f}.
  \end{align*}
  Noting that $\nabla \tilde U(x,f(x)) = \nabla U_{f+s}(x,f(x)+s)$, we conclude that
  \begin{align*}
    I(f+s,x) \leq I(f,x) \leq I(f+s,x) + Cs.
  \end{align*}
\end{proof}

As we saw in Proposition \ref{propNO:fToUfisLipschitz}, the function $U_f$ depends in a Lipschitz manner on the function $f$, as long as $f$ lies in $\mathcal{K}^*(\delta,\gam,m)$. The next Lemma will produce a barrier that will allow us to translate pointwise control of $U_f$ in $D_f$ into pointwise control of $\nabla U_f$ along $\Gamma_f$, this will be the key step in the proof of Theorem \ref{thmNO:LipschitzProperty}.

\begin{lemma}\label{lemNO:PointwiseBoundtoNormalDerivativeBoundForExtremalEqs}
  Fix $w\in C^{1,\gam}(\real^d)$, $X_0 \in \Gamma_w$, and $r_0 \in (0,\delta/2)$. Let $W:D_w \cap B_{r_0}(X_0)\to \mathbb{R}$ be a continuous function such that in the viscosity sense, 
  \begin{align*}
    & \mathcal{M}^+(D^2W) + \Lam\abs{\grad W} \geq -C_1 \textnormal{ in } D_w \cap B_{r_0}(X_0),\\
    & \mathcal{M}^-(D^2W) - \Lam\abs{\grad W} \leq C_1 \textnormal{ in } D_w \cap B_{r_0}(X_0),\\	
    & |W| \leq C_1 |X-X_0|^{1+\gam} \textnormal{ on } \Gam_w \cap B_{r_0}(X_0).
  \end{align*}
  Then, with a constant $C=C(d,\lambda,\Lambda,\gam,\norm{w}_{C^{1,\gam}(B_{r_0}(X_0))},r_0)$ we have
  \begin{align*}
    |\nabla W(X_0)| \leq C\cdot C_1.  
  \end{align*}	 
\end{lemma}

\begin{proof}
  With $\rho(X) = |X-X_0|^{1+\gam}$, let $\hat W^+$ solve
  \begin{align*}
	  \begin{cases}
    \mathcal{M}^+(\hat W^+) + \Lam\abs{\grad \hat W^+}  = -1 &\textnormal{in}\ D_w \cap B_{r_0}(X_0),\\
    \hat W^+  = \rho &\textnormal{on}\ \partial (D_w \cap B_{r_0}(X_0)).
	\end{cases}
  \end{align*}
  We know that \cite[Remark 3.2]{Trudinger-1988HolderGradientFullyNonlinearPRO-ROY-SOC-ED} shows that $\hat W^+ \in C^{1,\gam'}(\overline{D_w \cap B_{r_0/2}(X_0)})$; and in particular, that
  \begin{align*}
    |\nabla \hat W^+(X_0)| \leq C. 
  \end{align*}
  From the assumption, we have that $W \leq C_1 \hat W^+$ on $\partial (D_w \cap B_{r_0}(X_0))$ and in the viscosity sense,
  \begin{align*}
    \Lam\abs{\grad( C_1\hat W)}+\mathcal{M}^+(D^2(C_1 \hat W^+)) = -C_1 \leq \mathcal{M}^+(D^2W) + \Lam\abs{\grad W} \textnormal{ in } D_w \cap B_{r_0}(X_0).
  \end{align*}
  Thus, the comparison principle says that
  \begin{align*}
    W \leq \hat W^+ \textnormal{ in } D_w \cap B_{r_0}(X_0).
  \end{align*}
We can repeat this argument for the function $\hat W^-$, solving
	\begin{align*}
		\begin{cases}
		M^-(\hat W^-)-\Lam\abs{\grad \hat W^-}=1\ &\text{in}\ D_w \cap B_{r_0}(X_0),\\
		\hat W^-  = -\rho &\textnormal{on}\ \partial (D_w \cap B_{r_0}(X_0)),
		\end{cases}
	\end{align*}
	which can be obtained by setting $\hat W^-=-\hat W^+$.

  Since we also have $W(X_0) = \hat W^\pm(X_0) = 0$, it follows that $|\nabla W(X_0)| \leq |\nabla \hat W^+(X_0)|$, thus
  \begin{align*}
    |\nabla W(X_0)| \leq C\cdot C_1.
  \end{align*}
\end{proof}

The next two lemmas are concerned with how much the operator, $I$, can deviate when it is evaluated on two domains that touch at a point with $C^1$ contact.  Originally, it was thought that Lemma \ref{lemNO:EstimateAtC1ContactOfDomains-PART1} (below) would be sufficient, but this subsequently turned out not to be the case.  We have decided to leave it here because we believe it could be useful for future investigations.  

\begin{lemma}\label{lemNO:EstimateAtC1ContactOfDomains-PART1}
  Let $f,g$ be such that $f,g \in \mathcal{K}^*(\delta,\gam,m)\intersect (m\text{-}C^{1,\gam}(x_0))$, and $f-g \in C^{1,\gam}(\mathbb{R}^d)$.  Assume that $x_0 \in \mathbb{R}^d$ be a point such that $f(x_0)=g(x_0)$ and $\nabla f(x_0) = \nabla g(x_0)$.  In this case, there exists a $C=C(d,\lambda,\Lambda,\delta,\gam,m)$ such that 
\begin{align*}
	\abs{I(f,x_0)-I(g,x_0)}\leq C\norm{f-g}_{C^{1,\gam}(\real^d)}.
\end{align*}

\end{lemma}

\begin{proof}[Proof of Lemma \ref{lemNO:EstimateAtC1ContactOfDomains-PART1}]
	The assumptions of the lemma allow us to assert that there is a choice of $r>0$ and $c>0$ so that 
	\begin{align*}
		\forall\ x\in B_r(x_0),\ \ \ f(x)\geq f(x_0) + \grad f(x_0)\cdot (x-x_0)-c\abs{x-x_0}^{1+\gam},
	\end{align*}
	and similarly for $g$.  (Recall, we use the convention that $X\in\real^{d+1}$ whereas $x\in\real^d$.)  We will denote this function on the right hand side of the inequality as
	\begin{align*}
		\forall\ x\in\ B_r(x_0),\ \ w(x)=f(x_0) + \grad f(x_0)\cdot (x-x_0)-c\abs{x-x_0}^{1+\gam},
	\end{align*} 
	and we can assume that $w$ has been extended to all of $\real^d$ so that
	\begin{align*}
		\norm{w}_{C^{1,\gam}(\real^d)}\leq 2c+m
	\end{align*}
	and we will assume that $r$ is also small enough so that 
	\begin{align*}
		\forall\ x\in B_r(x_0),\ \ \ w(x)\geq \del.
	\end{align*}
	For notational convenience, let us call
	\begin{align*}
		X_0=(x_0,f(x_0))=(x_0,g(x_0))\in\Gam_f\intersect\Gam_g.
	\end{align*}
	Thus, by this choice of $w$, we have,
	\begin{align*}
		D:= D_w\intersect B_r(X_0)\subset D_f\ \ \text{and}\ \ D\subset D_g.
	\end{align*}

 Recalling Proposition \ref{propNO:fToUfisLipschitz}, it follows that 
  \begin{align*}
    \|U_g-U_f\|_{L^\infty(D)} \leq C\|f-g\|_{L^\infty(\mathbb{R}^d)}.
  \end{align*}	  
  Thanks to the globally Lipschitz nature of $U_f$, $U_g$, and the definition of $w$, along the top boundary of $\partial D_w$, this pointwise estimate can be improved to
  \begin{align*}
    |U_g(X)- U_f(X)| \leq C\|w\|_{C^{1,\gam}(\mathbb{R}^d)}|X-X_0|^{1+\gam},\ \textnormal{for}\ X\in \Gamma_w \cap B_{r}(X_0).
  \end{align*}
  (We note this is a consequence of the zero boundary values of $U_f$, $U_g$ on $\Gam_f$, $\Gam_g$, and that both top boundaries separate from $\Gam_w$ at a rate of $C\abs{X-X_0}^{1+\gam}$.)
  
Next, we can define the function on $D$, $\tilde U=U_f-U_g$.  We see that $\tilde U$ is a viscosity solution of
	\begin{align*}
		\M^+(D^2 \tilde U)+\Lam\abs{\grad\tilde U}\geq 0\ \ \text{and}\ \ 
		\M^-(D^2\tilde U)-\Lam\abs{\grad\tilde U}\leq 0.
	\end{align*}
	Therefore, we may apply Lemma \ref{lemNO:PointwiseBoundtoNormalDerivativeBoundForExtremalEqs}, with $C_1= C\norm{w}_{C^{1,\gam}}$ and noting that $f-g\in C^{1,\gam}$, in order to deduce that 
	\begin{align*}
		\abs{\grad \tilde U(X_0)}\leq C(\norm{U_f-U_g}_{L^\infty}+C\norm{w}_{C^{1,\gam}}).
	\end{align*}
	Hence, by collecting the previous estimates on $\norm{U_f-U_g}_{L^\infty}+C\norm{w}_{C^{1,\gam}}$, we see that
	\begin{align*}
		\abs{\grad U_f(X_0)-\grad U_g(X_0)}\leq C\norm{f-g}_{C^{1,\gam}}.
	\end{align*}
\end{proof}

This next lemma is similar to the previous, except that we adapt the assumptions to allow that the equations for the respective solutions need not be the same for $U_f$, $U_g$.

\begin{lemma}\label{lemNO:NormalDerivsAtC1ContactPointGeneralEqs}
Let $f,g$ be such that $f,g \in \mathcal{K}^*(\delta,\gam,m)\intersect (m\text{-}C^{1,\gam}(x_0))$, and $f-g \in C^{1,\gam}(\mathbb{R}^d)$, and let $x_0 \in \mathbb{R}^d$ be a point such that $f(x_0)=g(x_0)$ and $\nabla f(x_0) = \nabla g(x_0)$.  Further assume the existence $r>0$ and of functions, $V_f$ and $V_g$ that satisfy in the viscosity sense, for $X_0=(x_0,f(x_0))=(x_0,g(x_0))$,
	\begin{align*}
		\text{in}\ D_f\intersect B_r(X_0),\ \abs{F(D^2V_f,\grad V_f)}\leq C_1\\
		\text{in}\ D_g\intersect B_r(X_0),\ \abs{F(D^2V_g,\grad V_g)}\leq C_1.
	\end{align*}
Then, there exists a $C=C(d,\lambda,\Lambda,\delta,\gam,m,r)$ such that 
	\begin{align*}
		\abs{\grad V_f(X_0)- \grad V_g(X_0)}\leq C(\norm{f-g}_{C^{1,\gam}(\real^d)}+C_1).
	\end{align*}
\end{lemma}

\begin{proof}[Proof of Lemma \ref{lemNO:NormalDerivsAtC1ContactPointGeneralEqs}]
	The proof of Lemma \ref{lemNO:NormalDerivsAtC1ContactPointGeneralEqs} follows nearly identically to that of Lemma \ref{lemNO:EstimateAtC1ContactOfDomains-PART1}.  We simply note that the auxiliary function, $\tilde U$ will be defined as 
	\begin{align*}
		\tilde U= V_f-V_g.
	\end{align*}
	Because of the assumptions on $f$, $g$, $V_f$, $V_g$, we see that $\tilde U$ satisfies in the viscosity sense, 
	\begin{align*}
		\M^+(D^2 \tilde U)+\Lam\abs{\grad\tilde U}\geq -2C_1\ \ \text{and}\ \ 
		\M^-(D^2\tilde U)-\Lam\abs{\grad\tilde U}\leq 2C_1.
	\end{align*}
	The rest of the argument is identical, as this situation also falls within the scope of Lemma \ref{lemNO:PointwiseBoundtoNormalDerivativeBoundForExtremalEqs}.
\end{proof}

The next lemma, will be used in the proof of Theorem \ref{thmNO:LipschitzProperty}; it is a simple consequence of the definition of viscosity solutions, the uniform ellipticity of $F$, and the fact that $F$ is assumed to be rotationally invariant in the $D^2U$ argument.  We omit the proof.

\begin{lemma}\label{lemNO:EqForTheRotationOfViscSol}
	If $F$ satisfies the assumption of rotational invariance in Section \ref{sec:Assumption}-(c), $\Om$ is an open set, $\R$ is a rotation, $U\in C^{0,1}(\Om)$ is a viscosity solution of 
	\begin{align*}
		F(D^2U,\grad U)=0\ \text{in}\ \Om,
	\end{align*}
	then for $W=U\circ \R$, $W$ solves in the viscosity sense
	\begin{align*}
		\abs{F(D^2W,\grad W)}\leq C\Lam\abs{\R-\Id}\cdot\norm{\grad U}_{L^\infty},\ \text{in}\ 
		\R^{-1}\Om.
	\end{align*}
\end{lemma}

We are finally ready to prove Theorem \ref{thmNO:LipschitzProperty}.

\begin{proof}[Proof of Theorem \ref{thmNO:LipschitzProperty}]

  Let $x_0$ be an arbitrary point in $\real^d$, let $f,g$ be such that $f,g \in \mathcal{K}^*(\delta,\gam,m)$ and $f-g \in C^{1,\gam}(\mathbb{R}^d)$, we proceed in four steps. \\

\noindent \emph{Step 0.} ($f-g$ is small.)  We note that it will suffice to prove the result under the assumption that $\norm{f-g}_{L^\infty(\real^d)}\leq \frac{\del}{2}$.  Otherwise, we can simply use the fact that the constant, $C$, naturally depends upon $\del$, hence by the definition of $I$, (note, $I\geq0$)
\begin{align*}
	I(g,x)-I(f,x)\leq I(g,x)+I(f,x)\leq \frac{2 \tilde C}{(\del/2)}(\del/2)\leq \frac{2 \tilde C}{(\del/2)}\norm{f-g}_{L^\infty}\leq \frac{2 \tilde C}{(\del/2)}\norm{f-g}_{C^{1,\gam}}.
\end{align*}
Thus, from here on, we will assume that $\norm{f-g}_{L^\infty(\real^d)}\leq \frac{\del}{2}$.

  \noindent \emph{Step 1.} (Reduction to $f(x_0)=g(x_0)$.)

  Consider the function $\hat f(x) = f(x) + (g(x_0)-f(x_0))$.  By step 0, we know $\hat f\in K^*(\gam,\del/2,m)$. By Lemma  \ref{lemNO:IDependenceOnConstants}, we have
  \begin{align*}
    |I(f,x_0)-I(\hat f,x_0)| \leq C|f(x_0)-g(x_0)| \leq C\|f-g\|_{L^\infty(\mathbb{R}^d)}.
  \end{align*}
  Therefore, we have that $\hat f(x_0) = g(x_0)$ and  
  \begin{align*}
    I(g,x_0)-I(f,x_0) \leq C\|f-g\|_{L^\infty(\mathbb{R}^d)} + \|I(g,x_0)-I(\hat f,x_0)\|.
  \end{align*}
  From this observation, we conclude that to prove the theorem it is sufficient to show that 
  \begin{align}\label{eqnNO:LipschitzAtPointfAndgAreEqual}
    I(g,x_0)-I(f,x_0) \leq C\|f-g\|_{C^{1,\gam}(\mathbb{R}^d)} \textnormal{ at } x_0 \textnormal{ s.t. } f(x_0)=g(x_0). 
  \end{align}
  Accordingly, for the rest of the proof we assume that $x_0\in\mathbb{R}^d$ is such that $f(x_0)=g(x_0)$, and focus on proving \eqref{eqnNO:LipschitzAtPointfAndgAreEqual}.\\

  \noindent \emph{Step 2.} (Rotation to achieve $\grad f(x_0)=\grad g(x_0)$.)

  Let us denote by $n_f$ and $n_g$ respectively the outer normals to $\Gamma_f$ and $\Gamma_g$ at the point $(x_0,f(x_0))$. Let $\R$ denote the rotation in $\mathbb{R}^{d+1}$ given by 
  \begin{align*}
    \R n_{f} =  n_{g},\ \text{and}\  \R v = v \textnormal{ if } v \perp \text{span}\{n_f,n_g\}.
  \end{align*}
  We note that $\R$ is a rotation that sends the tangent plane to $\Gamma_f$ at $X_0$ to the tangent plane of $\Gamma_g$ at $X_0$. Furthermore, by definition, $\abs{\R-\Id}\leq C \abs{n_f-n_g}$, we see that
  \begin{align}\label{eqNO:ProofOfMainThmRminusIdEst}
    |\R-\Id|\leq C \abs{n_f-n_g} \leq C|\nabla f(x_0)-\nabla g(x_0)| \leq C\|\nabla f-\nabla g\|_{L^\infty}.
  \end{align}
  Then, we denote by $\T$ the isometry $\T X := \R(X-X_0)+X_0$. We note that we can choose $r_0$ small enough so that we can choose a function,  $w$,  that satisfies
  \begin{align*}
    B_{r_0}(X_0) \cap  \T(D_f) = B_{r_0}(X_0) \cap D_w,
  \end{align*}
  and we can still ensure $w\in K^*(\gam,\frac{1}{2}\del, 2 m)\intersect(m\text{-}C^{1,\gam}(x_0))$.
  We then define a function $\tilde U$ by
  \begin{align*}
    \tilde U(X) = U_f(\T^{-1}X)\;\;\forall\;x \in B_{r_0}(X_0) \cap D_w.
  \end{align*}
  Now we have the function $\tilde U: D_w \cap B_r(X_0) \to \mathbb{R}$ given by $\tilde U(x) = U_f(\T^{-1}x)$. Observe that
  \begin{align*}
    \nabla \tilde U(X_0) = \R\nabla U_f(X_0).
  \end{align*}
   Since $|\nabla U_f(X_0)|\leq C(d,\lambda,\Lambda,\delta,M)$, it follows that
  \begin{align*}
    |\nabla \tilde U(X_0) - \nabla U_f(X_0)| \leq | (\R-\Id)\nabla U_f(X_0)| \leq C|\R-\Id|,
  \end{align*}
  and we conclude, per (\ref{eqNO:ProofOfMainThmRminusIdEst}), that
  \begin{align}\label{eqn:NewOperator Lipschitz rotation estimate}
    |\nabla \tilde U(X_0) - \nabla U_f(X_0)| \leq C|\grad f(x_0)-\grad g(x_0)|.
  \end{align}
 
  \bigskip
  
  \noindent \emph{Step 3.}  We first confirm the equation satisfied by $\tilde U$ in $B_{r_0}(X_0)\intersect D_w$, and then we will invoke Lemma \ref{lemNO:NormalDerivsAtC1ContactPointGeneralEqs}.  Indeed, since $\grad U_f\in L^{\infty}(D_f)$ (via, e.g. Proposition \ref{propNO:SemiconcavityImpliesUfLipschitzAndBoundsOnNormalDeriv}), we see, by the rotational invariance assumed for $F$ and Lemma \ref{lemNO:EqForTheRotationOfViscSol}, that in the viscosity sense, by Lemma \ref{lemNO:EqForTheRotationOfViscSol},
\begin{align*}
	\abs{F(D^2\tilde U,\grad \tilde U)}\leq C\Lam\abs{\R-\Id}\cdot\norm{\grad U_f}_{L^\infty},\ \text{in}\ B_{r_0}(X_0)\intersect D_w.
\end{align*}
 This shows that both $\tilde U$ and $U_g$ satisfy the assumptions of Lemma \ref{lemNO:NormalDerivsAtC1ContactPointGeneralEqs} (applied to the respective domains of $D_w$ and $D_g$), and so we can conclude that
\begin{align*}
	\abs{\grad\tilde U(X_0)-\grad U_g(X_0)}\leq C(\norm{w-g}_{C^{1,\gam}}+\tilde C).
\end{align*}
Here, we have used that
\begin{align*}
	\tilde C = C\Lam\abs{\R-\Id}\cdot\norm{\grad U}_{L^\infty}.
\end{align*}

Collecting all of the above estimates, we see that under the assumption that $f(X_0)=g(X_0)$ and $\grad f(X_0)=\grad g(X_0)$, the estimate holds,
\begin{align*}
	\abs{\grad U_f(X_0)-\grad U_g(X_0)}\leq C\norm{f-g}_{C^{1,\gam}},
\end{align*}
where we note, by construction of the function, $w$, that we can choose $w$ so that $\norm{w-g}_{C^{1,\gam}}\leq \norm{f-g}_{C^{1,\gam}}$
Hence, the conclusion of the theorem under the step 3 assumption holds.  Hence, combining with steps 1 and 2 give the original result.

\end{proof}

\begin{proof}[Proof of Corollary \ref{corNO:OnePhaseLipschitzrSemiRegularity}]
	First, we give the argument for the case when $f-g-\phi$ attains a maximum.
  Indeed, we have for $s=f(x_0)-g(x_0)-\phi(x_0) \geq 0$ that 
  \begin{align*}
    f -\phi - s & \leq g \textnormal{ in } \mathbb{R}^d,\\
    f -\phi - s & = g \textnormal{ at } x_0.
  \end{align*}
  First, we note that by the assumptions that $f\in\K_*(\gam,\del,m)$, $g\in\K^*(\gam,\del,m)$, and $\phi\in C^{1,\gam}$, this inequality shows that in fact $g,f\in m\text{-}C^{1,\gam}(x_0)$.  Hence, by Lemma \ref{lemNO:PointwiseEvaluation}, we know that $I(f,x_0)$ and $I(g,x_0)$ are defined classically.

Then, the GCP says that 
\begin{align*}
	I(f,x_0) \leq I(g+\phi+s,x_0).
\end{align*}
Also since $s\geq 0$ and $g+\phi\in\K^*(\gam,\del,m)$, Lemma \ref{lemNO:IDependenceOnConstants} yields
  \begin{align*}
    I(g+\phi+s,x_0) \leq I(g+\phi,x_0).   
  \end{align*}
  Furthermore, by Theorem \ref{thmNO:LipschitzProperty} we know that
  \begin{align*}
   I(g+\phi,x_0)-I(g,x_0) \leq C\|\phi\|_{C^{1,\gam}(\mathbb{R}^d)}.
  \end{align*}
Hence, putting this all together,
\begin{align*}
	I(f,x_0)-I(g,x_0)&\leq  I(g+\phi+s,x_0)-I(g,x_0)\leq I(g+\phi,x_0)-I(g,x_0)\leq  C\norm{\phi}_{C^{1,\gam}(\real^d)}.
\end{align*}

The second part of the result, when $g-f-\phi$ attains a minimum is analogous.  In this case, $s\leq 0$, and the GCP says that
\begin{align*}
	I(f,x_0)\leq I(g-\phi-s,x_0).
\end{align*}
Lemma \ref{lemNO:IDependenceOnConstants} then shows that
\begin{align*}
	I(g-\phi-s,x_0)\leq I(g-\phi,x_0).
\end{align*}
The rest follows the same steps, invoking Theorem \ref{thmNO:LipschitzProperty}.

\end{proof}

The next Lemma quantifies how the dependence of $I(f,x)$ on values of $f$ far away from $x$ decays as the distance to $x$ grows (see also  \cite[Proposition 3.8]{GuSc-2014NeumannHomogPart1DCDS-A}).  It culminates in Lemma \ref{lemNO:NormSplittingLemma}, which is a requirement for producing a min-max result of the form promised in Theorem \ref{thm:MetaILipAndMinMax}, and it appears in Section \ref{sec:MinMax} as Assumption \ref{assumptionMM}-(vii).
\begin{lemma}\label{lemNO:DecayWhenAgreeInBr}
  There is a $C=C(d,\lambda,\Lambda,\delta,\gam,m,L)$, such that given $f,g \in \mathcal{K}^*(\delta,\gam,m)\intersect m\text{-}C^{1,\gam}(x_0)$, with $f,g \leq L$, $x_0 \in \mathbb{R}^d$, and $R>1$ all such that
  \begin{align*}
    f \equiv g \textnormal{ in } B_{R}(x_0),
  \end{align*}
  and if a.e. $B_R(x_0)$,  then
  we have
  \begin{align*}
    |I(f,x_0)-I(g,x_0)| \leq CR^{-2}\|f-g\|_{L^\infty(\mathbb{R}^d)} + CR^{-1-\gam}.	  
  \end{align*}	  

\end{lemma}

\begin{proof}[Proof of Lemma \ref{lemNO:DecayWhenAgreeInBr}]
	
	First, we note that the assumption that $f,g\in m\text{-}C^{1,\gam}(x_0)$ ensures that $I(f,x_0)$ and $I(g,x_0)$ are well defined, via Corollary \ref{corNO:PointwiseDefI}.

	We will proceed in two steps.  The first step is the result under the additional assumption that the maximum of $f$ over $B_R(x_0)$ occurs at $x_0$.  The second step is to use the $\mathcal K^*$ upper bound for $f$, with Theorem \ref{thmNO:LipschitzProperty}, to reduce to the first step, incurring an extra error.
	
We will only show half of the inequality in that
\begin{align}\label{eqNO:DecayLemmaUpperGoal}
	I(f,x_0)-I(g,x_0) \leq CR^{-2}\|f-g\|_{L^\infty(\mathbb{R}^d)} + CR^{-1-\gam}.
\end{align} 
The reverse inequality follows by constructing an appropriate subsolution, where below we treat the case of a supersolution.  The modifications are standard.

	\emph{Step 1} (Assume $f(x_0)$ is the maximum of $f$.)

	Let us assume, without loss of generality, that $x_0 = 0$ and for $\abs{x}\leq R$, $f(x)\leq f(0)$. Let us consider then, the set
  \begin{align*}
    T := \{ (x,x_{d+1}) \in \mathbb{R}^{d+1} \mid 0\leq x_{d+1}\leq f(x),\; x \in B_{R/2}(0)\}.
  \end{align*}
We will construct a barrier that traps $f$ on the top boundary, at $(0,f(0))$.
  For a function $\phi_0(y)$ ($y\in [0,L]$) to be determined, let 
  \begin{align*}
    \psi_{R}(x,x_{d+1}) = \tilde \psi(x/R) + R^{-2}\phi_0(x_{d+1}),\ \ \tilde \psi(x) := |x|^2(1+|x|^2)^{-1}.
  \end{align*}
  A straightforward computation shows that
  \begin{align*}
    \nabla \psi_R(x,x_{d+1}) = ( R^{-1}(\nabla \tilde \psi)(x/R),R^{-2}\phi_0'(x_{d+1}))
  \end{align*}
  and
  \begin{align*}
    D^2 \psi_R(x,x_{d+1}) = \left (\begin{array}{cc}
	R^{-2}(D^2 \tilde \psi)(x/R) & 0\\
	0 & R^{-2}\phi_0''(x_{d+1})
	\end{array} \right ).
  \end{align*}
  Then, if $\phi_0(x_{d+1})''\leq 0$ and $\phi_0'(x_{d+1})\leq 0$ for all $x_{d+1}\in [0,L]$, we have
  \begin{align*}
    \mathcal{M}^+(D^2\psi_{R}) +\Lambda |\nabla \psi_R|  \leq \Lambda R^{-2}\|D^2\tilde \psi\|_\infty + \lambda R^{-2} \phi_0''(x_{d+1}) + \Lambda R^{-2}\|\nabla \tilde \psi\|_\infty - \Lambda R^{-2}\phi_0'(x_{d+1}).
  \end{align*}
  Therefore, 
  \begin{align}\label{eqNO:DecayLemmaSuperSolBarrierGoal}
    \mathcal{M}^+(D^2\psi_{R}) +\Lambda |\nabla \psi_R|  \leq R^{-2} \left \{  \lambda  \phi_0''(x_{d+1}) - \Lambda \phi_0'(x_{d+1})+\Lambda \|\tilde \psi\|_{C^2} \right \}.
  \end{align}
  With the goal in mind that is to force this inequality to be non-positive (i.e. $\psi_R$ should be a super solution), we can now choose $\phi_0$. Let us take, for some $M,b,c>0$, and restricting $y\in[0,c]$,
  \begin{align*}
    \phi_0(y) = M(1-e^{-b(c-y)}).
  \end{align*}
  Then, for every $y\geq 0$ we have
  \begin{align*}
    \phi_0'(y) = -bMe^{-b(c-y)} \leq 0,\ \ \phi_0''(y) = -b^2 Me^{-b(c-y)} \leq 0,
  \end{align*}  
  and thus
  \begin{align*}
    \lambda  \phi_0''(y) - \Lambda \phi_0'(y)+\Lambda \|\tilde \psi\|_{C^2}  = (-\lambda b^2 +\Lambda b)Me^{-b(c-y)} + \Lambda \|\tilde \psi\|_{C^2}.
  \end{align*}
  We first choose $b>0$ to be large enough so that 
\begin{align*}
	-\lambda b^2 +\Lambda b\leq -1.
\end{align*}
Furthermore, we see that for $y\in[0,c]$,  and $c\leq L$
\begin{align*}
	e^{-b(c-y)}\geq e^{-bc}\geq e^{-bL},
\end{align*}
thus with our choice of $b$,
\begin{align*}
	(-\lambda b^2 +\Lambda b)Me^{-b(c-y)}\leq (-\lambda b^2 +\Lambda b)Me^{-bL}.
\end{align*}
Now, we can choose $M$ large enough, depending only on $\lam$, $\Lam$, $b$, $L$, so that
\begin{align*}
	(-\lambda b^2 +\Lambda b)Me^{-bL}\leq -\Lambda \|\tilde \psi\|_{C^2}.
\end{align*}
Finally, we set $c=f(0)$.
Hence, we have attained 
   \begin{align*}
     \mathcal{M}^+(D^2\psi_{R}) +\Lambda |\nabla \psi_R| \leq 0 \textnormal{ in } D_L,
   \end{align*}
   and we recall that by our assumption that $f=g$ in $B_R$, $D_f\intersect T=D_g\intersect T$, and that $D_f\subset D_L$.
    
We gather the following properties of $\psi_{R}$
   \begin{align*}
     &\psi_{R} \geq \tilde \psi(x/R) \textnormal{ on } \Gamma_{f},\\
     &\psi_{R} \geq c_0>0 \textnormal{ on } \{y=0\},\\
	 &\psi_{R} \geq c_0>0 \textnormal{ on } \{(x,y)\ :\ \abs{x}=R/2\}.
   \end{align*}
We further recall that by Proposition \ref{propNO:fToUfisLipschitz} we have 
\begin{align*}
	\norm{U_f-U_g}_{L^\infty(D_f\intersect D_g)}\leq C\norm{f-g}_{L^{\infty}(\real^d)}.
\end{align*}
Thus, this means that for a universal choice of $C$, we can show that 
\begin{align*}
	\text{on}\ \partial T,\ \ U_f-U_g\leq C\norm{f-g}_{L^\infty}\psi_R.
\end{align*}
Since $U_f$ and $U_g$ are viscosity solutions of the same equation, we see that $W=U_f-U_g$ is a viscosity subsolution of 
\begin{align*}
	\M^+(D^2W)+\Lam\abs{\grad W}\geq0.
\end{align*}
Using the comparison theorem in $T$ with the functions $W$ and $C\norm{f-g}_{L^\infty}\psi_R$, hence
\begin{align*}
	\forall\ X\in T,\ U_f(X)-U_g(X)\leq C\norm{f-g}_{L^\infty}\psi_R(X).
\end{align*}
By construction, we have that $U_f(0,f(0))=U_g(0,f(0))=\psi_R(0,f(0))=0$, and so
\begin{align*}
	\partial_nU_f(0,f(0))-\partial_nU_g(0,f(0))=\partial_n W(0,f(0))\leq \partial_n \psi_R(0,f(0))\leq CR^{-2}\norm{f-g}_{L^\infty(\real^d)},
\end{align*}
which gives the desired result in (\ref{eqNO:DecayLemmaUpperGoal}).  This completes the proof in the setting of Step 1.

\emph{Step 2} (reducing to $x_0$ is a max of $f$ in $B_R$).

Given that we are assuming $f\in K^*(\gam,\del,m)$, we know that for an appropriate choice of $c$, depending upon $m$,
\begin{align*}
	\forall\ x\in B_R(x_0),\ \ f(x)\leq f(x_0) + Cm\abs{\frac{x-x_0}{R}}^{1+\gam}.
\end{align*}
Hence, replacing both $f$ and $g$ by 
\begin{align*}
	\tilde f(x)=f(x)-Cm\abs{\frac{x-x_0}{R}}^{1+\gam}
	\ \ \text{and}\ \ \tilde g(x)=g(x)-Cm\abs{\frac{x-x_0}{R}}^{1+\gam},
\end{align*} 
we see that we satisfy the assumptions of Step 1.

Furthermore, we know, from Theorem \ref{thmNO:LipschitzProperty}, after extending the function $Cm\abs{\frac{x-x_0}{R}}^{1+\gam}$ to all of $\real^d$ in a way that does not increase its $C^{1,\gam}$ norm by a factor of more than, $2$, that we have
\begin{align*}
	\abs{I(f,x_0)-I(\tilde f,x_0)}\leq CmLR^{-1-\gam}.
\end{align*}
This, combined with the result in Step 1, shows the desired estimate.

\end{proof}

\begin{rem} It is worth comparing this decay rate with what one observes for the Dirichlet-to-Neumann map in the half-space for the Laplacian.  Of course, one obtains as the D-to-N for the Harmonic extension the half Laplacian, $-(-\Delta)^{1/2}f$.  Because the harmonic extension satisfies an equation in half-space, a simple rescaling argument can be shown that in the context of the previous Lemma (with $f\equiv g$ in $B_R(x_0)$), 
    \begin{align*}
      |(-\Delta)^{\frac{1}{2}}(f-g)(x_0)| \leq CR^{-1}\|f-g\|_{L^\infty}.
    \end{align*}
This is, of course, different from what we have obtained above.  The discrepancy can be attributed to two factors: we work in a finite width domain, with a Dirichlet condition at $y=0$; and our equation does not rescale in a positively homogeneous way (different powers from the Hessian and gradient).  Thus, as seen from our barrier, $\psi_R$, the term $R^{-2}$ is not surprising.  The way in which this estimate is invoked is completely unaffected by the actual decay on the right hand side of the estimate. 
\end{rem}

Next, we remove the assumption from the previous lemma that required $f\equiv g$ in $B_R(x_0)$.

\begin{lemma}\label{lemNO:NormSplittingLemma}
  If $C=C(d,\lambda,\Lambda,\delta,\gam,m,L)$ is as in Lemma \ref{lemNO:DecayWhenAgreeInBr}, that $f,g \in \mathcal{K}^*(\delta,\gam,m)$ and $f-g\in C^{1,\gam}(B_{2R}(x_0))$, with $f,g\leq L$, $x_0 \in \mathbb{R}^d$, and $R>1$, then
  \begin{align*}
    |I(f,x_0)-I(g,x_0)| \leq C\|f-g\|_{C^{1,\gam}(B_{2R}(x_0))} + CR^{-2}\|f-g\|_{L^\infty(\mathbb{R}^d)} + CR^{-1-\gam}.
  \end{align*}	  

\end{lemma}

\begin{proof}[Proof of Lemma \ref{lemNO:NormSplittingLemma}]
  Let $\psi$ be a smooth function such that $0\leq \psi\leq 1$, $\psi \equiv 1$ in $B_R(x_0)$, $\psi \equiv 0$ outside $B_{2R}(x_0)$, and such that
  \begin{align*} 
    \|\nabla \psi\|_{L^\infty} \leq CR^{-1},\; \|D^2\psi \|_{L^\infty} \leq CR^{-2}.
  \end{align*}
  Then, let
  \begin{align*}
    \hat f = \psi f + (1-\psi)g,
  \end{align*}
  In particular, $\hat f \equiv g$ outside $B_{2R(x_0)}$. Therefore, since $f-g\in C^{1,\gam}(B_{2R}(x_0))$, we have $\hat f-g\in C^{1,\gam}(\real^d)$ (as $\hat f-g=\psi\cdot(f-g)$).  Thus, by Theorem \ref{thmNO:LipschitzProperty}, it follows that
  \begin{align*}
    |I(g,x)-I(\hat f,x)| & \leq C\|g-\hat f\|_{C^{1,\gam}(\mathbb{R}^d)} \\ 
	  & = C\|\psi \cdot (g-f)\|_{C^{1,\gam}(B_{2R}(x_0))}\\
	  & \leq C\norm{\psi}_{C^{1,\gam}}\|f-g\|_{C^{1,\gam}(B_{2R}(x_0))}
  \end{align*}
  Meanwhile, since $\hat f \equiv f$ inside $B_{R}(x_0)$, by Lemma \ref{lemNO:DecayWhenAgreeInBr}, we have
  \begin{align*}
    |I(f,x)-I(\hat f,x)| & \leq CR^{-2}\|f-\hat f\|_{L^\infty(\mathbb{R}^d)} + CR^{-1-\gam}
  \end{align*}
  However, note that
  \begin{align*}
    \|f-\hat f\|_{L^\infty(\mathbb{R}^d)} = \| (1-\psi)(f-g)\|_{L^\infty(\mathbb{R}^d)} \leq \|f-g\|_{L^\infty(\mathbb{R}^d)}.
  \end{align*}
  Then, by the triangle inequality, we conclude that
  \begin{align*}
    |I(f,x_0)-I(g,x_0)| \leq C\|f-g\|_{C^{1,\gam}(B_{2R}(x_0))} + CR^{-2}\|f-g\|_{L^\infty(\mathbb{R}^d)} + CR^{-1-\gam}.
  \end{align*}

\end{proof}

\begin{proposition}\label{propNO:LipschtizEstimateUsingOscillationOfDifference}
  If $x_0$ is fixed, $f,g \in \mathcal{K}^*(\gam,\del,m)$ and $f-g\in C^{1,\gam}(B_{2R}(x_0))$, $f,g\leq L$, and $f\geq g$, then there exists $C=C(d,\lambda,\Lambda,\delta,\gam,m,L)$ so that for $R>1$, 
  \begin{align*}
    I(f,x_0) \leq I(g,x_0) + C\left (\osc \limits_{B_{2R}(x_0)} (f-g)+ \|\nabla f-\nabla g\|_{C^{\gam}(B_{2R}(x_0))} + R^{-2}\osc_{L^\infty(\mathbb{R}^d)}(f-g) +R^{-1-\gam} \right).
  \end{align*}
\end{proposition}

\begin{proof}[Proof of Proposition \ref{propNO:LipschtizEstimateUsingOscillationOfDifference}]

  Fix $x_0\in\mathbb{R}^d$ and $R>1$. Let $c := \inf \limits_{B_{2R}(x_0)}(f-g)$ and note that
  \begin{align*}
    \|f-c-g\|_{L^\infty(B_{2R}(x_0))} & \leq \osc \limits_{B_{2R}(x_0)}(f-g)\\
    \|f-c-g\|_{L^\infty(\mathbb{R}^d)} & \leq \osc \limits_{\mathbb{R}^d}(f-g)	
  \end{align*}
  Since, by assumption, $c\geq 0$, Lemma \ref{lemNO:IDependenceOnConstants} shows that
  \begin{align*}
    I(f,x_0) \leq I(f-c,x_0) \;\;\forall\;x\in\mathbb{R}^d.
  \end{align*}
  Then, bounding $I(f-c,x_0)-I(g,x_0)$ from above using Lemma \ref{lemNO:NormSplittingLemma}, it follows that
  \begin{align*}
    I(f,x_0)-I(g,x_0) & \leq C \|f-c-g\|_{C^{1,\gam}B_{2R}(x_0)}+ CR^{-2}\|f-c-g\|_{L^\infty(\mathbb{R}^d)} + CR^{-1-\gam},
  \end{align*}
  and using the above bounds on $\norm{f-c-g}_{L^\infty}$ in $B_{2R}$ and $\real^d$ with the respective oscillation of $f-g$, we have
  \begin{align*}
    I(f,x_0)-I(g,x_0) & \leq C( \osc \limits_{B_{2R}(x_0)} (f-g)+ \|\nabla f-\nabla g\|_{C^{\gam}(B_{2R}(x_0))} )+ CR^{-2}\osc_{L^\infty(\mathbb{R}^d)}(f-g) +CR^{-1-\gam}  ).
  \end{align*}

\end{proof}

A useful corollary of Proposition \ref{propNO:LipschtizEstimateUsingOscillationOfDifference} is that we can construct a bump function that does not increase the values of $I$ too much.

\begin{corollary}\label{corNO:BumpFunctionSmallI}
	Let $\phi(x) = \abs{x}^2(1+|x|^2)^{-1}$ and for $R>1$ and $C\geq 0$, define $\phi_R(x) = C+ \phi(x/R)$.
	Given any $\ep>0$, there exists $C_1(\ep)$, such that for all $R>C_1(\ep)$ for all $g\in\K^*(\gam,\del,m)\intersect (m\text{-}C^{1,\gam}(x_0))$ with $g+\phi_R\in\K^*(\gam,\del,m)$, 
	\begin{align*}
		\ I(g+\phi_R,x_0)-I(g,x_0) \leq \ep.
	\end{align*}
\end{corollary}

\begin{proof}[Proof of Corollary \ref{corNO:BumpFunctionSmallI}]
	We observe that 
	\begin{align*}
	  \osc \limits_{\mathbb{R}^d} \phi_R & = 1\\
	  \|\nabla \phi_R(x)\|_{L^\infty(\mathbb{R}^d)} & \leq R^{-1}\|\nabla \phi_1\|_{L^\infty(\mathbb{R}^d)}\\
	  [\nabla \phi_R(x)]_{C^\gam(\mathbb{R}^d)} & \leq R^{-1-\gam}[\nabla \phi_1(x)]_{C^\gam(\mathbb{R}^d)}.
	\end{align*}
	Next, we note that for all $R>1$, $g+\phi_R\in\K(\gam,\del, m+1)$.  Let us fix $x$, and let us introduce a temporary parameter, $\rho>1$.  We can invoke Proposition \ref{propNO:LipschtizEstimateUsingOscillationOfDifference} for $B_\rho$, to obtain
	\begin{align*}
	 \forall\ x\in\real^d,\ \  I(g+\phi_R,x)- I(g,x)\leq C \left ( \osc \limits_{B_{2\rho}(x)}\phi_R +  \|\nabla \phi_{R}\|_{C^\gam(B_{2\rho}(x))} + \rho^{-2} +\rho^{-1-\gam} \right ).
	\end{align*}
	First, we can choose $C_1(\ep)$ large enough so that for all $\rho>C_1(\ep)$, we have 
	\[
	C(\rho^{-2} +\rho^{-1-\gam})<\frac{\ep}{2}.
	\]
	 Next we take $\rho$ to be fixed, and we can choose $C_2(\ep,\rho)$ large enough so that 
	 \[
	 C \left ( \osc \limits_{B_{2\rho}(x)}\phi_R +  \|\nabla \phi_{R}\|_{C^\gam(B_{2\rho}(x))} \right)<\frac{\ep}{2},
	 \]
	 which is possibly because of the decay of $\grad \phi_R$ that is uniform in $x$.

\end{proof}


\section{The two-phase operator}\label{sec:TwoPhase}
  Let us now return to the analysis for the two phase operator. We will show that the two-phase boundary condition can be completely characterized as a combination of two operators in the same category as $I$, as in Section \ref{sec:NewOperator}.
  
We recall that the sets, $\K(\gam,\del,m)$, $\K^*(\gam,\del,m)$, $\K_*(\gam,\del,m)$ were defined in (\ref{eqNO:DefOfK})--(\ref{eqNO:DefOfKLowStar}).   For $L>\del$ fixed, we consider a modification of these sets as follows:
\begin{align}\label{eq2Ph:DefOfKWithUpperBoundL}
	\K(\gam,\del,m,L) = \{ f\in\K(\gam,\del,m)\ :\  f\leq L-\del  \};
\end{align}
and similarly,
\begin{align*}
	\K^*(\gam,\del,m,L) &= \{ f\in\K^*(\gam,\del,m)\ :\ f\leq L-\del\}\\
	\K_*(\gam,\del,m,L) &= \{ f\in\K_*(\gam,\del,m)\ :\ f\leq L-\del\}.
\end{align*}  
For $f$ in any of these sets, we define the domains,
\begin{align*}
	D^+_f &:= D_f=\{ (x,x_{d+1})\ :\ 0<x_{d+1}<f(x)  \}\\
    D^-_f &:= \{ (x,x_{d+1})\ :\ f(x)< x_{d+1} < L \}.
\end{align*}
Furthermore, we define $U_f$ as the unique continuous viscosity solution to the equations
  \begin{align}\label{eq2Ph:TwoPhaseUfDefinition}
    \left \{ \begin{array}{rrr}
    F_1(D^2U_f,\nabla U_f) = & 0 & \textnormal{in } D^+_f,\\
    F_2(D^2U_f,\nabla U_f) = & 0 & \textnormal{in } D_f^-,\\	
    U_f = & 0 & \textnormal{on } \Gamma_f,\\
    U_f = & 1 & \textnormal{on } \Gamma_0,\\			
    U_f = & -1 & \textnormal{on } \Gamma_L,
    \end{array} \right.		
  \end{align}
  which will have a unique solution whenever $f\in C^{0,1}(\real^d)$ and $\del\leq f\leq L-\del$, see Proposition \ref{propNO:ExistenceForUf}.
  (We recall that the notation for $\Gam_f$, $\Gam_0$, $\Gam_L$ is in Section \ref{sec:Notation}.)

In order to state the free boundary condition, we assume that we are given a Lipschitz continuous function
  \begin{align*}
    G:(0,\infty)^2 \to \mathbb{R},
  \end{align*}
  which satisfies the following uniform monotonicity conditions for a.e. $(a,b)$,
  \begin{align}\label{eq2Ph:GMonotonicityAssumption}
    \lambda_0 \leq \frac{\partial }{\partial a}G(a,b) \leq \Lambda_0,\;\;\lambda_0 \leq -\frac{\partial }{\partial b}G(a,b) \leq \Lambda_0.
  \end{align}  
  The function $U_f$ defined in \eqref{eq2Ph:TwoPhaseUfDefinition} is always continuous across $\Gamma_f$, but in general it will not be differentiable across it. However, in most situations, it will be differentiable separately on each side of $\Gamma_f$. In particular, when $f\in\K(\gam,\del,m,L)$, the following two limits exist
  \begin{align*}
    \nabla U^{+}_f(X) := \lim\limits_{X' \to X, X' \in D^+_f} \nabla U_f(X'),\\
    \nabla U^{-}_f(X) := \lim\limits_{X' \to X, X' \in D_f^-} \nabla U_f(X').	
  \end{align*}
When these operations do exist, they coincide with the slightly more general operators, $\partial^\pm_n U$, which we define as:
\begin{align}\label{eq2Ph:NormalPlusMinusDerivs}
	&\text{for}\ X_0\in\Gamma_f,\ \text{and}\ n(X_0)\ \text{the unit normal derivative to $\Gam_f$, pointing into the set}\ D^+f,\nonumber\\ 
	&\partial^+_nU(X_0):=\lim_{t\to0}\frac{U(X_0+tn(X_0))-U(X_0)}{t}\ \ \text{and}\ \  \partial^-_nU(X_0)=-\lim_{t\to0}\frac{U(X_0-tn(X_0))-U(X_0)}{t}. 
\end{align}
We note that $\partial^\pm_nU$ are normalized so that in the context of $U_f$ solving (\ref{eq2Ph:TwoPhaseUfDefinition}), we have that both $\partial^+_nU_f$ and $\partial^-_nU_f$ are positive, and we also note that the inward normal to $D^-_f$ is the vector $(-n(X_0))$. 
	 
We define the two-phase analogue to the one-phase operator (\ref{eqNO:DefOfI}) by 
  \begin{align}\label{eq2Ph:TwoPhaseOperatorDefinition}
    H(f,x) := G(\partial^+_n U_f(x,f(x)), \partial^-_n U_f(x,f(x))).
  \end{align}
As we are about to show, this ``two phase'' operator, (\ref{eq2Ph:TwoPhaseOperatorDefinition}), may be expressed as a composition with two one-phase operators of the form (\ref{eqNO:DefOfI}). This fact makes it easy to extend basically all of the results in the previous section to the operator in (\ref{eq2Ph:TwoPhaseOperatorDefinition}). Indeed, to (\ref{eq2Ph:TwoPhaseOperatorDefinition}) we associate two operators $I^+$ and $I^-$ of the form (\ref{eqNO:DefOfI}), as follows. First, given $f$ we define $U^+_f$ as the unique solution of
  \begin{align}\label{eq2Ph:UPlusDefinition}
	  \begin{cases}
    F_1(D^2 U^+_f,\nabla  U^+_f) =  0 & \textnormal{in } D^+_f,\\
     U^+_f = 0 & \textnormal{on } \Gamma_f,\\
     U^+_f =  1 & \textnormal{on } \Gamma_0.		
	\end{cases}	
  \end{align}
Similarly, we can define $U^-_f$ as the unique solution of 
\begin{align}\label{eq2Ph:UMinusNaturalDefInDMinus}
	\begin{cases}
	    F_2(D^2U^-_f,\nabla U^-_f) =0\ &\textnormal{in } D_f^-,\\	
	    U^-_f =0\ &\textnormal{on } \Gamma_f,\\			
	    U^-_f = -1\  &\textnormal{on } \Gamma_L
	\end{cases}
\end{align}
Using as notation $n^-(X_0)$ to be the inward normal to the set, $D^-_f$, we define the operators $I^\pm$ as the inward normals to their respective phases,
\begin{align}\label{eq2Ph:DefTwoPhaseIPlusMinusNaturalDefs}
	    I^+(f,x) := \partial_nU^+_f(x,f(x)),\ \ I^-(f,x):= -\partial_{n^-}U^-_f(x,f(x))=\partial_n^- U^-_f(x,f(x)).
\end{align}
This means that we will also write the two-phase operator as
  \begin{align}\label{eq2Ph:TwoPhaseOperatorDefinition-Using-I}
    H(f,x) := G(I^+(f,x),I^-(f,x)).
  \end{align}
We note again we enforce the convention that we seek operators such that both $I^\pm$ are non-negative quantities.  However, thanks to the boundary condition $U^-_f=-1$ on $\Gam_L$, it is not hard to check that $-I^-$ obeys the GCP over the set, $\K(\gam,\del,m,L)$.

In order to make an exact analog between $I^-$ defined in (\ref{eq2Ph:DefTwoPhaseIPlusMinusNaturalDefs}) and the operator, $I$, defined in (\ref{eqNO:DefOfI}), we introduce a transformation so that $I^-$ can be recognized as an operator that uses a solution to an equation in  $D^+_f=D_f$ (in the notation of Section \ref{sec:NewOperator}).  To this end, we need to transform both the equation for the negative phase and the lower boundary that was previously $\Gam_f$.  Thus, we define 
\begin{align*}
	\tilde F(Q,p):= -F(-Q,p),\ \ \text{and}\ \ \tilde f(x):= f(-x),
\end{align*}
and we take $\tilde U^-_f$ to be the unique solution of
  \begin{align}\label{eq2Ph:TildeUMinusDefinition}
	\begin{cases}
    \tilde F_2(D^2 \tilde U^-_f,\nabla \tilde U^-_f) =  0 & \textnormal{in } D_{L-\tilde f},\\
     \tilde U^-_f =  0 & \textnormal{on } \Gamma_{L-\tilde f},\\
     \tilde U^-_f =  1 & \textnormal{on } \Gamma_0.	
	\end{cases}			
  \end{align}
  The natural relationship between $I^-(f,x)$ and $\tilde U_f$ is verified by the following lemma.

\begin{lemma}\label{lem2Ph:IminusUtildeIFtildeEquality}
	For all $f\in\K(\gam,\del,m,L)$, if $I^-(f,x)$ is defined in (\ref{eq2Ph:DefTwoPhaseIPlusMinusNaturalDefs}), $I_{\tilde F_2}$ is defined as in (\ref{eqNO:BulkNonlinearDefI}) and (\ref{eqNO:DefOfI}) with $F$ replaced by $\tilde F_2$, and $\tilde U^-_f$ is defined in (\ref{eq2Ph:TildeUMinusDefinition}), then 
  \begin{align}\label{eq2Ph:IMinusAndITilde}
     I^-(f,x) = -\partial_n\tilde U^-_f(-x,L-\tilde f(-x))=-I_{\tilde F_2}(L-\tilde f,-x),
  \end{align}
  and $-I_{\tilde F}$ is an operator that satisfies the definitions and assumptions of Section \ref{sec:NewOperator}.
\end{lemma}

\begin{proof}[Proof of Lemma \ref{lem2Ph:IminusUtildeIFtildeEquality}] 
First observe that if $y \geq f(x)$, then $L-y \leq L-\tilde f(-x)$, and in fact
\begin{align*}
    (x,y) \in D_f^- \iff (-x,L-y) \in D_{L-\tilde f}.
\end{align*}
Furthermore, we list some similar and useful related observations:
\begin{align*}
	(x,y)\in \Gam_{L-\tilde f}\ \iff\  (-x,L-y)\in \Gam_f,\\
	(x,y)\in \Gam_f\ \iff\ (-x,L-y)\in\Gam_{L-\tilde f},\\
	n_f(x,f(x)) = n_{L-\tilde f}(-x,L-\tilde f(-x)).
\end{align*}
  Thus, given $U_f$ as in \eqref{eq2Ph:TwoPhaseUfDefinition}, define $V:D_{L-\tilde f}\to \mathbb{R}$ by
  \begin{align*}	
    V(x,y) = -U(-x,L-y).		
  \end{align*}	
  The function $V$ solves
  \begin{align*}
    \left \{ \begin{array}{rrr}
    -F_2(-D^2 V,\nabla V) = & 0 & \textnormal{in } D_{L-\tilde f},\\
    V = & 0 & \textnormal{on } \Gamma_{L-\tilde f},\\
    V = & 1 & \textnormal{on } \Gamma_0.		
    \end{array} \right.		
  \end{align*}
  In other words by uniqueness of solutions, we have $V = \tilde U^-_f$.

Using that if $(x,y)\in\Gam_f$, then
\begin{align*}
	\grad U_f(x,y)=\grad V(-x,L-y),
\end{align*}
we see that 
\begin{align*}
	\grad U_f(x,y)\cdot n_f(x,f(x))=\grad V(-x,L-y)\cdot n_{L-\tilde f}(-x,L-\tilde f(-x)).
\end{align*}
We also observe that if $F_2$ satisfies the ($\lam$-$\Lam$) ellipticity of Definition \ref{def:UniformlyElliptic}, then $\tilde F_2$ enjoys the same ellipticity constants and extremal operators.  
Furthermore, by definition (\ref{eqNO:BulkNonlinearDefI}) and (\ref{eqNO:DefOfI}) with $F$ replaced by $\tilde F_2$ we see that by construction, 
\begin{align*}
	\partial_n \tilde U^-_f(-x,L-\tilde f(-x))=I_{\tilde F}(L-\tilde f,-x)
\end{align*}
\end{proof}

We conclude this section with the analogous statements from Section \ref{sec:NewOperator} (Theorem \ref{thmNO:LipschitzProperty} and Corollary \ref{corNO:OnePhaseLipschitzrSemiRegularity}), but tailored to the two-phase operator, $H$.

\begin{theorem}\label{thm2Ph:TwoPhaseOperatorLipschitz}
    If $f$ and $g$ satisfy $f \in \mathcal{K}_*(\delta,\gam,m,L)$ and $g\in \mathcal{K}^*(\delta,\gam,m,L)$ as well as $\phi \in C^{1,\gam}(\mathbb{R}^d)$, $g+\phi\in\K^*(\delta,\gam,m)$, $f-\phi\leq L-\del$, $f-\phi$ is $m\text{-}C^{1,\gam}\text{-semi-convex}$, and all are such that
    \begin{align*}
      (f-g)-\phi \textnormal{ has a non-negative global maximum at } x_0 \in \mathbb{R}^d,
    \end{align*}	  
    then, with $C=C(d,\lambda,\Lambda, \Lambda_0, \delta,\gam,m)$, we have  
  \begin{align*}
    H(f,x_0)-H(g,x_0) \leq C \|\phi\|_{C^{1,\gam}(\mathbb{R}^d)}.
  \end{align*}
\end{theorem}

\begin{proof}
  First of all, observe that \eqref{eq2Ph:GMonotonicityAssumption} implies that
  \begin{align*}
    G(a,b)-G(a',b') \leq \Lambda_0(a-a')_++\Lambda_0(b'-b)_+,
  \end{align*}
  for all $a,a',b,b' \in (0,\infty)$. Therefore, recalling the definitions of $I^\pm$ in (\ref{eq2Ph:DefTwoPhaseIPlusMinusNaturalDefs}) and the monotonicity of (\ref{eq2Ph:GMonotonicityAssumption}), we see that
  \begin{align*}
    H(f,x_0)-H(g,x_0) & = G(I^+(f,x_0),I^-(g,x_0))-G(I^+(f,x_0),I^-(g,x_0)) \\
	  & \leq \Lambda_0 ( I^+(f,x)-I^+(g,x) )_+ + \Lambda_0 ( I^-(g,x)-I^-(f,x))_+.
  \end{align*}
  Now, according to Corollary \ref{corNO:OnePhaseLipschitzrSemiRegularity}, 
  \begin{align*}
    I^+(f,x_0)-I^+(g,x_0) \leq C\|\phi\|_{C^{1,\gam}(\mathbb{R}^d)}.
  \end{align*}
  Lemma \ref{lem2Ph:IminusUtildeIFtildeEquality} and Corollary \ref{corNO:OnePhaseLipschitzrSemiRegularity} (applied when a minimum is attained) show that 
  \begin{align*}
    (I^-)(g,x_0)-(I^-)(f,x_0) = -I_{\tilde F_2}(L-\tilde g,-x_0)- (-I_{\tilde F_2}(L-\tilde f,-x_0))  
	\leq C\|\phi\|_{C^{1,\gam}(\mathbb{R}^d)}.
  \end{align*}
  (Here we note that we used $f-\phi\leq L-\del$ and $f-\phi$ $m\text{-}C^{1,\gam}\text{-semi-convex}$ to obtain $L-\tilde f-(-\tilde \phi)\in \K^*(\gam,\del,m)$, and we used that $(L-\tilde f)-(L-\tilde g)-(-\tilde\phi)$ attains a minimum at $x_0$.)
  It then follows that
  \begin{align*}
    H(g,x_0)-H(f,x_0)  & \leq C\Lambda_0\|\phi\|_{C^{1,\gam}(\mathbb{R}^d)}. 
  \end{align*}
\end{proof}


\section{An integro-differential representation of $I$ and $H$}\label{sec:MinMax}

This section proves that the free boundary operator, $I$, defined in (\ref{eqNO:DefOfI}), and the two-phase operator, $H$, defined in (\ref{eq2Ph:TwoPhaseOperatorDefinition}), are in fact integro-differential operators, and they can be represented, via a min-max procedure, as claimed in Theorem \ref{thm:MetaILipAndMinMax}.  Although the perceptive reader will have noticed, a uniqueness theorem for parabolic equations involving $I$ or $H$ can be deduced from Corollary \ref{corNO:BumpFunctionSmallI}, we still think it is useful to pursue the integro-differential development for $I$.  We would like to point out that, in our opinion, we would have not have realized the properties of $I$ in Section \ref{sec:NewOperator} (especially in Corollary \ref{corNO:OnePhaseLipschitzrSemiRegularity} and Proposition \ref{propNO:LipschtizEstimateUsingOscillationOfDifference}) without expecting that $I$ is a min-max of integro-differential operators.  In particular, the integro-differential framework was essential for our choice to pursue $I$ as a Lipschitz mapping on subsets of $C^{1,\gam}$, instead of the possibly more obvious space of $C^{1,1}$.  This distinction is significant, as the reader will see in this section, because it means that the representative integro-differential operators will not contain any second order terms-- a welcome simplification for using the integro-differential theory.  Furthermore, as one will subsequently see, the preservation of modulus (claimed in Theorem \ref{thm:MainMetaVersion}) can be deduced purely from a comparison theorem for these fractional parabolic equations.  However, we hope that the integro-differential framework will be useful to gain higher regularity results.  In future contexts, it seems that, e.g. Theorem \ref{thm2Ph:TwoPhaseOperatorLipschitz} will not be enough, but rather one will need more refined information as hinted by Theorem \ref{thm:MetaILipAndMinMax}.  Thus, we have chosen to study some initial properties of the integro-differential representation of $I$ here.

The results of this section do not rely on the underlying free boundary problem, and so we treat these operators as a class of their own.

\subsection{The general class of operators}

To this end, we assume that $J$ is an operator that acts on functions on $\real^d$ and that satisfies the following assumptions.

\begin{assumption}\label{assumptionMM}
	$J$ has the following properties:
	\begin{enumerate}[(i)]
		\item $0<\gam<1$ is fixed;
		\item $\displaystyle J:\left(\Union_{\del>0}\Union_{m>\del}\K(\gam,\del,m)\right)\to C^0(\real^d)$.
		\item For each $\del$ and $m$, $J$ is a Lipschitz mapping on the sets $\K(\gam,\del,m)$,
		whose Lipschitz constant, $C(\del,m)$ increases as $\del$ decreases or $m$ increases.
		\item $J$ satisfies the GCP (Definition \ref{def:GCP}).
		\item $J$ is translation invariant.
		\item If $f \in \mathcal{K}(\gam,\del,m)$ and $c>0$ is a constant, then
\begin{align*}
 \forall\;x\in\mathbb{R}^d,\ J(f+c,x) \leq J(f,x).
\end{align*}
		\item $J$ enjoys the operator splitting property: $\exists\ C=C(\gam,\del,m)$ and $\exists\ \om$ (a modulus with $\om(R)\to0$ as $R\to\infty$), such that $\forall f,g\in \K(\gam,\del,m)$, for $R>1$
		\begin{align*}
			\norm{J(f,\cdot)-J(g,\cdot)}_{L^\infty(B_R)}\leq C(\norm{f-g}_{C^{1,\gam}(B_{2R})}+\om(R)\norm{f-g}_{L^\infty(\real^d)} + \om(R)).
		\end{align*}
	\end{enumerate}
\end{assumption}

\begin{rem}
We note that as presented, the assumptions for $J$ only include operators like $I(f)$, $G(I(f))$, $G(I(f))\cdot\sqrt{1+\abs{\grad f}^2}$, which appear in the one-phase problem.  However, thanks to Lemma \ref{lem2Ph:IminusUtildeIFtildeEquality}, we see that the operator, $H(f)$, which is required for the two-phase problem, reduces to the case of $I(f)$, with the additional constraint that all of the constants involved will also depend upon the height of the upper boundary, which is given by $L$ in Section \ref{sec:TwoPhase}.
\end{rem}

The first results we will show are that any $J$, as in Assumption \ref{assumptionMM}, can be represented as a min-max of integro-differential operators, followed by proving some properties of the corresponding extremal operators for the class that includes $J$.  The basis of the results in this section is taken from the main result in \cite{GuSc-2016MinMaxNonlocalarXiv}, and for convenience, we restate it here, in the context of the operators, $J$.

\begin{theorem}[from Theorem 1.6, Proposition 1.7 of \cite{GuSc-2016MinMaxNonlocalarXiv}]\label{thmMM:GuScResult}
	If $J$ satisfies items (iii), (iv), (vii) of Assumption \ref{assumptionMM}, then  
	\begin{align}\label{eqMM:MinMaxJInitial}
	  \forall\ f\in\K(\gam,\del,m),\ \ \   J(f,x)  = \min_{g \in \K(\gam,\del,m),}\max_{L\in \mathcal L(\K(\gam,\del,m))}
	  \{ J(g,x)+L(f-g,x)\},
	\end{align}
	where $\mathcal L(\K(\gam,\del,m))$ is a collection of linear operators, $L: C^{1,\gam}\to C^0$, that enjoy the following form:
	\begin{align}\label{eqMM:LinearInMinMaxOldVersion}
		L(f,x) = c(x)f(x) + b(x)\cdot\grad f(x) + 
		\int_{\real^d}\left(  f(x+h)-f(x)-\Indicator_{B_1}(h)\grad f(x)\cdot h   \right) \mu(x,dh),
	\end{align}
	and satisfy for some $C$, uniformly depending on the norm $\norm{J}_{\K(\gam,\del,m)\to C^{\gam'}}$,
	\begin{align}\label{eqMM:BoundsOnLinearIngredients}
		\norm{c}_{L^\infty},\ \norm{b}_{L^\infty}\leq C\ \ \text{and}\ \ 
		 \sup_{x\in\real^d}\ \int_{\real^d}\min\{\abs{h}^{1+\gam},1\}\mu(x,dh)\leq C.
	\end{align}
	Finally, the L\'evy measures satisfy, for a uniform constant across the family, $\L$, 
	\begin{align}\label{eqMM:LevyMeasureDecay}
		\sup_{x\in\real^d}\ \int_{\real^d\setminus B_{R}}\mu(x,dh)\leq C\om(R).
	\end{align}
\end{theorem}

Although for many purposes, this representation in (\ref{eqMM:MinMaxJInitial}) and (\ref{eqMM:LinearInMinMaxOldVersion}) will suffice, in our current context we will require (and prove) something slightly more precise.  It turns out that since $J$ is translation invariant, one can take a different approach to the representation in (\ref{eqMM:MinMaxJInitial}) and obtain more detail on the collection of $L$ in (\ref{eqMM:LinearInMinMaxOldVersion}).

A key building block that goes into the proof of Theorem \ref{thmMM:GuScResult} is the structure of linear functionals on $\K(\gam,\del,m)$ that enjoy the GCP.  We state it explicitly, and refer the proof to modifications of \cite{Courrege-1965formePrincipeMaximum} or \cite{GuSc-2016MinMaxNonlocalarXiv}.

\begin{proposition}\label{propMM:LinearFunctionalsInJDifferential}
	If $\ell:\K(\gam,\del,m)\to\real$ is a bounded linear functional that enjoys the GCP based at $x_0=0$ (Definition \ref{def:GCP}), then there exists $b\in\real$, $c\in\real$, and a measure, $\mu$, so that $\forall\ f\in\K(\gam,\del,m)$, 
	\begin{align*}
		 \ell(f) = cf(0) + b\cdot\grad f(0) + 
		\int_{\real^d}\left(  f(h)-f(0)-\Indicator_{B_1}(h)\grad f(0)\cdot h   \right) \mu(dh),
	\end{align*}
	where $b$, $c$, and $\mu$ satisfy the same bounds as in Theorem \ref{thmMM:GuScResult}, but none of them depend on $x$.
\end{proposition}

We will state the following theorem, and the proof will appear after two other results are established.

\begin{theorem}[Translation invariant min-max]\label{thmMM:TranslationInvariantMinMax}
	If $J$ is as in Assumption \ref{assumptionMM}, then $J$ admits a min-max representation as in (\ref{eqMM:MinMaxJInitial}) in which each $L$ is also \emph{translation invariant}; i.e. $L\in\mathcal L_{inv}$, and $\mathcal L_{inv}$ is the class which contains all linear operators of the form
	\begin{align}\label{eqMM:TranslationInvariantLinear}
		L(f,x) = cf(x) + b\cdot\grad f(x) + 
		\int_{\real^d}\left(  f(x+h)-f(x)-\Indicator_{B_1}(h)\grad f(x)\cdot h   \right) \mu(dh),
	\end{align}
	where each of $c$, $b$, $\mu$ are independent of $x$.  Furthermore, given a $C_1>0$, there exists a $C_2>0$ such that for all $J$ with a Lipschitz norm bounded by $C_1$, all such  $c$, $b$, and $\mu$, resulting from an $L\in\L_{inv}$, we have $\abs{c}\leq C_2$, $\abs{b}\leq C_2$, 
	\begin{align*}
		\int_{\real^d}\max\left(\abs{h}^{1+\gam},1\right)\mu(dh)\leq C_2,\ \ \text{and}\ \ 
		\int_{\real^d\setminus B_{R}}\mu(dh)\leq C_2\om(R).
	\end{align*} 
The class of operators, $\L_{inv}$ in (\ref{eqMM:TranslationInvariantLinear}), and the min-max in (\ref{eqMM:MinMaxJInitial}) both depend on $\gam$, $\del$, and $m$, via $\K(\gam,\del,m)$.
\end{theorem}

There are two important consequences of Theorem \ref{thmMM:TranslationInvariantMinMax}.  First, following as in \cite[Section 4.6 and Proposition 4.35]{GuSc-2016MinMaxNonlocalarXiv}, this identifies the natural class of extremal operators for $J$ as
\begin{align}\label{eqMM:NonlocalExtremalOpDef-Generic}
	M^+_{inv}(f,x)=\max_{L\in\L_{inv}}\left(  L(f,x) \right)
	\ \ \text{and}\ \ 
	M^-_{inv}(f,x)=\min_{L\in\L_{inv}}\left(  L(f,x) \right).
\end{align}
These extremal operators are defined specifically to produce, for all $f,g\in \K(\gam,\del,m)$, the inequalities:
\begin{align}\label{eqMM:ExtremalInequalityForJ}
	M^-_{inv}(f-g,x)\leq J(f,x)-J(g,x) \leq M^+_{inv}(f-g,x).
\end{align}
We note that the translation invariance shows that 
\begin{align*}
	\forall\ x\in\real^d,\ \ M^\pm_{inv}(f,x) = M^\pm_{inv}(\tau_xf,0).
\end{align*}
Second, the translation invariance of $J$ shows that all of the desired properties of $\L_{inv}$  can be obtained from studying the (nonlinear) functional
\begin{align*}
	j:\K(\gam,\del,m)\to\real\ \ \text{with}\ \ j(f):=J(f,0).
\end{align*}
This functional, $j$, also satisfies the GCP based at $x_0=0$ (Definition \ref{def:GCP}).
As will be shown in the remainder of this subsection, this means that once $J$ is fixed, the class $\L$ can be restricted even further so that one uses only those $L$ in the Clarke differential of $j$, i.e. those $L$ such that
\begin{align*}
	L(f,x)=\ell (\tau_x f),\ \ \text{for some choice of}\ \ \ell\in[\partial j]_{\K(\gam,\del,m)},
\end{align*}
where we recall in Definition \ref{def:ClarkeDifferential} the notation, $[\partial j]_{\K(\gam,\del,m)}$.   Thus, we can define a different extremal operator, depending explicitly on $J$ and $\K(\gam,\del,m)$ as
\begin{align}\label{eqMM:NonlocalExtremalOpDef-JSpecific}
	M^+_{J,\K(\gam,\del,m)}(f,x)=\max_{\ell\in[\partial j]_{\K(\gam,\del,m)}}\left(  l(\tau_x f) \right)\ \ \text{and}\ \ 
	M^-_{J,\K(\gam,\del,m)}(f,x)=\min_{\ell\in[\partial j]_{\K(\gam,\del,m)}}\left(  l(\tau_x f) \right),
\end{align}
and the inequalities in (\ref{eqMM:ExtremalInequalityForJ}) still hold.  We note, by definition, that $M^\pm_{J,\K(\gam,\del,m)}$ are translation invariant.  It is also true that these $M^\pm_{J,\K(\gam,\del,m)}$ serve as extremal operators for $J$, which we record here.

\begin{proposition}\label{propMM:ExtremalOp-JSpecific-AndLip}
	If $J$ is fixed, and $M^\pm_{J,\K(\gam,\del,m)}$ are defined in (\ref{eqMM:NonlocalExtremalOpDef-JSpecific}), then $J$ and $M^\pm_{J,\K(\gam,\del,m)}$ also obey the inequalities in (\ref{eqMM:ExtremalInequalityForJ}) with $M^\pm_{inv}$ replaced by $M^\pm_{J,\K(\gam,\del,m)}$.  Furthermore, $M^\pm_{J,\K(\gam,\del,m)}$ are Lipschitz functions, as mappings of $C^{1,\gam}(\real^d)\to C^{0}(\real^d)$ with a Lipschitz norm bounded by that of $J$.
\end{proposition}

\begin{proof}[Proof of Proposition \ref{propMM:ExtremalOp-JSpecific-AndLip}]
	This proof is a direct consequence of the fact that $[\partial j]_{\K(\gam,\del,m)}$ is non-empty (\cite[Proposition 2.1.2]{Clarke-1990OptimizationNonSmoothAnalysisBook}), combined with a mean value property enjoyed by $j$ and $\partial j$, due to Lebourg (see, e.g. \cite[Theorem 2.3.7]{Clarke-1990OptimizationNonSmoothAnalysisBook}).  It says that given any $f,g\in\K(\gam,\del,m)$, there exists an element, $\ell\in[\partial j]_{\K(\gam,\del,m)}$, so that 
	\begin{align*}
		j(f)-j(g)=\ell(f-g).
	\end{align*}
	Hence, taking a max over $[\partial j]_{\K(\gam,\del,m)}$, (we note as in \cite[Proposition 2.1.2]{Clarke-1990OptimizationNonSmoothAnalysisBook}, $[\partial j]_{\K(\gam,\del,m)}$ is weak-$*$ compact)
	\begin{align*}
		\forall,\ f,g\in\K(\gam,\del,m),\ \ j(f)-j(g)\leq \max_{\ell\in[\partial j]_{\K(\gam,\del,m)}}\left( \ell(f-g)  \right).
	\end{align*}
	A similar argument produces the lower inequality in (\ref{eqMM:ExtremalInequalityForJ}).  The Lipschitz nature of $M^\pm_{J,\K(\gam,\del,m)}$ was already apparent from the original invocation of \cite[Theorem 1.6, Proposition 1.7]{GuSc-2016MinMaxNonlocalarXiv}, which appeared in (\ref{eqMM:MinMaxJInitial})--(\ref{eqMM:BoundsOnLinearIngredients}); correspondingly, this implies that the bounds on the linear operators are uniform over these ingredients.

\end{proof}

Finally, we give the one last argument for the proof of Theorem \ref{thmMM:TranslationInvariantMinMax}.

\begin{proof}[Proof of Theorem \ref{thmMM:TranslationInvariantMinMax}]
	Following that last step of the proof of Proposition \ref{propMM:ExtremalOp-JSpecific-AndLip}, we see that 
	\begin{align*}
		\forall\ f,g\in\K(\gam,\del,m),\ \ j(f)\leq \max_{\ell\in[\partial j]_{J,\K(\gam,\del,m)}}\left(j(g)+\ell(f-g)\right).
	\end{align*}
	Now, taking a minimum over $g$, gives
	\begin{align*}
		j(f)\leq \min_{g\in\K(\gam,\del,m)}\max_{\ell\in[\partial j]_{\K(\gam,\del,m)}}\left( j(g)+\ell(f-g)  \right),
	\end{align*}
	and evaluating when $g=f$, shows
	\begin{align*}
		\min_{g\in\K(\gam,\del,m)}\max_{\ell\in[\partial j]_{J,\K(\gam,\del,m)}}\left( j(g)+\ell(f-g)  \right)\leq j(f).
	\end{align*}
	Now, recalling the fact that $J$ is translation invariant, we see that 
	\begin{align*}
		J(f,x)=j(\tau_x f),
	\end{align*}
	and now the theorem follows from that fact that any $\ell\in[\partial j]_{J,\K(\gam,\del,m)}$ also enjoys the comparison principle, and thus must be of the form (\ref{eqMM:TranslationInvariantLinear}), per Proposition \ref{propMM:LinearFunctionalsInJDifferential}.

\end{proof}

\begin{corollary}\label{corMM:JExtendsViaMinMax}
	For each $\gam$, $\del$, and $m$ fixed, the operator, $J$, can be extended-- with respect to the set, $\K(\gam,\del,m)$-- to a function on all of $C^{1,\gam}(\real^d)$, instead of just $\K(\gam,\del,m)\subset C^{1,\gam}$, and this extension is still Lipschitz and enjoys the GCP.  Specifically, if we define
	\begin{align}\label{eqMM:TildeJDef}
		\tilde J_{\K(\gam,\del,m)}: C^{1,\gam}(\real^d)\to C^0(\real^d)\ \ \text{via}\ \ 
		\tilde J_{\K(\gam\del,m)}(f,x) = \min_{g\in\K(\gam,\del,m)}\max_{\ell\in[\partial j]_{J,\K(\gam,\del,m)}}\left( j(g) + \ell((\tau_x f)-g) \right),
	\end{align}
	then $\tilde J$ is Lipschitz, $\tilde J=J$ on $\K(\gam,\del,m)$, and $\tilde J$ enjoys the GCP.
\end{corollary}

\begin{corollary}\label{corMMR:TildeJClassicalC11AtX}
	If $x$ is fixed, $\gam,\del,m$ are given, and $f$ is $m\text{-}C^{1,\gam}(x)$, then $\tilde J_{\K(\gam,\del,m)}(f,x)$ is classically defined via (\ref{eqMM:TildeJDef}).
 	
\end{corollary}

\begin{proof}[Proof of Corollary \ref{corMMR:TildeJClassicalC11AtX}]
	If $f$ is $m\text{-}C^{1,\gam}(x)$, then this means that there are two functions, $f^+$ and $f^-$, such that $f^\pm\in C^{1,\gam}(\real^d)$ and for all $y$, $f^-(y)\leq f(y)\leq f^+(y)$, with $f^-(x)=f(x)=f^+(x)$.  (Note, these $f^\pm$ would of course depend on the parameter, $r$, in Definition \ref{defNO:PointwiseC1gam}, but that is irrelevant.)  Hence,  the formula in (\ref{eqMM:TildeJDef}) holds classically by Proposition \ref{propMM:LinearFunctionalsInJDifferential} and the estimate on $\mu$ that appears in (\ref{eqMM:BoundsOnLinearIngredients}).

\end{proof}

An important consequence of the Assumption \ref{assumptionMM}-(vi) is that it shows that $c^{ij}\leq 0$, which we state and prove here.  

\begin{proposition}\label{propMM:ExtremalOperatorActionOnConstants}
  If $J$ is as in Assumption \ref{assumptionMM} and $M^+_{J,\K(\gam,\del,m)}$ is the extremal operator defined in (\ref{eqMM:NonlocalExtremalOpDef-JSpecific}), then
  \begin{align*}
    \forall\ x\in\real^d,\ \ M^+_{J,\K(\gam,\del,m)}(1,x)\leq 0.
  \end{align*}
  As a consequence, we see that necessarily for all $L\in\L_{inv}$ in (\ref{eqMM:TranslationInvariantLinear}), we have $c\leq0$.
\end{proposition}

\begin{proof}

  By the translation invariance of $M^+_{J,\K(\gam,\del,m)}$, it suffices to show that
  \begin{align*}
    M^+_{J,\K(\gam,\del,m)}(1,0) \leq 0.  	  
  \end{align*}
  Let us write $m_0$ to denote the Lipschitz functional $m_0(f) = M^+_{J,\K(\gam,\del,m)}(f,0)$ (the Lipschitz character is part of Proposition \ref{propMM:ExtremalOp-JSpecific-AndLip}), and recall that
  \begin{align*}
   m_0(f)= M^+_{J,\K(\gam,\del,m)}(f,0) & = \max\limits_{\ell \in [\partial j]_{\K(\gam,\del,m)}} \ell(f)\ \ \ \text{where $j$ is}\ j(f)=J(f,0).
  \end{align*}	  
  Assumption \ref{assumptionMM}-(vi) will show that all $\ell \in[\partial j]_{\K(\gam,\del,m)}$ will have a sign for $\ell(1)$.  Indeed, if $f\in C^{1,\gam}$, then by Assumption \ref{assumptionMM}-(vi), we have
  \begin{align*}
    \limsup \limits_{s\to 0^+}\frac{j(f+s\cdot 1)-j(f)}{s}\leq 0.
  \end{align*}	  
  From Definition \ref{def:ClarkeDifferential} for $[\partial j]_{\K(\gam,\del,m)}$, we conclude that
  \begin{align*}
    \text{if}\ \ell\in[\partial j]_{\K(\gam,\del,m)},\ \text{then}\ \ell(1)\leq 0.	   
  \end{align*}	  
  Thus, via the definition of $M^+_{J,\K(\gam,\del,m)}$ in (\ref{eqMM:NonlocalExtremalOpDef-JSpecific}), we see that $M^+_{J,\K(\gam,\del,m)}(1,0) = m_0(1)\leq 0$, as we wanted.

\end{proof}

Finally, to end this section, we will mention how the results of Sections \ref{sec:NewOperator} and \ref{sec:TwoPhase} ensure that the one-and-two phase operators are covered by the assumptions of this section.

\begin{theorem}\label{thmMM:IandHSatisfyAssumptions}
	The operators, $I$ and $H$, defined respectively in (\ref{eqNO:DefOfI}) and (\ref{eq2Ph:TwoPhaseOperatorDefinition}) satisfy Assumption \ref{assumptionMM}.  In the case of $H$, all of the constants and Lipschitz norms will depend on $\K(\gam,\del,m,L)$, instead of just $\K(\gam,\del,m)$ (where $L$ is the height of the top boundary in Section \ref{sec:TwoPhase}).  Furthermore, the same is true as well for the operators, for respectively the one-phase and two-phase evolutions:
	\begin{align*}
		G(I(f))\cdot\sqrt{1+\abs{\grad f}^2}\ \ \text{and}\ \ H(f)\cdot\sqrt{1+\abs{\grad f}^2}.
	\end{align*}
\end{theorem}

\begin{proof}
	For $I$, this is a direct consequence of the results in Theorem \ref{thmNO:LipschitzProperty}, Lemma \ref{lemNO:IHasGCP}, Proposition \ref{propNO:TranslationInvariant}, Lemma \ref{lemNO:IDependenceOnConstants}, and Lemma \ref{lemNO:NormSplittingLemma}.  Hence for the map $G(I(f))$, this follows from the Lipschitz nature of $G$.  For $H$, this follows immediately from the Lipschitz nature of $G$, combined with the observation of Lemma \ref{lem2Ph:IminusUtildeIFtildeEquality}.  Indeed, $I^+$ already has the desired properties (as in Section \ref{sec:NewOperator}), and by Lemma \ref{lem2Ph:IminusUtildeIFtildeEquality} $I^-$ does as well.  Everything can be checked for $H$ with a calculation similar to that of the proof of Theorem \ref{thm2Ph:TwoPhaseOperatorLipschitz}.
	
	Finally, for the cases of 
	\begin{align*}
		G(I(f))\cdot\sqrt{1+\abs{\grad f}^2}\ \ \text{and}\ \ H(f)\cdot\sqrt{1+\abs{\grad f}^2},
	\end{align*}
	we note that the mapping,
	\begin{align*}
		f\mapsto \sqrt{1+\abs{\grad f}^2}, 
	\end{align*}
	is bounded and Lipschitz on each of the sets, $\K(\gam,\del,m,L)$.  This, combined with the local nature of $\sqrt{1+\abs{\grad f}^2}$, is enough to preserve all of the assumptions listed in \ref{assumptionMM}, as soon as $I$ or $H$ satisfy them as well.
	
\end{proof}


\section{Comparison theorem for parabolic viscosity solutions}\label{sec:Comparison}

This section is dedicated to the uniqueness and existence of viscosity solutions for the class of fractional parabolic equations that contain those in Assumption \ref{assumptionMM}.  As we just saw in Theorem \ref{thmMM:IandHSatisfyAssumptions}, this contains the operators $G(I(f))\cdot\sqrt{1+\abs{\grad f}^2}$ and $H(f)\cdot\sqrt{1+\abs{\grad f}^2}$ for the free boundary flow.    We want to emphasize that this section should be viewed as a collection of small modifications to the arguments found in the paper of Silvestre, \cite[Appendix A]{Silv-2011DifferentiabilityCriticalHJ}, for some similar integro-differential equations.  At the end of this section, we will show how a similar comparison result can be deduced from Theorem \ref{thmNO:LipschitzProperty}, Lemma \ref{lemNO:PointwiseEvaluation}, and Corollary \ref{corNO:BumpFunctionSmallI}.


\subsection{Using the assumptions of Section \ref{sec:MinMax}}

In this part, we show the uniqueness for parabolic equations governed by general operators satisfying the assumptions of Section \ref{sec:MinMax}.  For such $J$ that satisfy Assumption \ref{assumptionMM}, we prove uniqueness for 
\begin{align}\label{eqCP:ParabolicMain}
	\begin{cases}
		\partial_t f = J(f)\ &\text{in}\ \real^d\times (0,T]\\
		f(\cdot,0)=f_0\ &\text{on}\ \real^d\times\{0\}.
	\end{cases}
\end{align}

\begin{rem}
	We present all the proofs in a fashion that applies to the one-phase evolution, using the operator, e.g. $J(f)=G(I(f))\cdot\sqrt{1+\abs{\grad f}^2}$.  This is apparent in the need to reference sets such as $\K(\gam,\del,m)$, and the dependence of arguments on the parameters, $\del$ and $m$.  All of these results hold for the two-phase problem, e.g. $J(f)=H(f)\cdot\sqrt{1+\abs{\grad f}^2}$ as well, with the additional constraint that constants and arguments will also depend upon $L$ and $\K(\gam,\del,m,L)$.
\end{rem}

 It turns out that since $J$ is translation invariant, the viscosity solutions theory is well developed and straightforward to implement.  The key step is the regularization via the operations of inf-convolution and sup-convolution (also known earlier as the Moreau-Yosida approximation on Hilbert spaces).  These operations were expanded upon by Lasry-Lions \cite{LaLi-1986RegularizationHilbertSpace}; first utilized for uniqueness of elliptic equations by Jensen in \cite{Jensen-1988UniquenessARMA}; and improved upon by Ishii in \cite{Ishii-1989UniqueViscSolSecondOrderCPAM}. A good demonstration of their utility appears in the work of Jensen-Lions-Souganids \cite{JeLiSo-88UniquenessSecondOrder}.  As mentioned above, most of the arguments that we need that pertain to the nonlocal terms come nearly verbatim from the work of Silvestre in \cite[Appendix A]{Silv-2011DifferentiabilityCriticalHJ}, and we will refer most of the details to that work.  We begin by recognizing the fact that for equations whose order is strictly less that 2, the inf/sup-convolution technique works even more seamlessly than in the case of second order equations.  This is due to the fact that a function with $C^{1,1}$ regularity in space is good enough to classically evaluate $J$.  Such benefits of working with sub and super solutions that can be evaluated classically can be seen in Caffarelli-Silvestre \cite[Section 5]{CaSi-09RegularityIntegroDiff} and Barles-Imbert \cite{BaIm-07}, as well as in \cite{Silv-2011DifferentiabilityCriticalHJ}.  For a general overview of viscosity solutions, we suggest the User's Guide \cite{CrIsLi-92}, and for an introduction of these techniques in the context of parabolic equations, we suggest the lecture notes of Imbert-Silvestre \cite{ImbertSilvestre-2013IntroToFullyNonlinearParabolic}.

First, we must recall the definition of a viscosity solution of the parabolic equation.  For our context, this is equivalent to any of the definitions that appear in Barles-Imbert \cite{BaIm-07}, Caffarelli-Silvestre \cite{CaSi-09RegularityIntegroDiff}, or Silvestre \cite{Silv-2011DifferentiabilityCriticalHJ}.

\begin{definition}\label{defCP:ViscositySolutionParabolicEq}
	We say that $f$ is a viscosity subsolution of (\ref{eqCP:ParabolicMain}) if $f$ is upper semi-continuous, that $\inf f>0$, and $f$ has the property that for all $(x,t)\in\real^d\times(0,T]$ for which  there exists a function, $\phi$, so that for some $\beta>0$ with $\gam+\beta<1$, $\phi\in C^{1,\gam+\beta}(\real^d)$ in space and $C^1$ in time, and for some $t>r>0$, $f-\phi$ attains a global maximum over $\real^d\times (t-r,t]$, then $\phi$ must satisfy
	\begin{align*}
		\partial_t\phi(x,t)\leq J(\phi,x).
	\end{align*}
	
	We say that $g$ is a viscosity supersolution if $g$ is lower semi-continuous, $\inf g >0$, and the above properties are replaced by $g-\phi$ attains a minimum for $\inf \phi>0$, and
	\begin{align*}
		\partial_t\phi(x,t)\geq J(\phi,x).
	\end{align*}
	
	We say that $f$ is a viscosity solution if $f$ is both a subsolution and a supersolution.
\end{definition}

\begin{rem}
	We note that we have chosen the space of test functions to have the natural regularity in space that is associated to Assumption \ref{assumptionMM}; this is instead of the usual $C^2$ regularity that typically appears in other works that use viscosity solutions for integro-differential equations.
\end{rem}

\begin{rem}
	We note that definition \ref{defCP:ViscositySolutionParabolicEq} is not identically the one in \cite[Section 2]{Silv-2011DifferentiabilityCriticalHJ}, however, it is equivalent, as evidenced in Barles-Imbert \cite{BaIm-07}.
\end{rem}

Thanks to the results in Section \ref{sec:MinMax}, it is easy to show that the definition of viscosity solutions extends to a class of functions with less regularity in space, namely test functions that, for some $\beta>0$ with $\gam+\beta<1$, are punctually $C^{1,\gam+\beta}$ in space at the point of contact will produce the same set of solutions.

\begin{lemma}\label{lemCP:SubSolutionClassicalEval}
	(This is the analog of the result of, e.g. \cite[Lemma 4.3]{CaSi-09RegularityIntegroDiff}.)  If $f$ is a viscosity subsolution of (\ref{eqCP:ParabolicMain}) and $(x,t)\in\real^{d}\times(0,T)$ is a point such that $f-\phi$ attains a maximum at $(x,t)$, for some $\phi$ that is punctually $C^{1,\gam+\beta}(x)$ in space and $C^1$ in time, with $\beta>0$ and $\gam+\beta<1$, then there exists a choice of  $\del_0$ small enough and $m_0$ large enough (depending on $\phi$), so that $\tilde J_{\K(\gam,\del_0,m_0)}(\phi,x)$ is defined classically and
	\begin{align*}
		\partial_t\phi(x,t)\leq \tilde J_{\K(\gam,\del_0,m_0)}(\phi,x).
	\end{align*}
	($\tilde J$ is defined in (\ref{eqMM:TildeJDef}) and its appearance is because it is not required that $\phi$ is in the set, $\K(\gam,\del_0,m_0)$.)
	The analogous result holds for $g$ that are supersolutions and for those $\phi$ that also additionally satisfy $\inf \phi>0$.  (The inequality also holds for $\del<\del_0$ and $m>m_0$.)
\end{lemma}

\begin{proof}
As we mentioned in the statement of the lemma, this is basically the result that is presented in \cite[Lemma 4.3]{CaSi-09RegularityIntegroDiff}.  The proof is also very similar.  We provide some details for the sake of readability.  First of all, the well defined nature of $\tilde J_{\K(\gam,\del,m)}(\phi,x)$ is a consequence of Corollary \ref{corMMR:TildeJClassicalC11AtX}, from which it is useful to note does not require the extra $\gam+\beta$ regularity of $\phi$ at $x$.  The extra regularity is required to ensure that $\tilde J_{\K(\gam,\del,m)}(\phi,x)$ actually obeys the correct inequality.

We observe that by the assumption on $\phi$, there exist functions, $f^+$, $f^-$, both in $C^{1,\gam+\beta}(\real^d)$, and that satisfy
\begin{align*}
	\forall\ y\in\real^d,\ \ f^-(y)\leq \phi(y)\leq f^+(y),\ \ \text{and}\ \ f^-(x)=\phi(x)=f^+(x).
\end{align*}
For each $r>0$, we can define the function, $\phi_r$, as
\begin{align*}
	\phi_r(y)=
	\begin{cases}
		f^+(y)\ &\text{if}\ y\in B_r(x)\\
		\phi(y)\ &\text{otherwise}.
	\end{cases}
\end{align*}
Furthermore, we note that we also have the ordering
\begin{align*}
	f^-\leq \phi\leq \phi_r\leq f^+,
\end{align*}
and that
\begin{align*}
	\grad f^-(x)=\grad \phi(x)=\grad \phi_r(x)=\grad f^+(x).
\end{align*}
Finally, we can assume without loss of generality that
\begin{align*}
	\partial_t \phi(x,t) = \partial_t f^+(x,t).
\end{align*}
The definition of viscosity solution shows that 
\begin{align*}
	\partial_t \phi(x,t)=\partial_t f^+(x,t)\leq J(f^+,x).
\end{align*}
Now we can use Corollaries \ref{corMM:JExtendsViaMinMax} and \ref{corMMR:TildeJClassicalC11AtX}, as well as the relevant extremal operators to see that since there exists some $\del_0$ and $m_0$ so that $f^+\in \K(\gam,\del_0,m_0)$, which we simply call $\K(\gam,\del_0,m_0)=\K$, 
\begin{align*}
	J(f^+,x)=\tilde J_\K(f^+,x)&\leq \tilde J_\K(\phi_r,x) + M^+_{J,\K}(f^+-\phi_r,x)\\
	&\leq \tilde J_\K(\phi,x) + M^+_{J,\K}(\phi_r-\phi,x) + M^+_{J,\K}(f^+-\phi_r,x).
\end{align*}
Again, we recall that all of the operators involving $\tilde J_\K$ and $M^+_{J,\K}$ are well defined on each of these functions, as they are all punctually $C^{1,\gam}(x)$.  The only question now is whether or not the last two terms vanish as $r\to0$.  This is where the slightly higher regularity is used.  Indeed, because the set, $\K$, is now fixed, and thanks to the bounds in Theorem \ref{thmMM:TranslationInvariantMinMax}, as well as the definition of $M^+_{J,\K}$, we see that all of these therms include a measure $\mu$ and an integral that are dominated by
\begin{align*}
\int_{B_r} C_\phi\abs{y}^{1+\gam+\beta}\mu(dy),
\end{align*}
where $C_\phi$ is a constant depending the fact that $\phi\in C^{1,\gam+\beta}(x)$.  Furthermore, we know that $M^+_{J,\K}$ is a max over a family of these $\mu$, all of the $\mu$ (arising from the set, $[J]_{\K}$), satisfy the uniform bound (\ref{eqMM:BoundsOnLinearIngredients}).  Hence, the remaining terms above, that depend on $r$, are all bounded by
\begin{align*}
	\int_{B_r} C_\phi\abs{y}^{1+\gam+\beta}\mu(dy)
	\leq C_\phi r^{\beta}\sup_{\mu\in[J]_\K}\int_{\real^d} \min\{\abs{y}^{1+\gam},1\}\mu(dy)
	\leq C_2\cdot C_\phi\cdot r^\beta.
\end{align*}
Hence, taking $r\to0$ finishes the claim of the lemma.

\end{proof}

Following either \cite{ImbertSilvestre-2013IntroToFullyNonlinearParabolic} or \cite{Silv-2011DifferentiabilityCriticalHJ}, we see that the inf and sup convolutions that are appropriate for parabolic equations are the following:

\begin{align}
  \phi^\varepsilon(x,t) & = \sup \limits_{y\in\real^d,s\in[0,\infty)} \phi(y,s)-\frac{1}{2\varepsilon}(t-s)^2-\frac{1}{2\varepsilon}|x-y|^2
  \label{eqCP:SupConvolution}\\
  \phi_\varepsilon(x,t) & = \inf \limits_{y\in\real^d,s\in[0,\infty)} \phi(y,s)+\frac{1}{2\varepsilon}(t-s)^2+\frac{1}{\varepsilon}|x-y|^2.
  \label{eqCP:InfConvolution}
\end{align}

The first result, which uses the inf/sup convolution, is a comparison theorem.  It follows exactly as in the proof of \cite[Lemma 3.2]{Silv-2011DifferentiabilityCriticalHJ}.

\begin{lemma}\label{lemCP:EquationForTheDifference}
  Assume that $f,g:\mathbb{R}^d\times [0,T]\to\mathbb{R}$ are respectively upper and lower semicontinuous functions such that $\del\leq f\leq m$ and $\del\leq g\leq m$, and that, in the viscosity sense, satisfy:
  \begin{align*}
    \partial_t f \leq J(f)\textnormal{ and } \partial_t g \geq J(g).	
  \end{align*}
  If $f^\ep$ and $g_\ep$ are the inf/sup convolutions of $f$, $g$, respectively (defined in (\ref{eqCP:InfConvolution}), (\ref{eqCP:SupConvolution})),
  then, for each $\ep\in(0,1)$ the function $w^\ep=f^\ep-g_\ep$ solves, in the viscosity sense
  \begin{align*}
    \partial_t w^\ep & \leq M^+_{J,\K(\gam,\del/2,m\ep^{-1})}(w^\ep)
  \end{align*}
  (here, we can take $m=\norm{f}_{L^\infty}+\norm{g}_{L^\infty}$).
\end{lemma}

\begin{proof}[Proof of Lemma \ref{lemCP:EquationForTheDifference}]
	
	We mention the method of \cite[Appendix A]{Silv-2011DifferentiabilityCriticalHJ} applies exactly because the only properties that are truly important are the GCP, translation invariance of $J$, and that $C^{1,\gam}$ functions are classically evaluate-able for $J$.  This first step is a standard use of the inf/sup convolutions, and it shows that in the viscosity sense,
	\begin{align*}
		\partial_t f^\ep\leq J(f^\ep)\ \ \text{and}\ \ \partial_t g_\ep\geq J(g_\ep).
	\end{align*}

	Next, assuming that for some test function, $\phi$, $w^\ep-\phi$ attains a maximum at $(x_0,t_0)$, if follows that $w^\ep$ has $C^{1,1}$ contact, in $x$ and $t$, from above at $(x_0,t_0)$.  However, by definition, both $f^\ep$ and $-g_\ep$ are semi-convex in space and time.  Thus, at $(x_0,t_0)$ we have that $f^\ep, g_\ep, w^\ep\in (\ep^{-1})\text{-}C^{1,1}(x_0)$, which also implies $f^\ep,g_\ep,w^\ep\in(\ep^{-1})\text{-}C^{1,\gam}(x_0)$, and that $\partial_t f^\ep$, $\partial_t g_\ep$, $\grad f^\ep$, $\grad g_\ep$ all exist at $(x_0,t_0)$.  Thus, by Lemma \ref{lemCP:SubSolutionClassicalEval} and Corollary \ref{corNO:PointwiseDefI}, we see that the equations hold classically for $f^\ep$ and $g_\ep$.  Furthermore, by Proposition \ref{propMM:ExtremalOp-JSpecific-AndLip} and Lemma \ref{lemCP:SubSolutionClassicalEval}, we see that classically, (for $m=\norm{f}_{L^\infty}+\norm{g}_{L^\infty}$)
	\begin{align*}
		\partial_t f^\ep-\partial_t g_\ep&\leq \tilde J_{\K(\gam,\del,m\ep^{-1})}(f^\ep,x_0)-\tilde J_{\K(\gam,\del,m\ep^{-1})}(g_\ep,x_0)\\
		&\leq M^+_{J,\K(\gam,\del,m\ep^{-1})}(f^\ep-g_\ep,x_0)
		\leq M^+_{J,\K(\gam,\del,m\ep^{-1})}(\phi,x_0),
	\end{align*}
	where the last inequality resulted from the fact that $M^+_{J,\K(\gam,\del,m\ep^{-1})}$ enjoys the GCP.  Finally, since we must have $\partial_t(f^\ep-g_\ep)=\partial_t\phi$ at $(x_0,t_0)$, we have
	\begin{align*}
		\partial_t \phi(x_0,t_0)\leq M^+_{J,\K(\gam,\del,m\ep^{-1})}(\phi,x_0).
	\end{align*}
	This confirms that $w^\ep$ is a viscosity subsolution of the desired equation.
  
\end{proof}

We will also need to know how the family of operators as in Assumption \ref{assumptionMM} operate on rescaled versions of smooth bump functions.  In particular, it is important to know that the following.

\begin{lemma}\label{lemCP:SmoothBumpMPlusEvaluation}
	If $\Phi$ and $\Phi_R$ are the smooth functions defined as 
	\begin{align*}
		\Phi(x)=\frac{\abs{x}^2}{1+\abs{x}^2},\ \ \text{and}\ \ \Phi_R(x)=\Phi\left(  \frac{x}{R}  \right),
	\end{align*}
	then, for a fixed $J$, given any $\del>0$, $m>\del$, $\rho>0$, there exists $R>1$, with $R=R(J,\rho,\del,m)$, so that
	\begin{align*}
		\sup_{x\in\real^d}M^+_{J,\K(\gam,\del,m)}(\Phi_R,x)\leq \rho.
	\end{align*}
	
\end{lemma}

\begin{proof}[Proof of Lemma \ref{lemCP:SmoothBumpMPlusEvaluation}]
	First, for a fixed $t>1$ to be made precise later, we collect a few facts about $\Phi_R$:
	\begin{align*}
		&0\leq \Phi_R(x)\leq 1\ \ \text{and}\ \ 
		\norm{\grad \Phi_R}_{L^\infty}\leq \frac{C}{R}\\
		&\abs{\Phi_R(x+h)-\Phi_R(x)-\Indicator_{B_1}(h)\grad \Phi_R(x)\cdot h}\Indicator_{B_1}(h)\leq \frac{C}{R^2},\\
		&\abs{\Phi_R(x+h)-\Phi_R(x)-\Indicator_{B_1}(h)\grad \Phi_R(x)\cdot h}\Indicator_{B_t\setminus B_1}(h)\leq
		\frac{Ct}{R}\\
		&\abs{\Phi_R(x+h)-\Phi_R(x)-\Indicator_{B_1}(h)\grad \Phi_R(x)\cdot h}\Indicator_{\real^d\setminus B_t}(h)\leq
		1.
	\end{align*}
	Now, using the uniform estimates on the L\'evy measures, $\mu$, we can first choose $t$ large enough according to Theorem \ref{thmMM:TranslationInvariantMinMax}, so that, uniformly across $[\partial j]_{\K(\gam,\del,m)}$, 
	\begin{align*}
		\int_{\real^d\setminus B_t}\abs{\Phi_R(x+h)-\Phi_R(x)-\Indicator_{B_1}(h)\grad \Phi_R(x)\cdot h}\mu(dh)\leq \frac{\rho}{3}.
	\end{align*}
	Next, given this $t$, we can choose $R$ large enough so that uniformly across $b$ from $[\partial j]_{\K(\gam,\del,m)}$
	\begin{align*}
		b\cdot\grad \Phi_R\leq \frac{\rho}{3}
	\end{align*}
	and
	\begin{align*}
		\int_{B_t}\abs{\Phi_R(x+h)-\Phi_R(x)-\Indicator_{B_1}(h)\grad \Phi_R(x)\cdot h}\mu(dh)\leq \frac{\rho}{3}.
	\end{align*}
	Finally, thanks to Proposition \ref{propMM:ExtremalOperatorActionOnConstants}, we see that for any of the constants, $c$, appearing in (\ref{eqMM:TranslationInvariantLinear}), we have
	\begin{align*}
		c\Phi_R(x)\leq 0.
	\end{align*}
	Hence, by the definition of $M^+_{J,\K(\gam,\del,m)}$ we have obtained the desired inequality.

\end{proof}

The previous lemma allows us to construct a smooth strict supersolution, that also acts a bump function.

\begin{lemma}\label{lemCP:StrictSupersol}
	Given any $\del>0$, $m>\del$, $\rho>0$, and $C>0$ there exists a choice of $R=R(J,\rho,\del,m)$ so that for all $h>0$, 
	\begin{align*}
		\forall\ (x,t)\in\ \real^d\times[0,T],\ \ \Psi(x,t) = C+h\Phi_R(x)+h\rho t,
	\end{align*}
	is a classical, strict super solution of
	\begin{align*}
		\partial_t\Psi>M^+_{J\K(\gam,\del,m)}(\Psi).
	\end{align*}
\end{lemma}

\begin{proof}[Proof of Lemma \ref{lemCP:StrictSupersol}]
	This is a direct calculation, but it is essential to invoke Proposition \ref{propMM:ExtremalOperatorActionOnConstants} to ensure the desired inequality.  In particular, Proposition \ref{propMM:ExtremalOperatorActionOnConstants}, or rather that $c\leq 0$ in the definition of $M^+_{J,\K(\gam,\del,m)}$, show that since $C\geq0$,
	\begin{align*}
		M^+_{J,\K(\gam,\del,m)}(C+h\Phi_R)\leq h M^+_{J,\K(\gam,\del,m)}(\Phi_R).
	\end{align*}
	Thus, with $\rho$ given, we can invoke Lemma \ref{lemCP:SmoothBumpMPlusEvaluation} with $(\rho/2)$ to get the desired inequality.
\end{proof}

\begin{lemma}\label{lemCP:Comparison}
  If $\del>0$, $m>\del$ are fixed, and $w:\mathbb{R}^d\times [0,T]\to \mathbb{R}$ is a bounded, upper semi-continuous function such that, in the viscosity sense
  \begin{align*}
    \partial_t w \leq M^+_{J,\K(\gam,\del,m)}(w),	  
  \end{align*}	  
  then, $\sup \limits_{\mathbb{R}^d\times [0,T]} w_+ \leq \sup \limits_{\mathbb{R}^d}w_+(\cdot,0)$
\end{lemma}

\begin{proof}
	
	First, we will provide the proof when $\sup_{\real^d}w(\cdot,0)\leq 0$.
	Given the result of Lemma \ref{lemCP:StrictSupersol}, we see that the proof of \cite[Lemma 3.3]{Silv-2011DifferentiabilityCriticalHJ} applies directly.  Indeed if we assume, for the sake of contradiction, $\sup_{\real^d\times[0,T]} w>0$, then for some choice of $C\geq0$, and $h>0$, $w-\Psi$ (where $\Psi$ is as in Lemma \ref{lemCP:StrictSupersol}) will attain a global max for $t>0$.   The definition of viscosity solution applied to $\Psi$ as a test function gives a contradiction.
	
	Next, if we do not necessarily assume that $\sup_{\real^d}w(\cdot,0)\leq 0$, we can replace $w$ by the function,
	\begin{align*}
		\tilde w=w-\sup_{\real^d}w(\cdot,0).
	\end{align*}
 	Thus, if $c=\sup_{\real^d}w(\cdot,0)\geq0$, we see that $w=\tilde w+c$, and so in the viscosity sense, $\partial_t w=\partial_t \tilde w$.  Furthermore, since $c\geq 0$, by Assumption \ref{assumptionMM} (vi) and Proposition \ref{propMM:ExtremalOperatorActionOnConstants}, in the viscosity sense,
	\begin{align*}
		M^+_{J,\K(\gam,\del,m)}(w)=M^+_{J,\K(\gam,\del,m)}(\tilde w + c)\leq M^+_{J,\K(\gam,\del,m)}(\tilde w).
	\end{align*}
	Hence, in the viscosity sense,
	\begin{align*}
		\partial_t\tilde w = \partial_t w \leq M^+_{J,\K(\gam,\del,m)}(w)\leq M^+_{J,\K(\gam,\del,m)}(\tilde w),
	\end{align*}
	and the first case is applicable.
	
\end{proof}

\begin{corollary}\label{corCP:OrderingOfSubSupers}
  If $f,g:\mathbb{R}^d\times [0,T] \to \mathbb{R}$ are bounded and, with for some $\del>0$, $f,g\geq\del$, are  respectively a sub and a supersolution, in the viscosity sense, of
\begin{align*}
	\partial_t u=J(u),
\end{align*}
and  $f(x,0)\leq g(x,0)$ for all $x$, then,
  \begin{align*}
    f(x,t)\leq g(x,t)\;\;\forall\;(x,t)\in\mathbb{R}^d\times [0,T].
  \end{align*}
\end{corollary}

\begin{proof}[Proof of Corollary \ref{corCP:OrderingOfSubSupers}]
	If we take the inf/sup convolutions of respectively, $f$ and $g$, then for $w^\ep=f^\ep-g_\ep$ and for each $\ep\in(0,1)$ fixed, Lemma \ref{lemCP:EquationForTheDifference} shows that $w^\ep$ is a viscosity solution of (for $m=\norm{f}_{L^\infty}+\norm{g}_{L^\infty}$)
	\begin{align*}
		\partial_t w^\ep\leq M^+_{J,\K(\gam,\del/2,m\ep^{-1})}(w^\ep).
	\end{align*}
	Hence, Lemma \ref{lemCP:Comparison} shows that for each $\ep$,
	\begin{align*}
		\sup \limits_{\mathbb{R}^d\times [0,T]} f^\ep-g_\ep \leq \sup \limits_{\mathbb{R}^d}(f^\ep-g_\ep)_+(\cdot,0).
	\end{align*}
	Taking $\ep\to 0$ gives the result.
	
\end{proof}

\subsection{Existence}\label{sec:Existence}

The method of existence will, again, follow \cite[Appendix A]{Silv-2011DifferentiabilityCriticalHJ} (which comes from Ishii's adaptation of Perron's method, e.g. \cite{Ishii-1987PerronForHamiltonJacobiDUKE}, \cite{Ishii-1989UniqueViscSolSecondOrderCPAM}, or \cite{CrIsLi-92}), provided  we can construct certain barrier functions.  Indeed, nearly the same barriers from \cite{Silv-2011DifferentiabilityCriticalHJ} will work, with a minor modification.

Assume that $f_0:\real^d\to\real$ is uniformly continuous and $f_0\geq\del$.  First, we construct two smooth functions to serve as lower and upper initial data.  Assume that $b:\real^d\to\real$, $b$ is smooth, $0\leq b\leq 1$, $b(0)=1$, and $\text{support}(b)\subset B_1$.  

Fix $x_0$.  We will produce lower and upper barriers for $f_0$ at the point $x_0$.  Define, for $\rho>0$,
\begin{align*}
	b_\rho(x):= b\left( \frac{x-x_0}{\rho} \right).
\end{align*}
We claim that given any $\ep>0$, we can produce a choice of $\rho$ so that the following functions will trap $f_0$.  We define the upper and lower barrier's for $f_0$ at $x_0$ as
\begin{align*}
	U_0(x):= b_\rho(x)f_0(x_0) + \left( 1-b_\rho(x)  \right)\sup_{\real^d}(f_0) + \ep\\
	L_0(x):= b_\rho(x)f_0(x_0) + \left( 1-b_\rho(x)  \right)\inf_{\real^d}(f_0) -\ep.
\end{align*}
Specifically, given $\ep>0$, we can find a $\rho>0$ so that
\begin{align*}
	\forall\ x\in\real^d\ \ U_0(x)\geq f_0(x)\ \ \text{and}\ \ L_0(x)\leq f_0(x).
\end{align*}

With these barriers in hand, we see that Perron's method for viscosity solutions applies to the equation (\ref{eqCP:ParabolicMain}).  We summarize this as

\begin{theorem}\label{thmCP:Existence}
	If $f_0$ is uniformly continuous and bounded on $\real^d$, then there exists a unique viscosity solution, $f$, that solves (\ref{eqCP:ParabolicMain}).
\end{theorem}


\subsection{Comparison using the free boundary operators directly}

In this part, we show how the previous results indeed do work, with basic modifications, under the assumptions of the free boundary operators, $I$ and $H$, above.  These arguments will be valid regardless of whether or not one has proved that $I$ has an integro-differential representation as in Section \ref{sec:MinMax}. 

\begin{theorem}\label{thmCP:ComparisonOnFBOperatorDirectly}
	  If $I$ is as in (\ref{eqNO:DefOfI}) or $H$ as in (\ref{eq2Ph:TwoPhaseOperatorDefinition}), $f,g:\mathbb{R}^d\times [0,T] \to \mathbb{R}$ are bounded and, for some $\del>0$, $f,g\geq\del$, and for $H$ also satisfy $f,g\leq L-\del$, are  respectively a sub and a supersolution, in the viscosity sense, of
	\begin{align*}
		\partial_t u=G(I(u))\cdot\sqrt{1+\abs{\grad u}^2},\ \ \text{or}\ \ 
		\partial_t u=H(u)\cdot\sqrt{1+\abs{\grad u}^2}
	\end{align*}
	and  $f(x,0)\leq g(x,0)$ for all $x$, then,
	  \begin{align*}
	    f(x,t)\leq g(x,t)\;\;\forall\;(x,t)\in\mathbb{R}^d\times [0,T].
	  \end{align*}
\end{theorem}

The proof procedes nearly identically to the arguments in Lemmas \ref{lemCP:EquationForTheDifference} and \ref{lemCP:Comparison}, except that we don't invoke the extremal operator, $M^+_{J,\K(\gam,\del,m)}$, and we use Corollary \ref{corNO:BumpFunctionSmallI} instead of Lemma \ref{lemCP:StrictSupersol}.

First we will demonstrate the case for the equation for $I$.
To this end, we again take $f^\ep$ and $g_\ep$ to be the sup-convolution and inf-convolution of $f$ and $g$.  We assume, for the sake of contradiction, that $\sup_{\real^d\times[0,T]}f^\ep-g^\ep>0$.  As above, this means that for some $C\geq0$, $\rho>0$, and $h>0$, there will exist $\hat x$ and $\hat t>0$ such that $f^\ep-g_\ep-\Psi$ attains a zero maximum at $(\hat x,\hat t)$ ($\Psi$ is taken to be as in Lemma \ref{lemCP:StrictSupersol}).   Furthermore, we can choose the parameters $C,\rho,h$ so that we also maintain
\begin{align*}
	\del\leq g_\ep + \Psi,\ \ \text{and}\ \ f+\Psi\leq L-\frac{\del}{2},
\end{align*}
which is a technical requirement for applying Corollary \ref{corNO:BumpFunctionSmallI}; the restriction relative to $f^\ep$ and $\Psi$ is only relevant for the case of the operator, $H$, which we treat second.

This means that at $(\hat x, \hat t)$, both $f^\ep$ and $g_\ep$ have classical derivatives with respect to $t$ at $\hat t$ and also that $I(f^\ep,\hat x)$, $I(g_\ep,\hat x)$ are classically defined.  We can use the Corollary \ref{corNO:BumpFunctionSmallI} to get a contradiction.  Indeed, assume that $s>0$ is generic.  We see that by the Lipschitz nature of $G$ and $\sqrt{1+\abs{\grad f}}$,
\begin{align*}
	&G(I(g_\ep+\Psi))\sqrt{1+\abs{\grad g_\ep+\grad\Psi}^2}
	-G(I(g_\ep))\sqrt{1+\abs{\grad g_\ep}^2}\\
	&\ \ \leq C\norm{\left(\sqrt{1+\abs{\grad g_\ep+\grad\Psi}^2}\right)}_{L^\infty}\cdot
	\abs{I(g_\ep+\Psi)-I(g_\ep)}
	+C\norm{G(I(g_\ep))}_{L^\infty}\abs{\grad \Psi}
\end{align*}  
Hence for a sufficiently large $R$, using the definition of $\Psi$ and Corollary \ref{corNO:BumpFunctionSmallI}, we can force the right hand side to be controlled by $s$.  We note that this choice of $R$ will depend upon $\ep$, $\del$, and $m$, which is suitable since $\ep$ is fixed (as in the proof of Corollary \ref{corCP:OrderingOfSubSupers}).  Thus, at $(\hat x, \hat t)$ we first use the GCP, followed by the previous inequality to obtain
\begin{align*}
	\partial_t(g_\ep+\Psi)\leq\partial_t f^\ep &\leq G(I(f^\ep,\hat x))\sqrt{1+\abs{\grad f^\ep(\hat x)}^2}
	\leq G(I(g_\ep+\Psi,\hat x))\sqrt{1+\abs{\grad g_\ep(\hat x)+\grad \Psi(\hat x)}^2}\\
	&\leq G(I(g_\ep,\hat x))\sqrt{1+\abs{\grad g_\ep(\hat x)}^2}+s \leq \partial_t g_\ep + s.
\end{align*}
(Here, we note the first inequality is only to include the case that $\hat t=T$, otherwise if $\hat t<T$, we would have equality.)
In other words, this implies that
\begin{align*}
	h\rho\leq s.
\end{align*}
Since any $s>0$ is admissible, once $\rho$ and $h$ have been fixed, we see, by choosing $s=\frac{1}{2}h\rho$, that this gives a contradiction.

For the case of $H$, we would give nearly the same proof, except that we need to use Corollary \ref{corNO:BumpFunctionSmallI} twice, and we follow the argument for Theorem \ref{thm2Ph:TwoPhaseOperatorLipschitz}.  Here is a sketch of the argument, following as above, and the notation for $\tilde f$, $\tilde F_2$, etc... are as that preceding Lemma \ref{lem2Ph:IminusUtildeIFtildeEquality}.  The term we need to focus on is
\begin{align*}
	H(f^\ep,\hat x)-H(g_\ep,\hat x).
\end{align*}
Indeed, we have
\begin{align*}
H(f^\ep,\hat x) - H(g_\ep,\hat x)
	\leq \Lam_0(I^+(f^\ep,\hat x) - I^+(g_\ep,\hat x))_+ 
	+ \Lam_0(-I_{\tilde F_2}(L-\tilde g_\ep,\hat x) + I_{\tilde F_2}(L-\tilde f^\ep, \hat x))_+.
\end{align*}
At this point, the first term is handled exactly as above, using the fact that $f^\ep$ is touched from above at $\hat x$ by the function $g^\ep+\Psi$.  For the second term, we use the fact that our assumptions on $f^\ep$, $g_\ep$, and $\Psi$, are such that 
\begin{align*}
	(L-\tilde f^\ep) + (-\tilde\Psi)\ \ \text{is touched from below by}\ \ (L-\tilde g_\ep)\ \text{at}\ \hat x.
\end{align*}
Furthermore, the requirement that $f+\Psi\leq L-\frac{\del}{2}$ ensures that both $(L-\tilde f^\ep)$ and $(L-\tilde f^\ep)+(-\tilde \Psi)$ are in $\K^{*}(\gam,\del/2,m\ep)$ (as $f^\ep$ is semi-convex, so $-f^\ep$ is semi-concave).  By the GCP,
\begin{align*}
	I_{\tilde F_2}((L-\tilde f^\ep)+(-\tilde\Psi),\hat x)\geq I_{\tilde F_2}(L-\tilde g,\hat x),
\end{align*}
and by Corollary \ref{corNO:BumpFunctionSmallI},
\begin{align*}
	I_{\tilde F_2}(L-\tilde f^\ep,\hat x) - I_{\tilde F_2}((L-\tilde f^\ep)+(-\Psi),\hat x)\leq s.
\end{align*}
These two inequalities are enough to finish the proof as in the first case.


\section{Free boundary viscosity solutions}\label{sec:WeakSolutions}

In this section we study various notions of viscosity solutions of free boundary problems like \eqref{eqIn:MainFBEvolution}. Because we have in mind a one-to-one correspondence between viscosity solutions of the free boundary problem and viscosity solutions of a parabolic equation, we will introduce a definition of viscosity solution that is different from the standard one -- as presented in, e.g.  \cite{AthanaCaffarelliSalsa-1996RegFBParabolicPhaseTransitionACTA}, \cite{Caffarelli-1988HarnackApproachFBPart3PISA}, \cite{Kim-2003UniquenessAndExistenceHeleShawStefanARMA}, \cite{KimPozar-2011ViscSolTwoPhaseStefanCPDE}.  We will then show that our definition contains the existing ones when the assumptions overlap.

\subsection{Viscosity solutions for the free boundary evolution (\ref{eqIn:MainFBEvolution})} By a \emph{test interface in $[a,b]$}, we shall mean the following: a smooth hypersurface in $S \subset \Sigma_0 \times [a,b]$ such that each time slice $S(t)$ separates $\Sigma_0$ ($\Sigma_0 = \mathbb{R}^{d+1}_+$ or $\Sigma_0 = \mathbb{R}^d\times [0,L]$) into two connected components, that is $\Sigma_0 \setminus S(t) = S(t)^+ \cup S(t)^-$ where $S(t)^+$ and $S(t)^-$ are open sets, $S(t)$ is a positive distance away from $\partial \Sigma_0$, and
\begin{align*}
  \begin{array}{lr} \Gamma_0 \subset \overline{S(t)^+} & \textnormal{(one phase case),}\\
  \Gamma_0 \subset \overline{S(t)^+},\; \Gamma_L \subset \overline{S(t)^-} & \textnormal{(two phase case)}.
  \end{array}
\end{align*}
Given a test interface $S$ in $[a,b]$, we define $U_S:\Sigma_0 \times [a,b]\to \mathbb{R}$ as the function which for each time $t\in [a,b]$ is the unique solution to one of the following Dirichlet problems; in the two-phase case ($\Sigma_0 = \mathbb{R}^d\times (0,L)$), it solves the equation,
\begin{align}\label{eqnWeak:U sub S Dirichlet problem two phase}
  \left \{ \begin{array}{rl}
    F_1(D^2U_S,\grad U_S) & = 0 \textnormal{ in } S^+,\\
    F_2(D^2U_S,\grad U_S) & = 0 \textnormal{ in } S^-,\\
    U_S & = 0 \textnormal{ on } S,\\
    U_S & = -1 \textnormal{ on } \Gamma_L,\\	
    U_S & = 1 \textnormal{ on } \Gamma_0,	\end{array}\right.
\end{align}
and in the one-phase case $(\Sigma_0 = \mathbb{R}^{d+1}_+)$, it solves the equation,
\begin{align}\label{eqnWeak:U sub S Dirichlet problem one phase}
  \left \{ \begin{array}{rl}
    F(D^2U_S,\grad U_S) & = 0 \textnormal{ in } S^+,\\
    U_S & = 0 \textnormal{ on } S,\\
    U_S & = 1 \textnormal{ on } \Gamma_0,	\end{array}\right.
\end{align}
Let us recall the definition of classical subsolutions and supersolutions.
\begin{definition}\label{def:weak solutions Classical subsolution definition}
  Let $U:\Sigma_0 \times [a,b]\to \mathbb{R}$. This function is said to be a classical subsolution  (respectively supersolution) of (\ref{eqIn:MainFBEvolution}) if
  \begin{enumerate}
    \item The set $\partial\{ U>0\}$ is a codimension $1$ differentiable submanifold of $\Sigma_0 \times [a,b]$ such that each time slice $\partial \{ U>0\} \cap \Sigma_0 \times \{t\}$ ($a\leq t\leq b$) is a codimension $1$ differentiable submanifold of $\Sigma_0$, and $U$ is twice differentiable in space and differentiable in time in $\Sigma_0 \times [a,b]\setminus \{U\neq 0\}$.
    \item For every fixed $t$, the function $U= U(\cdot,t)$ solves, in the viscosity sense,
    \begin{align*}
      F_1(D^2U,\nabla U) & \geq 0 \textnormal{ in } \{U>0\} \;\; (\textnormal{resp.} \leq),\\
      F_2(D^2U,\nabla U) & \geq 0 \textnormal{ in } \{U<0\}\;\; (\textnormal{resp.} \leq),	  	  		
    \end{align*}		
    and if $V$ denotes the normal velocity of $\partial\{U>0\}$ (in the outer normal direction), then
    \begin{align*}
      V \leq G(\partial^+_\nu U,\partial^-_\nu U) \textnormal{ along } \partial\{U>0\} \;\; (\textnormal{resp.} \geq).		
    \end{align*}			
  \end{enumerate}
  Furthermore, $U$ is said to be a classical solution if it is both a subsolution and a supersolution.
  
  The definition for the one-phase is entirely analogous and we omit it.
  
\end{definition}

\begin{definition}\label{defWeak:Touching from above definition}
  A function $U$ is said to be touched from above at $(X_0,t_0)$ by $S$ if $S$ is a test interface in $[t_0-\tau,t_0]$ for some $\tau>0$ such that
  \begin{align*}
    \{ U>0\} \cap \{ t_0-\tau \leq t\leq t_0\} \subset \overline{S^+} \textnormal{ and } (X_0,t_0) \in \partial\{U>0\} \cap  S.
  \end{align*}
\end{definition}
The notion of viscosity solution we will be using is the following.
\begin{definition}\label{defWeak:subsolution weak definition}
  A viscosity subsolution (respectively supersolution) of (\ref{eqIn:MainFBEvolution}) in $[a,b]$ is an upper semicontinuous function (resp. lower semi continuous function)
  \begin{align*}
   U:\Sigma_0 \times[a,b]\to \mathbb{R},
  \end{align*}
  which is required to have the following properties.  It satisfies the pointwise bounds
  \begin{align*}
   & U\leq 1 \;(\textnormal{resp. } U\geq 1) \textnormal{ on } \Gamma_0, U \leq -1 ;\ (\textnormal{resp. } U\geq -1) \textnormal{ on } \Gamma_L\  \textnormal{(two-phase case)},\\
   & U\leq 1 \;(\textnormal{resp. } U\geq 1) \textnormal{ on } \Gamma_0\ \textnormal{(one-phase case)};
  \end{align*} 
  and satisfies the following relations in the viscosity sense: in the two-phase case they are
  \begin{align*}
    F_1(D^2U,\nabla U) & \geq 0 \textnormal{ in } \{U>0\}^0 \;\; (\textnormal{resp.} \leq),\\
    F_2(D^2U,\nabla U) & \geq 0 \textnormal{ in } \{U<0\}^0 \;\; (\textnormal{resp.} \leq),	
  \end{align*}
  and in the one-phase case they are
  \begin{align*}
    F(D^2U,\nabla U) & \geq 0 \textnormal{ in } \{U>0\}^0 \;\; (\textnormal{resp.} \leq).
  \end{align*}
  Last but not least,  for any test interface $S$ touching $U$ from above at $(X_0,t_0) \in \partial\{U>0\}$ we have
  \begin{align*}
    V_{S}(X_0,t_0) & \leq G(\partial^+_\nu U_S,\partial^-_\nu U_S)(X_0,t_0)\;\; (\textnormal{resp.} \geq), \textnormal{(two-phase case)},\\
    V_{S}(X_0,t_0) & \leq G(\partial^+_\nu U_S)(X_0,t_0)\;\; (\textnormal{resp.} \geq), \textnormal{(one-phase case)}.
  \end{align*}	  
\end{definition}
It should be useful to discuss the notion of viscosity solutions between what we have proposed above and the definition given by Kim in \cite{Kim-2003UniquenessAndExistenceHeleShawStefanARMA} for the one-phase Hele-Shaw problem.  Here we recall the definition in \cite{Kim-2003UniquenessAndExistenceHeleShawStefanARMA}. 

\begin{definition}\label{defWeak:Kim notion of viscosity solution}
  A non-negative upper semicontinuous function $U:\Sigma_0 \times [0,\infty)\to \mathbb{R}$ is said to be a viscosity subsolution to the one-phase Hele-Shaw problem, if	
\begin{itemize}	
  \item[\textnormal{(i)}] $U=U_0$ at $t=0$, $U\leq 1$ on $\Gamma_0$;
  \item[\textnormal{(ii)}] $\overline{\{U>0\}} \cap \{t=0\} = \overline{ \{ U(X,0)>0\}}$;
  \item[\textnormal{(iii)}] for each $T\geq 0$ the set $\overline{\{U>0\}}\cap \{t\leq T\}$ is bounded and
  \item[\textnormal{(iv)}] for every $\Phi \in C^{2,1}_{x,t}$ that has a local maximum of
    \begin{align*}
      U-\Phi \textnormal{ in } \overline{\{U>0\}} \cap \{ t\leq t_0\} \cap Q \textnormal{ at } (X_0,t_0),		
    \end{align*}
    \begin{itemize}
      \item[\textnormal{(a)}] $-\Delta \Phi(x_0,t_0)\leq 0$ if $U(X_0,t_0)>0$,
      \item[\textnormal{(b)}] $\textnormal{min}(-\Delta \Phi,\partial_t \Phi-|\nabla \Phi|^2)(X_0,t_0) \leq 0$, if $(X_0,t_0) \in \partial\{U>0\}$, $U(X_0,t_0)=0$. 
    \end{itemize}				
\end{itemize}
    A non-negative lower semicontinuous function $U:\Sigma_0 \times [0,\infty)\to \mathbb{R}$ is said to be a viscosity supersolution to the one-phase Hele-Shaw problem, if	
\begin{itemize}	
  \item[\textnormal{(i)}] $U=U_0$ at $t=0$, $U\geq 1$ on $\Gamma_0$;
  \item[\textnormal{(ii)}] for every $\Phi \in C^{2,1}_{X,t}$ that has a local minimum of
    \begin{align*}
      U-\Phi \textnormal{ in } \overline{\{U>0\}} \cap \{ t\leq t_0\} \cap Q \textnormal{ at } (x_0,t_0),		
    \end{align*}
    \begin{itemize}
      \item[\textnormal{(a)}] $-\Delta \Phi(x_0,t_0)\geq 0$ if $(X_0,t_0) \in \{U>0\}$,
      \item[\textnormal{(b)}]  $\textnormal{max}(-\Delta \Phi,\partial_t \Phi-|\nabla \Phi|^2)(X_0,t_0) \geq 0$, if $(X_0,t_0) \in \partial \{U>0\}$, $|\nabla \Phi(X_0,t_0)|\neq 0$ and
      \begin{align*}
        \{\Phi>0\}\cap \{U>0\} \cap B \neq 0		  
      \end{align*}		  
      for any ball $B$ centered at $(X_0,t_0)$.
    \end{itemize}				
\end{itemize}
Then, $U$ is a viscosity solution if it is both a viscosity supersolution and its upper semicontinuous envelope $U_*$ is a viscosity subsolution.

\end{definition}

\begin{lemma}\label{lemWeak:weak subsols S are weak subsols classic sense}
  For the one-phase Hele-Shaw problem, if $U$ is a viscosity subsolution (supersolution) in the sense of Definition \ref{defWeak:Kim notion of viscosity solution} then it is a viscosity subsolution (supersolution) in the sense of Definition \ref{defWeak:subsolution weak definition}.

\end{lemma}

\begin{proof}
  Let $U$ be a solution in the sense of Definition \ref{defWeak:Kim notion of viscosity solution}. Fix $(X_0,t_0) \in \partial \{U>0\}$ and let $S$ be a test interface in $[t_0-\tau,t_0]$ touching $U$ from above at $(X_0,t_0)$, for some small $\tau>0$. Let $U_S$ be as given by \eqref{eqnWeak:U sub S Dirichlet problem one phase} with $F(D^2U,\nabla U) = \Delta U$.
  
  Fix $\varepsilon>0$ and let $\tilde \Phi_\varepsilon:\Sigma_0 \times [t_0-\tau,t_0]$ be the function given by
\begin{align*}
  \left \{ \begin{array}{rl}
    \Delta \tilde \Phi_\varepsilon + \varepsilon & = 0 \textnormal{ in } S^+,\\
    \tilde \Phi_\varepsilon & = 0 \textnormal{ on } S,\\
    \tilde \Phi_\varepsilon & = 1 \textnormal{ on } \Gamma_0.	\end{array}\right.
\end{align*}
  In particular, $\tilde \Phi_{\varepsilon}$ is superharmonic in $S^+$ and that the following pointwise limits hold
  \begin{align*}
    \lim \limits_{\varepsilon\to 0}\tilde \Phi_{\varepsilon} = U_S \textnormal{ in } S^+,\;\; \lim\limits_{\varepsilon \to 0} \partial^+_n \tilde \Phi_\varepsilon = \partial^+_n U_S \textnormal{ on } S.
  \end{align*}
  On the other hand, it is clear that $\tilde \Phi_{\varepsilon}$ is $C^{2,1}$ in $S^+$. Let $\Phi_{\varepsilon}$ be a smooth extension of $\tilde \Phi_{\varepsilon}$ to $\Sigma_0 \times [t_0-\tau,t_0]$.
  
  By construction, for any $\varepsilon$ we have that $U-\Phi_\varepsilon$ has a local maximum at some point $(X_0,t_0) \in S$ (seen as a function in $\{U>0\} \cap \{t\leq t_0\}$). Applying Definition \ref{defWeak:Kim notion of viscosity solution}, it follows that at $(X_0,t_0)$ at least one of the following inequalities must hold
  \begin{align*} 
    \Delta \Phi_\varepsilon \geq 0\ \  \textnormal{or}\ \  \partial_t \Phi_\varepsilon \leq |\nabla \Phi_\varepsilon|^2. 
  \end{align*}
  Since $\Delta \Phi = -\varepsilon <0$ along $S$, we conclude that
  \begin{align*}
    \partial_t \Phi_\varepsilon \leq |\nabla \Phi_\varepsilon|^2 \textnormal{ at } (X_0,t_0),\ \ \text{and hence}\ \  V_{S}(X_0,t_0) \leq \partial^+_n\Phi_\varepsilon(X_0,t_0).
  \end{align*}
  Letting $\varepsilon \to 0$, we conclude that
  \begin{align*}
    V_{S}(X_0,t_0) \leq \partial^+_nU_S(X_0,t_0).
  \end{align*}  
  and we conclude that $U$ is a solution in the sense of Definition \ref{defWeak:subsolution weak definition}.
\end{proof}
\begin{rem}
  The use of a slightly superharmonic function $\Phi_\varepsilon$ instead of the harmonic function $U_S$ is a type of argument familiar from the theory of viscosity solutions when dealing with boundary conditions, cf. \cite[Proposition A.2]{GuSc-2014NeumannHomogPart1DCDS-A} where the Neumann condition for a non-linear problem is studied. A similar situation is seen in free boundary problems, see for instance the discussion following \cite[Definition 2.5]{DeSilva-2011FBRegWithRHS-IFBs}.
\end{rem}

It will be necessary to understand how $\partial^{\pm}_nU_S$ varies with vertical shifts of $S$, and accordingly we introduce some notation for such shifts. Given a test interface $S$ and $h \in \mathbb{R}$, we define
\begin{align}\label{eqnWeak:Surface translation}
  S_h := \{ X=(x,x_{d+1}) \;\mid\;  (x,x_{d+1}-h)\in S \}. 
\end{align}
In other words $S_h$ is the surface resulting from shifting $S$ in the upward direction by $h$ (if $h<0$, then $S_h$ is $S$ shifted down by $|h|$). The following Lemma is, essentially, an extension of Lemma \ref{lemNO:IDependenceOnConstants} with a slightly different set  up that will be convenient in the next subsection.

\begin{lemma}\label{lemWeak:test interface translation}
  Let $S$ be a test interface. There is a constant $C=C_S$ such that for all sufficiently small $h$, and any $(X,t) \in S_h$ ($S_h$ as in \eqref{eqnWeak:Surface translation}), we have
  \begin{align*}
    |\partial^+_n U_{S_h}(X-he_{d+1},t)-\partial^+_n U_S(X,t)|  \leq C|h|.
  \end{align*}	  
  Here, $U_S$ and $U_{S_h}$ are the functions given by \eqref{eqnWeak:U sub S Dirichlet problem two phase}.
\end{lemma}

\begin{rem}\label{remWeak:test interface translation two phase}
  Using an argument analogous to that in Lemma \ref{lem2Ph:IminusUtildeIFtildeEquality}, one can use the one-phase problem estimate in Lemma \ref{lemWeak:test interface translation} to obtain the analogous bound for the two-phase problem 
  \begin{align*}
     |G(\partial^+_nU_{S_h},\partial^-_nU_{S_h})-G(\partial^+_nU_{S},\partial^-_nU_{S})| \leq C_{G,S}|h|.	  
  \end{align*}	  

\end{rem}

\begin{proof}
  We shall consider only $h$ such that $|h|$ is no larger than half the distance from $S$ to $\{x_{d+1}=0\}$, and we focus solely on the case $h>0$ (the proof is similar for $h<0$). As in Lemma \ref{lemNO:IDependenceOnConstants}, we shall compare $U_S$ to a vertical shift of $U_{S_h}$. Thus, we consider $\tilde U$ defined by
  \begin{align*}
    \tilde U(X,t) := U_{S_h}(X-he_{d+1},t),\;\; \textnormal{ defined in } S^+ \cap \{ (X,t) \mid X=(x,x_{d+1}),\;x_{d+1}\geq h\}.
  \end{align*}
  Observe that $\tilde U = U_S = 0 \textnormal{ on } S$. On the other hand, taking into account that $\tilde U \equiv 1$ on $\{x_{d+1}=h\}$, that $U_S \equiv 1$ on $\{x_{d+1}=0\}$, and the available Lipschitz bound for $U_S$ in $\{0\leq x_{d+1}\leq h\}$, we have
  \begin{align*}
    |\tilde U - U_S| \leq C_Sh \textnormal{ on } \{ x_{d+1}=h \}.
  \end{align*}
  Then, applying the maximum principle in $S^+\setminus \{0\leq x_{d+1}\leq h\}$, it follows that
  \begin{align*} 
    |\tilde U-U_S| \leq C_Sh \textnormal{ in } S^+\setminus \{0\leq x_{d+1}\leq h\}.	   
  \end{align*}	  
  Then, using that $\tilde U-U_S= 0$ on $S$ and applying Proposition \ref{propNO:SemiconcavityImpliesUfLipschitzAndBoundsOnNormalDeriv}, it follows that for some constant $C$ depending on $S$ (as well as the dimension and the ellipticity of $F$),
  \begin{align*}
    |\partial^+_n (\tilde U-U_S)| \leq Ch,\;\textnormal{ on } S.
  \end{align*}	  
  Since $\partial^+_n \tilde U(X,t) = \partial^+_n U_{S_h}(X-he_{d+1},t)$, the Lemma is proved.
\end{proof}

\subsection{Correspondence between viscosity solutions of the free boundary evolution and viscosity solutions of the parabolic equation}  

In this subsection, we are going show that under the graph assumption for $\partial\{U>0\}$, viscosity solutions for the free boundary problem correspond to viscosity solutions for the parabolic equations.

First, we show a viscosity solution to the parabolic equation yields a viscosity solution to the free boundary problem (in the sense of Definition \ref{defWeak:subsolution weak definition}).  We recall for the sake of convenience that in this level set context, when $U_f$ solves (\ref{eqIn:MainFBEvolution}), the equation for the normal velocity for $U_f$ becomes
\begin{align*}
	\partial_t U_f = G(\partial_n^+ U_f, \partial_n^- U_f)\abs{\grad U_f};
\end{align*}
and for the function $f$, it is equivalent to
\begin{align}\label{eqWeak:ParabolicEquation}
	\partial_t f = G(I^+(f),I^-(f))\cdot\sqrt{1+\abs{\grad f}^2}.
\end{align}
(We recall that $I^\pm$ are defined in Section \ref{sec:TwoPhase})

\begin{lemma}\label{lemWeak:Lipschitz f visc sol implies U_f sol}
  If $f$ is a globally Lipschitz viscosity subsolution (respectively supersolution)  of the equation (\ref{eqWeak:ParabolicEquation}) in $\mathbb{R}^d\times [a,b]$, then $U_f$ is a viscosity subsolution  (resp. supersolution) in $\Sigma_0 \times [a,b]$ of the evolution (\ref{eqIn:MainFBEvolution}).
\end{lemma}

\begin{proof}
Let $(X_0,t_0) \in \Gamma_f$ and let $S$ be a test interface in $[t_0-\delta,t_0]$ touching $U_f$ from above at $(X_0,t_0)$.  We first need to reduce to the case in which there is an intermediate test interface that is a \emph{graph} of some function over $\real^d$. We claim there is some $\phi$ such that $U_f\leq U_\phi \leq U_S$. For $\varepsilon>0$ define $\phi^\varepsilon$
  \begin{align*}
    \phi^\varepsilon(x,t) & := \sup \Big \{ P_{h,x',t'}(x,t) \mid (h,x',t') \textnormal{ s.t. } D_{P_{h,x',t'}} \subset \{U_S>0\} \Big \}.	
  \end{align*}
  where
  \begin{align*}
    P_{h,x',t'}(x,t) & := h-\tfrac{1}{2\varepsilon}|x-x'|^2 - \tfrac{1}{2\varepsilon}|t-t'|^2.
  \end{align*}
  (For convenience, we recall that the notation for the set $D_{P_{h,x',t'}}$ appears in (\ref{eqIn:DefOfDomainD}).)
  It is clear that $D_{\phi^\varepsilon}\subset \{U_S>0\}$. Therefore, applying the comparison principle yields
  \begin{align*}
    U_{\phi^\varepsilon} \leq U_S \textnormal{ everywhere}, \textnormal{ and } \;U_{\phi^\varepsilon} = U_S \textnormal{ on } \Gamma_{\phi^\varepsilon} \cap S.
  \end{align*}	  
  The interface $\partial\{U_S>0\}$ is smooth and touches $\Gamma_f$ at $(X_0,t_0)$, since $f$ is Lipschitz, the normal to $\partial \{U_S>0\}$ cannot be orthogonal to $e_{d+1}$. This means that if $\varepsilon$ is chosen sufficiently small (depending on the smoothness of $S$) then there is some $x_0 \in \mathbb{R}^d$ such that
  \begin{align*} 
    (X_0,t_0) \in \Gamma_{\phi^\varepsilon} \cap \partial \{U_S>0\}, \textnormal{ and } X_0 = (x_0, \phi^\varepsilon(x_0,t_0)).
  \end{align*}	  
   In particular, for such $\varepsilon$ we have $G(\partial^+_n U_S,\partial^-_n U_S) \geq G(\partial^+_n U_{\phi^\varepsilon},\partial^-_n U_{\phi^\varepsilon})$ and $V^\varepsilon \geq V$ at this contact point (where $V^\ep$ is the normal velocity of $\partial\{ U_{\phi^\ep}>0  \}$). Moreover, from the construction of $\phi^\varepsilon$ and the fact that $U_f\leq U_S$, it follows that
  \begin{align*}
    f(x,t) \leq \phi^{\varepsilon}(x,t) \textnormal{ and } f(x_0,t_0) = \phi^{\varepsilon}(x_0,t_0).
  \end{align*}	   
  Since $\phi^{\varepsilon}$ is pointwise $C^{1,1}$ at $(x_0,t_0)$ and $f$ is a viscosity subsolution of (\ref{eqWeak:ParabolicEquation}), it follows that (cf. Lemma \ref{lemCP:SubSolutionClassicalEval})
  \begin{align*}
    \partial_t \phi^{\varepsilon} \leq 	G(I^+(\phi^{\varepsilon}),I^-(\phi^\ep))\cdot\sqrt{1+|\nabla \phi^{\varepsilon}|^2},\;\textnormal{ at } (x_0,t_0).    
  \end{align*}
  Recalling the level set formulation as it pertains to the defining function that is $\Phi(X,t)=f(x,t)-x_{d+1}$, we see that as the set $\partial\{ U_{\phi^\ep}>0 \}=\Gam_{\phi^\ep}$, the normal velocity $V^\ep$ is
\begin{align*}
	\frac{\partial_t \phi^\ep}{\sqrt{1+\abs{\grad \phi^\ep}^2}}=V^\ep.
\end{align*}
  In other words, in light of the definitions of $I^\pm$, we see that
\begin{align*}
    V^{\varepsilon} \leq G(\partial^+_n U_{\phi^\varepsilon}, \partial^-_n U_{\phi^\varepsilon}) \textnormal{ at } (X_0,t_0).  \end{align*}
  After combining all of the above arguments, we have established that
  \begin{align*}
    V\leq V^\ep \leq G(\partial^+_n U_{\phi^\varepsilon}, \partial^-_n U_{\phi^\varepsilon})
	\leq  G(\partial^+_n U_S, \partial^-_n U_S)\ \textnormal{at}\ (X_0,t_0).  
  \end{align*}
  This shows that $U_f$ is a viscosity subsolution of the free boundary flow.

  If $f$ is a supersolution, the argument is similar, so we only highlight some of the steps: suppose now the test interface $S$ touches $U_f$ from below at a free boundary point $(X_0,t_0)$, then for $\varepsilon>0$ we define $\phi_\varepsilon$ by
  \begin{align*}
    \phi_{\varepsilon}(x,t) & := \inf \Big \{ Q_{h,x',t'}(x,t)  \mid (h,x',t') \textnormal{ s.t. } \{U_S>0\} \subset D_{Q_{h,x',t'}} \Big \},
  \end{align*}
  where
  \begin{align*}
    Q_{h,x',t'}(x,t) & := h+\tfrac{1}{2\varepsilon}|x-x'|^2 + \tfrac{1}{2\varepsilon}|t-t'|^2.
  \end{align*}
  As before, from the construction of $\phi_{\varepsilon}$ one can see that $\partial \{U_S>0\} \subset D_{\phi_\varepsilon}$. The comparison principle then says that
  \begin{align*}
    U_S \leq U_{\phi_\varepsilon} \textnormal{ everywhere, and } U_{\phi_\varepsilon} = U_S \textnormal{ on } \Gamma_{\phi_\varepsilon} \cap S.
  \end{align*}
  The fact that $f$ is Lipschitz and the smothness $\partial\{U_S>0\}$ means that the normal to $\partial\{U_S>0\}$ at the contact point cannot be perpendicular to $e_{d+1}$. It follows that for all sufficiently small $\varepsilon$ there is a $x_0 \in \mathbb{R}^d$ such that $X_0 = (x_0,\phi_\varepsilon(x_0,t_0))$ and
  \begin{align*}
    \phi_\varepsilon(x,t) \leq  f(x,t) \textnormal{ and } \phi_\varepsilon(x_0,t_0) = f(x_0,t_0).
  \end{align*}
  Arguing as in the case of a subsolution, one arrives at
  \begin{align*}
    V \geq G(\partial^+_n U_S, \partial^-_n U_S) \textnormal{ at } (X_0,t_0).  
  \end{align*}

\end{proof}

\begin{lemma}\label{lemWeak:ParabolicViscSolImpliesFBViscSol}
  If $f$ is a viscosity subsolution (respectively supersolution) of (\ref{eqWeak:ParabolicEquation}) in $\mathbb{R}^d \times [a,b]$, then $U_f$ is a viscosity subsolution (resp. supersolution) of (\ref{eqIn:MainFBEvolution}) in $\Sigma_0 \times [a,b]$.
  
\end{lemma}

\begin{proof}
  Let $f$ be a viscosity subsolution of the problem in $\mathbb{R}^d\times [a,b]$. Let $\varepsilon>0$ be appropriately small and define the sup-convolution $f^{\varepsilon}$; we see $f^\ep$ is still a viscosity subsolution and it is also a Lipschitz continuous function. By Lemma \ref{lemWeak:Lipschitz f visc sol implies U_f sol}, it follows that $U_{f^\varepsilon}$ is a viscosity subsolution of the free boundary problem for every $\varepsilon>0$.

  Let $S$ be a test interface touching $U_f$ from above at $(X_0,t_0)$, $X_0=(x_0,f(x_0,t_0))$. For each sufficiently small $\varepsilon>0$ there is $h_\varepsilon \in \mathbb{R}$ such that (see \eqref{eqnWeak:Surface translation}) $S_{h_\varepsilon}$ touches $f_\varepsilon$ from above at a point of the form $(X_\varepsilon,t_\varepsilon) \in S_{h_\varepsilon}$, $X_\varepsilon = (x_\varepsilon,f_\varepsilon(x_\varepsilon,t_\varepsilon))$. Moreover, the $(X_\varepsilon,t_\varepsilon)$ and $h_\varepsilon$ can be chosen so that
  \begin{align}\label{eqnWeak:stability viscosity solution contact points}
     \lim \limits_{\varepsilon\to 0}(X_\varepsilon,t_\varepsilon) = (X_0,t_0),\;\; \lim \limits_{\varepsilon\to 0} h_\varepsilon = 0.
  \end{align}
  Then, as noted above $U_{f_\varepsilon}$ is a viscosity subsolution, and it follows that 
  \begin{align*}
    \partial_t V_{S_{h_\varepsilon}}(X_\varepsilon,t_\varepsilon) \leq G(\partial^+_n U_{S_{h_\varepsilon}}(X_\varepsilon,t_\varepsilon)).
  \end{align*}
  On the other hand, by Lemma \ref{lemWeak:test interface translation} (and Remark \ref{remWeak:test interface translation two phase} for the two-phase case),
  \begin{align*}
    |\partial^+_n U_{S}(X_\varepsilon,t_\varepsilon)-\partial^+_n U_{S_{h_\varepsilon}}(\tilde X_\varepsilon,t_\varepsilon)| \leq Ch_\varepsilon,\;\;\tilde X_\varepsilon := X_\varepsilon - h_\varepsilon e_{d+1},
  \end{align*}
  and it follows that
  \begin{align*}
    \partial_t V_{S}(\tilde X_\varepsilon,t_\varepsilon) \leq G(\partial^+_n U_{S}(\tilde X_\varepsilon,t_\varepsilon))+C|h_\varepsilon|.
  \end{align*}
  Now, given that $(\tilde X_\varepsilon,t_\varepsilon) \to (X_0,t_0)$ as $\varepsilon \to 0$, we conclude that
  \begin{align*}
    \partial_t V_{S}(X_0,t_0) \leq G(\partial^+_n U_{S}(X_0,t_0)).
  \end{align*}
  Here we note that the reason this works is that all of the calculations occur at the level of $U_S$ and $U_{S_{h_\ep}}$, which are both smooth; hence the limiting operations are straightforward.
  This proves $U_f$ is a subsolution. The argument for a supersolution is entirely analogous and we omit the details.
  
\end{proof}

The next proposition is a basic observation about how test functions for the lower dimensional, non-local parabolic problem yield to test functions of the free boundary problem. 

\begin{proposition}\label{prop:test functions for h yield test functions for U}
  Let $\phi$ be an admissible test function touching $f$ from above at some $t_0 \in [a,b]$ and $x_0 \in \Sigma_0$. Then, $U_{\phi}$ touches $U_{f}$ from above at $(X_0,t_0)$ where $X_0 = (x_0,f(x_0,t_0))$.
\end{proposition}

\begin{proof}
  First of all, since $f\leq \phi$ everywhere, it follows that
  \begin{align*} 
    D_f\subset D_\phi.
  \end{align*}
  Then, from the definition of $U_\phi,U_f$ it follows that $U_\phi \geq U_f$ on $\partial D_f$ and $U_\phi \geq U_f$ on $\partial D^-_f$. Applying the comparison principle to each phase yields
  \begin{align*}
    & U_f\leq U_\phi \textnormal{ in } D_f \textnormal{ and } U_f\leq U_\phi \textnormal{ in } D^-_f.
  \end{align*}	  
  This means that $U_f\leq U_\phi$ for all times $t\in [a,t_0]$. Since $f=\phi$ at $(x_0,t_0)$, we have $U_f= U_\phi =0$ at time $t_0$ at the point $(x_0,f(x_0)) = (x_0,\phi(x_0))$, proving the proposition.
\end{proof}

\begin{lemma}\label{lemWeak:FBViscSolImpliesParabolicViscSol}
  Let $U$ be a viscosity subsolution (supersolution) of (\ref{eqIn:MainFBEvolution}) in $\Sigma_0 \times [a,b]$ whose free boundary is given as the graph of some upper semi-continuous $f$ (lower semi-continuous), then $f$ is a viscosity subsolution (supersolution) of (\ref{eqWeak:ParabolicEquation}) in $\mathbb{R}^d \times [a,b]$.
\end{lemma}

\begin{proof}
  Let $\phi$ be a test function which touches $f$ from above at $(x_0,t_0)$. According to Proposition \ref{prop:test functions for h yield test functions for U}, $U_\phi$ is a test function which touches $U$ from above at $(X_0,t_0)$, where $X_0= (x_0,f(x_0,t_0))$. Since $U$ is a viscosity subsolution, it follows that 
  \begin{align*}
    V_{U_\phi} \leq G(\partial^+_n U_\phi,\partial^-_n U_\phi),\;\textnormal{ at } (X_0,t_0)
  \end{align*}
  This implies that at $(X_0,t_0)$ we have
  \begin{align*}
     \frac{\partial_t \phi}{\sqrt{1+|\nabla \phi|^2}} \leq G(I^+(\phi), I^-(\phi)).
  \end{align*}
  Since $\phi$ was arbitrary, it follows $f$ is a viscosity subsolution. The proof for supersolutions is entirely analogous and we omit it.
\end{proof}


\section{Propagation of the modulus of continuity}\label{sec:PropagationOfModulus}

After all of the build-up of the preceding sections, we can now demonstrate that a modulus of continuity of initial data will be preserved by the fractional parabolic equation.  This result is very simple once the comparison theorem for viscosity solutions has been established.

\begin{lemma}\label{lemPOM:PropagateModulus}
  Let $J$ be as in assumption \ref{assumptionMM}, $T>0$ and $f:\mathbb{R}^d\times [0,T]\to\mathbb{R}$ a continuous viscosity solution of 
  \begin{align*}
    \partial_t f & = J(f) \textnormal{ in } \mathbb{R}^d\times [0,T],\\
    f(x,0) & = f_0(x) \textnormal{ in } \mathbb{R}^d.
  \end{align*}	  
  If $f_0$ is continuous with modulus of continuity $\om(\cdot)$, then the same will be true of $f(\cdot,t)$ for all $t\in[0,T]$. In particular, we have the estimate
  \begin{align*}
    \abs{f(x,t)-f(y,t)}\leq \om(\abs{x-y}).	  
  \end{align*}	  
\end{lemma}

\begin{proof}
  Let $h\in\mathbb{R}^d$ be fixed, and consider the function,
  \begin{align*}
    w(x,t) := (\tau_hf)(x,t)-f(x,t).
  \end{align*}
  By assumption,
  \begin{align*}
    \forall\ x\in\real^d,\ \ w(x,0) \leq \om(|h|).
  \end{align*}
  By translation invariance of $J$ and Lemma \ref{lemCP:EquationForTheDifference}, $w$ satisfies, in the viscosity sense
  \begin{align*}
    \partial_t w & \leq M^+_J(w).
  \end{align*}
  In light of the above, Lemma \ref{lemCP:Comparison} implies that for any $(x,t)\in\mathbb{R}^d\times [0,T]$,
  \begin{align*}
    w(x,t) \leq \om(|h|).
  \end{align*}
  Since this holds for an arbitrary vector $h$, the lemma is proved.
  
\end{proof}


\section{The proofs of Theorems \ref{thm:MainMetaVersion} and \ref{thm:MetaILipAndMinMax}}\label{sec:Proofs}

In this section, we collect all of the various facts of the previous sections that combine to prove Theorems \ref{thm:MainMetaVersion} and \ref{thm:MetaILipAndMinMax}.

\subsection{Proof / explanation of Theorem \ref{thm:MainMetaVersion}}
To be precise about the equation for the one-phase problem, we mean that for all $t$, $\partial\{U_f>0\}=\graph(f(\cdot,t))$, that $U_f$ solves the problem (\ref{eqIn:MainFBEvolution}), in the free boundary viscosity sense of Section \ref{sec:WeakSolutions}, and that the normal velocity is given by (\ref{eqNO:DefOfI})
\begin{align*}
	V(x,f(x))=G(\partial_n U_f(x,f(x)))=G(I(f,x)).
\end{align*}
Furthermore, the only restriction on $f$ is that $f\geq \del$ and that $f$ is bounded and continuous on $\real^d\times[0,T]$.

To be precise about the equation for the two-phase problem, we take the analogous set-up, but the normal velocity is prescribed as (\ref{eq2Ph:TwoPhaseOperatorDefinition}),
\begin{align*}
	V(x,f(x))=G(\partial^+_n U_f(x,f(x)), \partial^-_n U_f(x,f(x)))=H(f,x).
\end{align*}  
The notion of solution is again the free boundary viscosity solution of Section \ref{sec:WeakSolutions}.  We have a further restriction that $\del\leq f\leq L-\del$ and $f$ is continuous on $\real^d\times[0,T]$.

We note that per the assumptions of Section \ref{sec:NewOperator}, for the one-phase problem, and Section \ref{sec:TwoPhase} for the two-phase problem, we always assume the free boundary is a \emph{bounded} graph for all times.  Section \ref{sec:WeakSolutions} (Lemmas \ref{lemWeak:ParabolicViscSolImpliesFBViscSol} and \ref{lemWeak:FBViscSolImpliesParabolicViscSol}) contains the proof of (i), that under the graph assumption, the free boundary viscosity solution is equivalent to the parabolic viscosity solution of 
\begin{align*}
	\begin{cases}
		\partial_t f=J(f)\ &\text{in}\ \real^d\times(0,T)\\
		f(\cdot,0)=f_0\ &\text{on}\ \real^d.
	\end{cases}
\end{align*}
Thanks to Theorem \ref{thmMM:IandHSatisfyAssumptions}, both 
\begin{align*}
J(f)=G(I(f))\cdot\sqrt{1+\abs{\grad f}^2}\ \  \text{and}\ \  J(f)=H(f)\cdot\sqrt{1+\abs{\grad f}^2}
\end{align*} 
 are admissible operators for solving the parabolic equation, for respectively the one and two phase problems (for $I$, $H$, defined respectively in (\ref{eqNO:DefOfI}), (\ref{eq2Ph:TwoPhaseOperatorDefinition})).

The proof of the propagation of modulus claimed in Theorem \ref{thm:MainMetaVersion} (ii) is a consequence of the correspondence of solutions in (i), combined with the result of Lemma \ref{lemPOM:PropagateModulus}, which is applicable thanks to Theorem \ref{thmMM:IandHSatisfyAssumptions}.  This was, of course, only possible because of the comparison theorem deduced in Section \ref{sec:Comparison}.

The proof of the existence of viscosity solutions claimed in Theorem \ref{thm:MainMetaVersion} (iii), is a consequence of the correspondence of solutions in (i), with the proof of existence in Section \ref{sec:Existence}.

The precise statement of Theorem \ref{thm:MetaILipAndMinMax} has the operator $I$ defined with equation (\ref{eqNO:BulkNonlinearDefI}), through the formula (\ref{eqNO:DefOfI}).  The proof of Lipschitz nature of $I$ in Theorem \ref{thm:MetaILipAndMinMax}, is a consequence of Theorem \ref{thmNO:LipschitzProperty}.  The min-max formula is a result of Theorem \ref{thmMM:TranslationInvariantMinMax}, combined with verification the $I$ satisfies the results of Assumption \ref{assumptionMM}, which is in Theorem \ref{thmMM:IandHSatisfyAssumptions}.


\section{Comments about our assumptions, and some questions}\label{sec:CommentsOnAssumptions}

There are many assumptions that accompany our Theorems \ref{thm:MainMetaVersion} and \ref{thm:MetaILipAndMinMax}.  Some are for simplicity and some are more nuanced.  Here we discuss how some could be removed and how some may be difficult to overcome.  At the end, we suggest some lines of open questions that could be useful.

\subsection{Translation invariance}  In all of the above situations, we have only treated free boundary evolutions in which the equations in the positive and negative phases of $U$, as well as the free boundary condition, $G$, are translation invariant.  We would like to say that we did this out of reasons of simplicity, but we cannot claim that motivation, as it is not clear how difficult it will be to remove this assumption.  The main place it was used in a tangible fashion was in the proof of Lemma \ref{lemNO:IDependenceOnConstants}, and this lemma played a fundamental role in most of the major results that followed-- in particular Theorem \ref{thmNO:LipschitzProperty} and all subsequent results invoking it.  Although the translation invariance led to a simpler min-max representation in Section \ref{sec:MinMax}, it does not play an essential role there.  It does play, however, a substantial role in the arguments of Section \ref{sec:Comparison}.  This is because (surprisingly) it is still unknown to what extent the uniqueness theory for viscosity solutions of integro-differential equations is applicable.  Presumably, one expects the uniqueness theory to parallel that of second order equations, but to date, there are only a handful of works that can handle equations with non-trivial dependence on $x$; we defer to the work of Barles-Imbert \cite{BaIm-07} for the latest results and earlier references.  In particular, if one pursues the $x$-dependent min-max formulation in Section \ref{sec:MinMax}, it is not clear if the resulting integro-differential operators will fall into the scope of \cite{BaIm-07}.

\subsection{Rotational invariance}  In principle, this should only be a technical assumption that is an artifact of our method to establish the Lipschitz continuity of the operator, $I$, culminating in Theorem \ref{thmNO:LipschitzProperty}.  The interested reader can probe more deeply into our method of proof for Theorem \ref{thmNO:LipschitzProperty}, where we rotated two domains so that they would intersect in a $C^1$ fashion.  The way we chose to deal with this situation required to rotate the underlying equations, hence needing an invariance assumption in the absence of better estimates.  It would be helpful to remove this assumption.

\subsection{Graph assumption}  As the careful reader will see, we \emph{assumed} in Theorem \ref{thm:MainMetaVersion} that if the free boundary evolution, $\partial\{U(\cdot,t)>0\}$ is a graph for all times, then one can deduce a correspondence between the free boundary viscosity solution and the parabolic viscosity solution.    This raises the question of whether or not the graph assumption is reasonable?  Answering this question is essential to the future use of the methods suggested in this paper. This ties in with the star-shaped assumption that appears for some of the preliminary results in \cite{ChoiJerisonKim-2007RegHSLipInitialAJM}.  In \cite{ChoiJerisonKim-2007RegHSLipInitialAJM}, they were able to first focus on star-shaped initial free boundaries (which are equivalent to being a graph over some sphere), and then they subsequently extended the regularity results to more general boundaries.  It is conceivable that something similar could be done in our context, but it is not obvious how to carry out this procedure at the moment.


\subsection{Some Questions}

We believe there are at least a few natural questions resulting from the results presented above, and we mention a few of them.  This list is by no means complete, but we think these are some of the most interesting questions.

\textbf{$x$-dependence.}  As mentioned above, all ingredients ($F_1$, $F_2$, $G$) were taken to be independent of $x$.  What can be said when $x$ dependence is allowed?  This may lead to a better understanding of where one may hope to push the uniqueness and regularity theory for viscosity solutions of integro-differential operators like the ones that appear in (\ref{eqMM:LinearInMinMaxOldVersion}).

\textbf{Anisotropic $G$.} What happens if $G$ is allowed to depend, in an appropriately monotone way, on the full gradients, $\grad U^+_f$ and $\grad U^-_f$, instead of simply $\partial_n^+U_f$ and $\partial_n^-U_f$?  Any reasonable ellipticity restriction on $G$ should keep it within the realm of the GCP and would likely enjoy properties similar to those listed in Assumption \ref{assumptionMM}.  Is this indeed the case?

\textbf{L\'evy measures.}  Is it possible to find more detailed information about the L\'evy measures, $\mu^{ij}$, that appear in Theorem \ref{thm:MetaILipAndMinMax}?  Can it be proved that they are in a class for which some regularity theory is known (e.g. \cite{CaSi-09RegularityIntegroDiff}, \cite{Chan-2012NonlocalDriftArxiv}, \cite{ChDa-2012NonsymKernels}, or \cite{SchwabSilvestre-2014RegularityIntDiffVeryIrregKernelsAPDE})?  Are $\mu^{ij}$ in a different class that has not yet been studied, but possibly still enjoy certain regularity properties for solutions of $\partial_t f=I(f)$?  Answering these questions will be an essential building block in order that the nonlocal theory can shed new light on these free boundary problems.

\textbf{Graphs over a manifold.}  In this paper, we made the assumption that the free boundary was a graph over $\real^d$ that evolves with time.  Can the results here be extended to include the case where the free boundary, \emph{for short times}, is the evolution of a graph over some submanifold of $\real^{d+1}$ (in this case, the natural submanifold would be that of the initial conditions, $\partial\{U(\cdot,0)>0\}$)?  This situation is not incompatible with the min-max theory developed in Section \ref{sec:MinMax}.  Indeed, the results of \cite{GuSc-2016MinMaxNonlocalarXiv} were proved for operators acting on functions in, e.g. $C^{1,\gam}(\partial\Om)$, whenever $\partial\Om$ is a nice enough hyper-surface.  Furthermore, this idea of tracking the free boundary as a graph over a reference manifold is compatible with some of the results in \cite{ChoiJerisonKim-2007RegHSLipInitialAJM}, where the preservation of the star-shape property is equivalent to being a graph over a sphere in the radial direction.  Again, even if one can justify the graph assumption, this would likely result in nonlocal parabolic equations that have non-trivial $x$-dependence (in fact, the notion of translation invariance makes no sense in this situation).  We note that in the case of rotationally invariant operators in the respective positive and negative phases, this would result in a rotationally invariant integro-differential operator for functions on the sphere-- this should be amenable to similar ideas we used for translation invariant equations for functions on $\real^d$.  Studying the operator on a manifold that does not necessarily admit an invariant group of transformations puts the resulting parabolic equation for the free boundary into a class of equations for which existence, uniqueness, regularity, etc... is not well established (i.e. viscosity solutions for nonlocal parabolic equations with $x$-dependent coefficients, posed over a manifold).

\textbf{Equations in divergence form.}  One should note that we did not treat the case when in the set $\{U>0\}$, we have the equation 
\begin{align*}
	F(D^2U,\grad U)=\div(A(x)\grad U)=0.
\end{align*}
In most reasonable circumstances, this equation enjoys the maximum principle, and hence also the free boundary operator will enjoy the GCP.  Note that this equation is not translation invariant (otherwise, it is contained in our results).  However, the methods and results for the divergence theory are a bit different from those employed herein for the fully nonlinear theory, and so there may be different outcomes.  Thus, the question is: how much can be said about the operator, $I$, under this assumption?  It may be possible that within the scope of Section \ref{sec:MinMax}, the L\'evy measures could be shown to have nice properties with respect to Lebesgue measure on $\real^d$, akin to some of the existing integro-differential results.  The counter-balance to this benefit would be the fact that the integro-differential equation would no longer be translation invariant.  However, the known results about Green's functions and Poisson kernels for divergence equations may shed extra light on the situation.  It is conceivable that under some extra regularity assumptions on $A(x)$, one may obtain a class of integro-differential operators that enjoys known Schauder-type results, or could be within reach of minor modifications of such existing results.

\textbf{Problems of Stefan type.}  If one replaces equation (\ref{eqIn:MainFBEvolution}) with a parabolic problem, then the resulting free boundary problem is known as the Stefan problem.  It could be possible that the ideas in Sections \ref{sec:NewOperator} and \ref{sec:MinMax} might extend to the Stefan set-up.  This would require a new result, similar to that of Section \ref{sec:MinMax} that applied to operators with the \emph{space-time} GCP instead of just the GCP in space alone.  One might conjecture that the resulting linear operators in the min-max would be nonlocal in both space and time, not just nonlocal in space, as was proved here.  That would mean that the space-time nonlocal theory could be relevant, such as, e.g. Allen \cite{Allen-2016HolderRegularityNondivergenceNonlocalParabolic-ArXiv} (which treats a more general form of the Caputo-Marchaud time derivative), or even a nonlinear (in time) extension of similar results.


\bibliography{refs}
\bibliographystyle{plain}
\end{document}